\documentclass[11pt]{amsart}
\usepackage[dvipsnames]{xcolor}
\usepackage{a4wide, amsmath, amsthm, amsfonts, amssymb, caption, graphicx, hyperref, mathrsfs, subcaption, textcomp, tikz, tikz-cd, ytableau, mathscinet, mathtools}
\usepackage{mathdots}
\usepackage[shortlabels]{enumitem}
\usepackage[T1]{fontenc}
\usepackage[nameinlink,noabbrev,capitalise]{cleveref}
\usepackage{subfiles}

\newtheorem{theorem}{Theorem}[section]
\newtheorem{proposition}[theorem]{Proposition}

\newtheorem{corollary}[theorem]{Corollary}
\newtheorem{lemma}[theorem]{Lemma}
\newtheorem{question}[theorem]{Question}
\newtheorem{definition}[theorem]{Definition}

\theoremstyle{definition}

\newtheorem{remark}[theorem]{Remark}
\newtheorem{example}[theorem]{Example}
\newtheorem{notation}[theorem]{Notation}

\numberwithin{equation}{section}

\hypersetup{
  colorlinks,
  citecolor=Purple,
  linkcolor=Blue,
  urlcolor=Green}

\newcommand{\Z}{\mathbb{Z}}
\newcommand{\D}{\mathbb{D}}
 
\newcommand{\R}{\mathbb{R}}
\newcommand{\C}{\mathbb{C}}
\newcommand{\V}{\mathbb{V}}
 
\newcommand{\A}{\mathcal{A}}

\DeclareMathOperator{\Br}{Br}
\DeclareMathOperator{\Aug}{Aug}
\DeclareMathOperator{\im}{im}

\newcommand{\La}{\Lambda}

\newcommand{\w}{\mathfrak{w}}
\newcommand{\frakr}{\mathfrak{R}}
\newcommand{\hooklongrightarrow}{\lhook\joinrel\longrightarrow}

\DeclareMathOperator{\Ob}{Ob}
\DeclareMathOperator{\id}{id}
\DeclareMathOperator{\MCS}{MCS}
\DeclareMathOperator{\GL}{GL}
\DeclareMathOperator{\diag}{diag}
\DeclareRobustCommand{\ateb}{\text{\reflectbox{$\beta$}}}
\DeclareRobustCommand{\atleD}{\text{\reflectbox{$\Delta$}}}
\DeclareRobustCommand{\ammag}{\text{\reflectbox{$\gamma$}}}
\DeclareRobustCommand{\ohr}{\text{\reflectbox{$\rho$}}}
\DeclareRobustCommand{\atebud}{\text{\reflectbox{\rotatebox[origin=c]{180}{$\ateb$}}}}
\DeclareRobustCommand{\ammagud}{\text{\reflectbox{\rotatebox[origin=c]{180}{$\ammag$}}}}
\DeclareRobustCommand{\Deltaud}{\text{\reflectbox{\rotatebox[origin=c]{180}{$\Delta$}}}}
\DeclareRobustCommand{\Lambdaud}{\text{\reflectbox{\rotatebox[origin=c]{180}{$\Lambda$}}}}
\DeclareRobustCommand{\ohrud}{\text{\reflectbox{\rotatebox[origin=c]{180}{$\ohr$}}}}

\begin{document}

\title{Decompositions of augmentation varieties via weaves and rulings}

\author[J. Asplund]{Johan Asplund}
\address{Department of Mathematics, Stony Brook University, Stony Brook, NY 11794}
\email{johan.asplund@stonybrook.edu}
\author[O. Capovilla-Searle]{Orsola Capovilla-Searle}
\address{Department of Mathematics, Oregon State University, Corvallis, OR 97331-4605} \email{capovilo@oregonstate.edu}
\author[J. Hughes]{James Hughes}
\address{Department of Mathematics, Duke University, Durham, NC 27708}
\email{jhughes@math.duke.edu}
\author[C. Leverson]{Caitlin Leverson}
\address{Mathematics Program, Bard College, Annandale-on-Hudson, NY 12504} \email{cleverson@bard.edu}
\author[W. Li]{Wenyuan Li}
\address{Department of Mathematics, University of Southern California, Los Angeles, CA 90089} \email{wenyuan.li@usc.edu}
\author[A. Wu]{Angela Wu}
\address{Department of Mathematics and Statistics, Bucknell University, Lewisburg, PA 17837} \email{a.wu@bucknell.edu}

\begin{abstract}
The braid variety of a positive braid and the augmentation variety of a Legendrian link both admit decompositions coming from weaves and rulings, respectively. We prove that these decompositions agree under an isomorphism between the braid variety and the augmentation variety. We also prove that both decompositions coincide with a Deodhar decomposition and another decomposition coming from the microlocal theory of sheaves. Our proof relies on a detailed comparison between weaves and Morse complex sequences. Among other things, we show that the cluster variables of the maximal cluster torus of the augmentation variety can be computed from the Legendrian via Morse complex sequences.
\end{abstract}

\maketitle
\thispagestyle{empty}

\tableofcontents

\section{Introduction}

\subsection{Context} Legendrian submanifolds and their Lagrangian fillings are important objects in contact and symplectic topology. Symplectic field theory provides invaluable tools for furthering classification efforts, building on Gromov's theory of pseudoholomorphic curves \cite{Gromov,EGH}. One such tool is the Chekanov--Eliashberg differential graded algebra (dga) $\A(\La)$ of a Legendrian submanifold $\La$ \cite{ChekanovDGA}. In particular, by studying augmentations of the Chekanov--Eliashberg dga, i.e., dga maps $\epsilon\colon \A(\La)\to \C$, one can distinguish Legendrian submanifolds and Lagrangian fillings of such \cite{ChekanovDGA, EkholmHondaKalman, YuPan}, or even obstruct the existence of such fillings \cite{DimitroglouRizell_2016,GaoRutherford}.

Augmentations are intimately tied to normal rulings, a combinatorial tool independently developed by Fuchs \cite{Fuchs03} and Chekanov--Pushkar{\cprime} \cite{ChekanovRuling,ChekanovPushkar}, that in turn is related to Morse theory and generating families. Studying the space of augmentations, known as the augmentation variety $\Aug(\La)$, allows for a more detailed relationship between these objects that comprises an essential ingredient in our story. The study of generating families through Morse theory leads to the definition of a Morse complex sequence (MCS) \cite{Henry}. Building on this work, Henry and Rutherford showed that the collection of all normal rulings of $\La$ yields a decomposition of the augmentation variety $\Aug(\La)$ into pieces of the form $(\C^\ast)^t \times \C^c$ constructed via MCSs \cite{HenryRutherford15}. 

Recently, additional tools have emerged to study Legendrian links $\La(\beta)$ associated to positive braids $\beta$ and their exact Lagrangian fillings. One line of inquiry, pursued in the series of papers \cite{CGGS1, CGGS2, CGGLSS}, studies braid varieties $X(\beta)$, a class of algebraic varieties coming from representation theory and cluster algebras that are isomorphic to augmentation varieties $\Aug(\Lambda(\beta))$. Braid varieties are studied via weaves, which are certain graphs encoding exact Lagrangian cobordisms \cite{CasalsZaslow20}, and are used in constructing and identifying cluster structures associated to braid varieties \cite{CGGLSS}. These cluster structures are one of the most effective approaches currently known for systematically studying and distinguishing exact Lagrangian fillings of Legendrian links \cite{CasalsGao22, GaoShenWeng, ABL22, Hughes23, CSHW23, CasalsGao24}. In addition, weaves can be used to define multiple decompositions of the braid variety. 

The decompositions for both augmentation varieties and braid varieties provide important information such as the point count over finite fields, which is also related to the homology and mixed Hodge structure of the varieties \cite{Mellit,Su,LamSpeyerII} and the geometric $P=W$ conjecture \cite{SuDual}. The information of point counts enables recovery of particular terms of the HOMFLY-PT polynomial of the smooth knot type of a Legendrian knot \cite{HenryRutherford15,LeversonRutherford,GalashinLam}. (Relatedly, the homology and mixed Hodge structure is also used to recover the HOMFLY-PT homology of the smooth knot type \cite{STZ_ConstrSheaves,Trinh}.)

\subsection{Main results}\label{sec:main_result_intro}

Let $\beta$ be a positive braid on $n$ strands and let $\Delta$ denote the half twist that is the positive braid lift of the permutation $w_0$, the longest word in $S_n$ in the Bruhat order. Suppose that the Demazure product $\delta(\beta)$ of $\beta$ is equal to $w_0$ (see \cref{def:demazure_prod}). One can associate the following algebraic varieties to $\beta$:
\begin{itemize}
    \item The augmentation variety $\Aug(\La(\beta\Delta))$ of the Legendrian $(-1)$-closure of $\beta\Delta$ (equipped with one marked point per strand of the braid); see \cref{sec:augmentation_variety}.
    \item The type-A braid variety $X(\beta)$ where $G = \GL(n,\C)$; see \cref{sec:braid_variety}.
    \item The type-A braid-Richardson variety $R^\circ_{\pi,\beta}$ where $G = \GL(n,\C)$; see \cref{sec:deodhar_decomp}.
    \item The framed moduli space of microlocal rank $1$ sheaves $\mathcal M_1^\textit{fr}(\La(\beta\Delta))$ (with one marked point per strand of the braid); see \cref{sec:braid_sheaf}.
\end{itemize}
Some of these algebraic varieties are known to be isomorphic. Namely, \cite{CGGSBS} proved $X(\beta) \cong R^\circ_{w_0,\beta}$, \cite{CasalsLi,CasalsWeng} proved $X(\beta) \cong \mathcal M_1^\mathit{fr}(\La(\beta\Delta))$, and, in the special case $\beta = \Delta\gamma$, \cite{CGGS1} proved $X(\beta) \cong \Aug(\La(\beta\Delta))$. Furthermore, these algebraic varieties are known to be equal to disjoint unions of pieces of the form $(\C^\ast)^t \times \C^c$:
\begin{description}
    \item[Ruling decomposition] Henry--Rutherford \cite{HenryRutherford15} proved that
    \[
    \Aug(\La(\beta\Delta)) = \bigsqcup_\rho\, (\C^\ast)^{s(\rho)} \times \C^{r(\rho)-\binom n2},
    \]
    where the disjoint union is taken over all normal rulings of a fixed front diagram of $\La(\beta\Delta)$ and the quantities $s(\rho)$ and $r(\rho)$ denote the number of switches and $0$-graded returns of $\rho$, respectively.
    \item[Weave decomposition] Casals--Gorsky--Gorsky--Simental \cite{CGGS1} showed that the braid variety of $\beta$ admits a decomposition via weaves
    \[
    X(\beta) = \bigsqcup_{\w}\, (\C^\ast)^{t(\w)} \times \C^{c(\w)},
    \]
    where the disjoint union is taken over certain finite sets of simplifying weaves and the quantities $t(\w)$ and $c(\w)$ denote the number of trivalent vertices and cups of the weave, respectively.
    \item[Deodhar decomposition] The Deodhar decomposition was originally defined for open Richardson varieties and later extended to braid varieties in \cite{GalashinLamTrinhWilliams} as a crucial component of the construction of a cluster structure on the braid variety \cite{GLSS}. This decomposition is defined in terms of distinguished sequences of permutations and has pieces of the form $(\C^\ast)^t \times \C^c$; see \cref{sec:deodhar_decomp} for details.
    \item[Sheaf decomposition] Shende--Treumann--Zaslow \cite{STZ_ConstrSheaves} showed that the moduli stack of microlocal rank $1$ sheaves for a Legendrian link $\Lambda(\beta\Delta)$ in $\R^3$ admits a decomposition as a disjoint union of pieces of the form $(\C^\ast)^t \times \C^c$, also defined in terms of normal rulings of a fixed front of $\Lambda(\beta\Delta)$. Such decompositions can also be defined on the framed moduli space $\mathcal M_1^\mathit{fr}(\La(\beta\Delta))$; see \cref{sec:rulings_weaves_sheaves}.
\end{description}
Our main result extends the isomorphism $X(\beta) \cong \Aug(\La(\beta\Delta))$ to all braids $\beta$ with Demazure product equal to $w_0$, and proves that all of the above decompositions coincide under these isomorphisms.
\begin{theorem}\label{thm:intro_main}
    Let $\beta \in \Br_n^+$ such that $\delta(\beta) = w_0$. There are isomorphisms
    \[
    \begin{tikzcd}[row sep=scriptsize,column sep=2cm]
        \Aug(\La(\beta\Delta)) & X(\beta) \lar{\cong}[swap]{\text{\cref{thm:-1=braid vty_new}}} \rar{\text{\cref{prop:sheaves_braid_var}}}[swap]{\cong} \dar{\text{\cref{prop:braid_richardson}}}[swap]{\cong} & \mathcal M_1^\textit{fr}(\La(\beta\Delta)) \\
        & R^\circ_{w_0,\beta} &
    \end{tikzcd}
    \]
    under which all the following decompositions coincide:
    \begin{itemize}
        \item The ruling decomposition of $\Aug(\La(\beta\Delta))$.
        \item The weave decomposition of $X(\beta)$ given by right simplifying weaves corresponding to the normal rulings of $\La(\beta\Delta)$.
        \item The Deodhar decomposition of $R^\circ_{w_0,\beta}$ given by (right inductive) distinguished sequences.
        \item The sheaf decomposition of $\mathcal M_1^\textit{fr}(\La(\beta\Delta))$.
    \end{itemize}
\end{theorem}
More precisely, our main result (\cref{thm:intro_main}) consists of the following comparisons of decompositions.

\begin{theorem}[{\cref{thm:hr_and_weave_decomps}}]\label{thm:intro_hr_weave}
Let $\beta \in \Br_n^+$ such that $\delta(\beta) = w_0$. For every normal ruling $\rho$ of $\La(\beta\Delta)$, there exists a right simplifying weave $\mathfrak{w}_{\rho}$ such that the ruling decomposition of $\Aug(\La(\beta\Delta))$ coincides with the weave decomposition of $X(\beta)$ given by this collection of right simplifying weaves $\mathfrak{w}_{\rho}$ under the isomorphism from \cref{thm:-1=braid vty_new}.
\end{theorem}
\begin{theorem}[{\cref{thm:deodhar=weave}}]\label{thm:intro_Deodhar=Weave}
Let $\beta \in \Br_n^+$ such that $\delta(\beta) = w_0$. For every distinguished sequence $\mathfrak{v}$ of $w_0$ there exists a right simplifying weave $\mathfrak{w}_{\mathfrak{v}}$ such that the Deodhar decomposition of $R^\circ_{w_0,\beta}$ agrees with the weave decomposition of $X(\beta)$ given by this collection of right simplifying weaves $\mathfrak{w}_{\mathfrak{v}}$ under the isomorphism $R^\circ_{w_0,\beta} \cong X(\beta)$ from \cref{prop:braid_richardson}.
\end{theorem}
\begin{theorem}[{\cref{thm:deodhar=hr_decomp}}]
Let $\beta \in \Br_n^+$ such that $\delta(\beta) = w_0$. For every normal ruling $\rho$ of $\La(\beta\Delta)$ there exists a distinguished sequence $\mathfrak{v}_{\rho}$ of $w_0$ such that the ruling decomposition of $\Aug(\La(\beta\Delta))$ coincides with the Deodhar decomposition of $R^\circ_{w_0,\beta}$ determined by the distinguished sequences $\mathfrak{v}_{\rho}$ under the isomorphisms from \cref{thm:-1=braid vty_new,prop:braid_richardson}.
\end{theorem}
\begin{theorem}[{\cref{thm:ruling_sheaf=ruling_aug}}]\label{thm:hr_stz}
    Let $\beta \in \Br_n^+$ such that $\delta(\beta) = w_0$. The ruling decomposition of $\Aug(\La(\beta\Delta))$ coincides with the sheaf decomposition of $\mathcal M_1^\textit{fr}(\La(\beta\Delta))$ under the isomorphism from \cref{thm:-1=braid vty_new,prop:sheaves_braid_var}.
\end{theorem}

\subsection{Outline of the proof of~\cref{thm:intro_hr_weave}.}

The construction of the weave decomposition of the braid variety $X(\beta)$ uses moduli spaces of so-called \emph{trivial monodromy algebraic weaves} (see \cref{sec:weave_decomp} for details). The starting point of the proof of \cref{thm:intro_hr_weave} is to provide a contact-geometric analog of trivial monodromy weaves using Morse complex sequences.

Roughly, given a trivial monodromy weave $\w \colon \beta \to \beta'$, the contact-geometric analog consists of a sequence of braids $(\beta=\beta_1,\ldots,\beta_q=\beta')$ and an associated sequence of Morse complex sequences related by so-called MCS braid moves with trivial monodromy. See \cref{fig:intro_weave_mcs} for a visual summary of this correspondence, and \cref{sec:decomp_braid_varieties_mcs} for a detailed discussion.
\begin{figure}[!htb]
    \centering
    \includegraphics{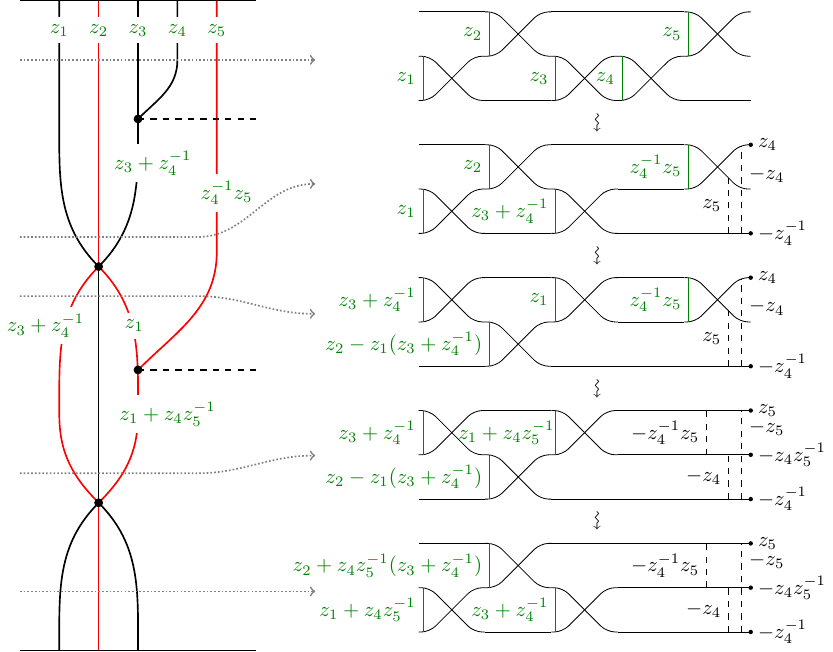}
    \caption{The correspondence between trivial monodromy weaves and Morse complex sequences. The dashed handleslide marks and marked points (on the right) are related to the horizontal dashed segments in the weave (on the left).}
    \label{fig:intro_weave_mcs}
\end{figure}

The idea to prove \cref{thm:intro_hr_weave} is to associate to each normal ruling $\rho$ of $\La(\beta\Delta)$ a right simplifying weave $\w_\rho$. This gives a bijection between the set of pieces of the two decompositions
\[
\phi_\rho \colon (\C^\ast)^{s(\rho)} \times \C^{r(\rho)-\binom n2} \hooklongrightarrow \Aug(\La(\beta\Delta)) \quad \text{and} \quad \phi_{\w_\rho} \colon (\C^\ast)^{t(\w_\rho)} \times \C^{c(\w_\rho)} \hooklongrightarrow X(\beta).
\]
Comparing these pieces via the above correspondence between trivial monodromy weaves and sequences of Morse complex sequences, we show that their images agree under the isomorphism $\alpha \colon X(\beta) \overset{\cong}{\to} \Aug(\La(\beta\Delta))$.

\begin{remark}
    The collection of right simplifying weaves that we construct is not unique. However, any such collection corresponding to the normal rulings obtained by the map $\rho \mapsto \w_\rho$ induces the same weave decomposition. We expect that the collection of right simplifying weaves are unique up to weave equivalence (see \cref{rmk:non_unique_weave}).
\end{remark}

To prove the main result, we first describe the correspondence between trivial monodromy weaves and Morse complex sequences on a categorical level. For any positive integer $n$, we consider the weave category $\mathfrak{W}_n$ whose objects are the set $\Br_n^+$ of positive braids on $n$ strands, and whose morphisms are combinatorial diagrams called algebraic weaves (see \cref{sec:weave_decomp} for details). Letting $\mathfrak{C}$ denote the category of algebraic correspondences whose objects are algebraic varieties, there is a functor
\[
\mathfrak{X} \colon \mathfrak{W}_n \longrightarrow \mathfrak{C}
\]
where $\mathfrak{X}(\beta) \coloneqq X(\beta)$ is the braid variety. For any weave $\w \colon \beta \to \beta'$, there is an algebraic correspondence
\[
X(\beta) \longleftarrow \mathfrak{X}(\w) \longrightarrow X(\beta'),
\]
where $\mathfrak{X}(\w)$ is the moduli space of trivial monodromy algebraic weaves on $\w$. This algebraic correspondence has the property that for any weave $\w \colon \beta \to \Delta$ with $c$ cups and $t$ trivalent vertices, there is an injective map $\mathfrak{X}(\w) \colon \C^c \times (\C^\ast)^t \hookrightarrow X(\beta)$. We then obtain decompositions of $X(\beta)$ by considering certain tuples of weaves $(\w_1,\ldots,\w_k)$ and their associated maps $\mathfrak{X}(\w_i)$; see \cref{sec:weave_decomp}.

For any positive integer $n$, we define a category $\mathfrak{B}_n$ whose set of objects is $\Br_n^+$ and whose morphisms are sequences of braids such that any two consecutive braids are related via the moves $\sigma_i\sigma_{i+1}\sigma_i \to \sigma_{i+1}\sigma_i\sigma_{i+1}$, $\sigma_i \sigma_j \to \sigma_j \sigma_i$ ($|i - j|\geq 2$), the trivalent move $\sigma_i^2 \to \sigma_i$, the cup move $\sigma_i^2 \to 1$, and the cap move $1 \to \sigma_i^2$; see \cref{sec:mcs_cat}. 

We define a functor $\mathfrak M \colon \mathfrak{B}_n \to \mathfrak{C}$ via Morse complex sequences associated to a sequence of braids related by braid moves; we call them trivial monodromy MCS sequences. This allows us to recast the weave decomposition of $X(\beta)$ in terms of Morse complex sequences and Legendrian front diagrams through trivial monodromy MCS sequences. 

\begin{theorem}[{\cref{thm:functors_from_mcs}}]\label{thm:intro_functors_comm}
    Let $n\in \Z_{\geq 1}$. There are functors $\mathfrak{A} \colon \mathfrak{B}_n \to \mathfrak{W}_n$ and ${\mathfrak M \colon \mathfrak{B}_n \to \mathfrak{C}}$ such that the following diagram commutes:
    \[
    \begin{tikzcd}[sep=scriptsize]
        \mathfrak{B}_n \ar[dr,swap,"\mathfrak M"] \ar[rr,"\mathfrak{A}"] && \mathfrak{W}_n \ar[dl,"\mathfrak{X}"] \\
        & \mathfrak{C}&
    \end{tikzcd}.
    \]
\end{theorem}

Next, we consider only the morphisms in the braid category $\mathfrak{B}_n$ that arise from an underlying normal ruling of the Legendrian $(-1)$-closure of $\beta\Delta$ for any given braid $\beta \in \Br_n^+$. This yields a wide subcategory that we denote by $\mathfrak{R}_n \subset \mathfrak{B}_n$. By comparing the ruling decomposition and the weave decomposition, we associate to each normal ruling $\rho$ a right simplifying weave $\w_{\rho}$. The construction of $\w_\rho$ is done in such a way that the switches of $\rho$ correspond to the trivalent vertices of $\w_\rho$, and departures of $\rho$ in $\beta$ correspond to the cups in $\w_\rho$; for example see \cref{fig:ruling_weave_intro}. The preimage $\mathfrak{A}^{-1}(\w_\rho)$ is a morphism $\mathfrak{r}_\rho$ in the ruling category $\mathfrak{R}_n$ that we call a right inductive morphism associated with $\rho$.

\begin{figure}[!htb]
    \centering
    \includegraphics{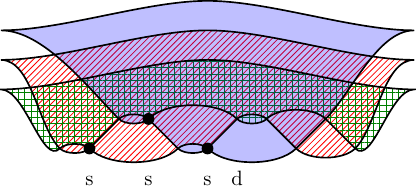}
    \hspace*{2.5cm}
    \raisebox{3ex}{\includegraphics{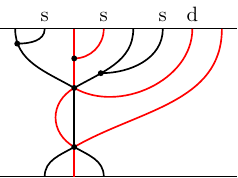}}
    \caption{Left: A normal ruling $\rho$ of the Legendrian $(-1)$-closure of $(\sigma_1^2\sigma_2^2)^2\Delta$. Right: The corresponding right simplifying weave $\w_\rho$. The labels ``s'' and ``d'' denote switches and departures, respectively.}
    \label{fig:ruling_weave_intro}
\end{figure}

Associated to the normal ruling $\rho$ of $\La(\beta\Delta)$, Henry--Rutherford \cite{HenryRutherford15} defined an injective map 
\[
\phi_\rho \colon \C^{r(\rho)-\binom n2} \times (\C^\ast)^{s(\rho)} \hooklongrightarrow \Aug(\La(\beta\Delta)),
\]
and we compare this injective map with the injective map $\phi_{\mathfrak{r}_\rho}$ induced by $\mathfrak M(\mathfrak{r}_\rho)$ in the following theorem.

\begin{theorem}[{\cref{thm:corr_decomps}}]\label{thm:main_technical}
   Let $\mathfrak{r}_\rho\colon \beta\to \Delta$ be a right inductive morphism in $\mathfrak{R}_n$ associated with the normal ruling $\rho$ of $\La(\beta\Delta)$, and assume it has $c$ cup moves and $t$ trivalent moves. There exists a bijective map $f_{\mathfrak{r}_\rho} \colon \C^{c} \times (\C^\ast)^{t} \to \C^{r(\rho)-\binom n2} \times (\C^\ast)^{s(\rho)}$ determined by $\mathfrak{r}_\rho$ such that the following diagram commutes:
    \[
    \begin{tikzcd}[sep=scriptsize]
        \C^{c} \times (\C^\ast)^{t} \rar{\phi_{\mathfrak{r}_\rho}} \dar[swap]{f_{\mathfrak{r}_\rho}} & X(\beta) \dar{\alpha}[swap]{\cong} \\
        \C^{r(\rho)-\binom n2} \times (\C^\ast)^{s(\rho)} \rar{\phi_{\rho}} & \Aug(\La(\beta\Delta))
    \end{tikzcd}.
    \]
    In particular, under the isomorphism $\alpha$, the images of $\phi_\rho$ and $\phi_{\mathfrak{r}_\rho}$ are equal.
\end{theorem}

\begin{remark}
    In \cref{thm:main_technical} we consider normal rulings of the $(-1)$-closure of $\beta\Delta$. In the case of rainbow closures, i.e., when $\beta = \Delta\gamma$ for some $\gamma \in \Br^+_n$, the bijective map $f_{\mathfrak{r}_\rho}$ is in fact explicitly computable using the combinatorics of MCS sequences; see \cref{ex:hopf}.
\end{remark}

\subsection{Further results}

\subsubsection{Cluster variables via Morse complex sequences}
A cluster structure on the braid variety $X(\beta)$ was defined by Casals--Gorsky--Gorsky--Le--Shen--Simental \cite{CGGLSS} and independently by Galashin--Lam--Sherman-Bennett--Speyer \cite{GLSS}, in particular confirming a conjecture of Leclerc for Richardson varieties \cite{Leclerc}. In the former construction, a Demazure weave $\w \colon \beta \to \Delta$ gives a toric chart $(\C^\ast)^t\hookrightarrow X(\beta)$, and the set of Lusztig cycles indexed by the trivalent vertices of $\w$ (corresponding to a special choice of generating set for $H_1(L(\w))$ where $L(\w)$ is the Legendrian surface defined by the weave $\w$) defines the toric coordinates known as cluster coordinates or cluster variables.

We expand the toolbox for contact and symplectic applications of cluster structures by realizing cluster variables coming from right inductive weaves in terms of MCSs. In particular, we show that the formal variables in the SR-form MCS associated to a normal ruling, called handleslide marks, have a cluster theoretic interpretation (we call them formal variables since we do not assign particular complex numbers to these variables).

\begin{theorem}[{\cref{thm:cluster_variables_mcs}}]\label{thm:cluster_variables_mcs_intro}
    Let $\beta \in \Br_n^+$, and let $\rho_0$ be the maximally switched normal ruling of $\La(\beta\Delta)$. There exists an algorithm to compute the cluster variables associated with a right inductive morphism $\mathfrak{r}_{\rho_0} \colon \beta \to \Delta$ from the formal framed SR-form MCS associated with $\rho_0$, together with the combinatorics of the Lusztig cycles determined by $\mathfrak{r}_{\rho_0}$.
\end{theorem}
\begin{remark}
   The $s$-variables used in the definition of cluster variables (\cref{dfn:s-variable_weave}) are precisely the handleslide marks of the switches in the formal framed SR-form MCS associated with the maximally switched normal ruling which can be obtained by applying the A-to-SR-form algorithm from \cite{HenryRutherford15} to the formal (framed) A-form MCS associated to $\La(\beta\Delta)$, with respect to $\rho_0$; see \cref{sec:cluster_via_mcs} for details. Our algorithm is a translation of the algorithm from~\cite{CGGLSS} to the MCS setting.
\end{remark}

\subsubsection{Decompositions via cycle deletions}

Given a Demazure weave $\w \colon \beta \to \Delta$, one can sometimes generate a weave decomposition of $X(\beta)$ using a process we call cycle deletion; see \cref{sec:cycle_del_weaves} for details. The process is a combinatorial incarnation of $1$-surgery performed on the surface $L(\w)$ at Lagrangian disks with boundary on (the Lagrangian projection of) $L(\w)$, representing Lusztig cycles. Via our functor $\mathfrak{A}$ in \cref{thm:intro_functors_comm}, we give a complete description of Lusztig cycles and cycle deletion in terms of Morse complex sequences in \cref{sec:lusztig_cycles_braid_cat}.

Cycle deletion is not defined at arbitrary Lusztig cycles, but only at so-called \textsf Y-trees. We show in \cref{sec:lusztig_cycles_braid_cat} that cycle deletions at \textsf Y-trees define $2$-morphisms in the braid and weave categories $\mathfrak{B}_n$ and $\mathfrak{W}_n$, respectively. We also obtain the following result.

\begin{proposition}[{\cref{prop:2-functor_mcs_weave}}]
    Let $n\in \Z_{\geq 1}$. There is a $2$-functor $\mathfrak{A} \colon \mathfrak{B}_n \to \mathfrak{W}_n$ whose underlying $1$-functor is the one from \cref{thm:intro_functors_comm}.
\end{proposition}

An interesting question is: for which braids can one start with an inductive weave and repeatedly perform cycle deletion to generate a weave decomposition of the braid variety? A braid for which this is possible is called cycle decomposable.

\begin{proposition}[{\cref{prop:torus_n2_decomp}}]
    For every $k\geq 1$, the braid $\sigma_1^{k+1} \in \Br_2^+$  is cycle decomposable.
\end{proposition}

The question of which braids are cycle decomposable is subtle. Firstly, the set of Lusztig cycles (and \textsf Y-trees) is not invariant under weave equivalence. Secondly, a necessary condition for a braid to be cycle decomposable is that it admits an inductive weave such that all of its Lusztig cycles are \textsf Y-trees. Given a braid, it is not clear whether such an inductive weave exists; even if it exists it may be difficult to find it. See \cref{sec:decompositions_cycle_del} for further discussion regarding these questions.

Even in cases where we are able to decompose $X(\beta)$ via cycle deletion, we show in \cref{ex:deletion_neq_ruling} that the resulting decomposition does not always agree with the decomposition given by right simplifying weaves that corresponds to the Deodhar or ruling decomposition. In fact, even the codimension 1 pieces do not always agree.

\begin{remark}
For braids of the form $\beta = \Delta\gamma$ for $\gamma \in \Br_n^+$, a result of Casals--Weng \cite[Section 4.9]{CasalsWeng} showed that the zero loci of the cluster variables correspond to the codimension 1 pieces by cycle deletions from a right inductive weave. For general $\beta$, Galashin--Lam--Sherman-Bennett--Speyer \cite[Section 7.2]{GLSS} showed that the zero loci of the cluster variables correspond to the codimension 1 pieces in the Deodhar decomposition. Since the cluster structures \cite{CasalsWeng}, \cite{CGGLSS}, and \cite{GLSS} agree (see \cite[Theorem 1.1]{CGGSBS} for the comparison of the second and third ones in general, and \cite[Theorem 1.2]{CasalsLi} for the comparison between the first and second ones in the Bott--Samelson case), we expect that for double Bott--Samelson cells, the codimension 1 pieces by cycle deletions from a right inductive weave agree with the pieces in the Deodhar decomposition, but for general braid varieties they do not agree (in particular, for general braid varieties, the cluster variables cannot be defined using cycle deletions as in \cite{CasalsWeng}).
\end{remark}

\subsubsection{Microlocal sheaves and Morse complex 2-families}
    The functor $\mathfrak{A}$ in \cref{thm:intro_functors_comm}, in fact, provides an explicit comparison between weaves and Morse complex 2-families (MC2F) introduced by Rutherford--Sullivan \cite{rutherford2018generating}.
    
    If $\mathfrak{m}$ is a morphism in $\mathfrak{B}_n$, we denote by $\Lambda(\mathfrak{m})$ the Legendrian lift of the exact immersed Lagrangian cobordism in $T^\ast \D^2$ induced by the MCS braid moves in $\mathfrak{m}$; see \cref{sec:weaves_aug_sheaves} for details. Given such a morphism, it turns out that the algebraic variety $\mathfrak M(\mathfrak{m})$ in \cref{thm:intro_functors_comm} is isomorphic to the space of equivalence classes of MC2Fs. This allows us to prove the following ``representations are sheaves'' result. Note that augmentations are one-dimensional representations.
    \begin{theorem}[{\cref{thm:sheaves_mc2f}}]\label{thm:sheaves_mc2f_intro}
        Let $\mathfrak{m} \colon \beta \to\beta'$ be a morphism in $\mathfrak{B}_n$. There is a bijection from the equivalence classes of simple sheaves on $\D^2 \times \R$ with singular support on $\Lambda(\mathfrak{m})$ to the equivalence classes of MC2Fs on the front projection of $\Lambda(\mathfrak{m})$.
    \end{theorem}
    \begin{remark}
        The above result proves that representations are sheaves over $\Z/2\Z$ since the equivalence between MC2Fs and augmentations of the Legendrian cellular dga, and the equivalence between the Legendrian cellular dga and the Chekanov--Eliashberg dga are proved over $\Z/2\Z$ \cite{rutherfordsullivan2019,rutherford2018generating}. In fact, this is the inverse to the map constructed in \cite{rutherford2021sheaves}. When we consider sheaves over a general field, we expect that our result implies augmentations are sheaves for Legendrian weaves (our definition of MC2F includes marked points and gives rise to sheaves with non-trivial microlocal monodromy along the Legendrian, in contrast to the construction in \cite{rutherford2021sheaves}, which does not include marked points and only gives sheaves with trivial microlocal monodromy instead of all simple sheaves \cite[Section 6.4]{rutherford2021sheaves}).
    \end{remark}
    \begin{remark}
        The general principle that ``representations of the Chekanov--Eliashberg dga are sheaves,'' first posited by Shende--Treumann--Zaslow~\cite{STZ_ConstrSheaves} was proven for augmentations of Legendrian knots~\cite{NRSSZ}. For more general Legendrians, see \cite[Theorem 1.4]{ganatra2024microlocal} combined with \cite[Theorem 1.13]{ganatra2022sectorial} or \cite[Theorem 1.1]{chantraine2017geometric}, and further combined with \cite[Theorem 5.8]{bourgeois2012effect}, \cite[Theorem 83]{ekholm2017duality}, and \cite[Theorems 1.1 and 1.2]{asplund2020chekanov}. In this context, the contribution of our proof of \cref{thm:sheaves_mc2f} is a direct and combinatorial comparison between augmentations and microlocal sheaves for a particular class of Legendrian surfaces.
        Moreover, as an additional feature of our comparison, our approach shows that the augmentations-sheaves correspondence is explicitly compatible with the sheaf and ruling decomposition of the augmentation variety. 
    \end{remark}
 
\subsection{Organization}
    In \cref{sec:background} we discuss background on Legendrian links, normal rulings, and weaves. In \cref{sec:decomp_braid_varieties_mcs} we give a definition of Morse complex sequences. We define weave decomposition and Deodhar decompositions on the braid variety and the ruling decomposition on the augmentation variety. We review how to compute cluster variables of the maximal cluster torus associated to right inductive weaves. In \cref{sec:mcs_and_alg_weaves} we define the braid category and recast the weave decomposition on the braid variety in terms of Morse complex sequences. We prove in \cref{sec:normal_rulings_cat} that the ruling decomposition of the augmentation variety and the weave decomposition of the braid variety of $(-1)$-closures of $\beta\Delta$ agree. Finally, in \cref{sec:weave_ruling_decomp_cycle_deletion} we show that the cluster variables of the maximal cluster torus can be computed via Morse complex sequences. We discuss cycle deletion both from the point of view of weaves and Morse complex sequences, how it sometimes generates a decomposition of a braid variety, and open questions related to cycle deletion.

\subsection*{Acknowledgments}
This project was initiated at American Institute of Mathematics (AIM) during the workshop ``Cluster algebras and braid varieties'' in January 2023. We are grateful to AIM and the workshop organizers Roger Casals, Mikhail Gorsky, Melissa Sherman-Bennett, and Jos\'{e} Simental. We thank Youngjin Bae, Marco Castronovo, and Melissa Sherman-Bennett for useful discussions at the initial stages of this project, and also Roger Casals, Zijun Li, Lenny Ng, and Jos\'e Simental for helpful conversations. We also thank Roger Casals, Lenhard Ng, Dan Rutherford, Melissa Sherman-Bennett, and Jos\'e Simental for comments on an earlier draft.

JA was partially supported by the Knut and Alice Wallenberg Foundation. OCS was supported by NSF grant DMS-2103188. WL was partially supported by the AMS-Simons Travel Grant.

\section{Background on braids, rulings and weaves}\label{sec:background}

In this section we define the geometric and algebraic objects that we work with. We also prove that braid varieties are augmentation varieties.

\subsection{Braids}
     The set of $n$-stranded braids forms a group under concatenation; it called the braid group and is denoted by $\Br_n$.
    
     \begin{notation}
        Let $\sigma_1,\ldots,\sigma_{n-1}$ denote the Artin generators of the braid group $\Br_n$ where the strands are labeled from $1$ to $n$ from bottom to top. Let $\Br_n^+$ be the submonoid of the braid group generated by non-negative powers of the Artin generators. We also let
        \[
            \Delta_n \coloneqq (\sigma_1)(\sigma_2 \sigma_1) \cdots (\sigma_{n-1}\cdots \sigma_1),
        \]
        denote the \emph{half twist on $n$ strands}. We usually denote the half twist by $\Delta$.
    \end{notation}
     For any element $w \in S_n$ of the symmetric group, we can choose a presentation of $w$ as words of elementary permutations $s_i \in S_n$, and lift it to a positive braid $\beta(w)$ by lifting $s_i$ to the braid $\sigma_i$. We define the \emph{length of $w \in S_n$}, denoted by $\ell(w) \in \Z$, to be the minimal length positive braid lift $\beta(w) \in \Br_n^+$.
    
    \begin{definition}[Demazure product]\label{def:demazure_prod}
        The \emph{Demazure product} is the map $\delta \colon \Br_n^+ \to S_n$ defined by
        \[
            \delta(\beta \sigma_i) = \begin{cases}
                \delta(\beta) s_i, & \text{if }\ell(\delta(\beta)s_i) = \ell(\delta(\beta)) + 1, \\
                \delta(\beta), & \text{if }\ell(\delta(\beta)s_i) = \ell(\delta(\beta)) - 1 .\\
            \end{cases}
        \]
         The Demazure product of the half twist is denoted by $w_0 \coloneqq \delta(\Delta)$.
    \end{definition}

\subsection{Legendrian links}
Let $M$ be an $m$-dimensional manifold. The cotangent bundle of $M$ equipped with the one-form $\lambda \coloneqq \sum_{i=1}^m y_i dx_i$ is an exact symplectic manifold, where $(x_1,\ldots,x_m)$ are coordinates on $M$ and $(y_1,\ldots,y_m)$ are coordinates in the fiber directions of the cotangent bundle. The \emph{$1$-jet space of $M$} is the contact manifold $J^1M \coloneqq T^\ast M \times \R_z$ equipped with the contact structure $\xi \coloneqq \ker(dz-\lambda)$.

A smooth embedding $\Lambda\subset J^1 M$ of a collection of copies of $S^{m-1}$ is called a \emph{Legendrian link} if $T_x \Lambda \subset \xi_x$ for all $x \in \Lambda$. A Legendrian link with only a single component is called a \emph{Legendrian knot}. Two Legendrian links are \emph{Legendrian isotopic} if they are smoothly isotopic through Legendrian links.

The two projections
\[
    \pi_F \colon J^1 M \longrightarrow M \times \R, \qquad \pi_L \colon J^1 M \longrightarrow T^\ast M
\]
are called the \emph{front projection}, and the \emph{Lagrangian projection} respectively. The image of a Legendrian link $\Lambda$ in $\pi_F$ is called a \emph{front diagram} for $\Lambda$. In the case where $M = \R$, $\pi_F$ and $\pi_L$ are respectively the projections to the $xz$-plane and the $xy$-plane. By the Legendrian condition, $\pi_F(\Lambda)$ is an immersed curve with no vertical tangencies and cusp singularities that are locally modeled on $y^2 = \pm x^3$ near the origin, and all the double points are such that the strand with the smallest slope is the overstrand. Conversely any such curve lifts to a unique Legendrian in $\R^3$. In contrast, $\pi_L(\Lambda)$ is an immersed curve with zero oriented area. There is a standard way of passing from the front projection to the Lagrangian projection of a Legendrian link $\Lambda \subset \R^3$, called Ng's resolution; see \cite[Section 2.1]{ng2003computable}.

\begin{definition}[$(-1)$-closure of a braid]
    Let $\beta \in \Br_n^+$. The \emph{$(-1)$-closure of $\beta$} is the Legendrian lift $\Lambda(\beta) \subset \R^3$ of the oriented front diagram depicted in \cref{fig:closure}.
\end{definition}

\begin{figure}[!htb]
    \includegraphics{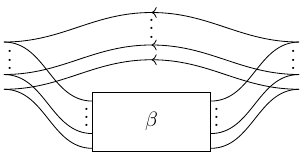}
    \caption{The front projection of the $(-1)$-closure of $\beta$ (in plat position).}
    \label{fig:closure}
\end{figure}

In this paper, we are mainly concerned with Legendrian links $\Lambda$ that are the $(-1)$-closure of $\beta\Delta$, where $\beta \in \Br_n^+$.

\begin{definition}[Nearly plat]
    Let $D$ be the front diagram of a Legendrian link $\Lambda \subset \R^3$.
     We say the front diagram $D$ is in \emph{nearly plat position} if every cusp (both left and right cusps) and every crossing has a distinct $x$-coordinate.
\end{definition}

\subsection{Normal rulings}\label{sec:rulings}
In this section we briefly review the definition of a (graded) normal ruling, independently introduced by Fuchs \cite{Fuchs03} and Chekanov--Pushkar{\cprime} \cite{ChekanovRuling,ChekanovPushkar}.

\begin{definition}[Maslov potential]\label{dfn:maslov}
    A \emph{Maslov potential} on a front diagram $D$ is a locally constant function 
    \[
    \mu \colon \pi_F^{-1}(D \setminus \{\text{cusps}\}) \longrightarrow \Z
    \]
    that increases by one when we pass upwards through a cusp of $D$, and decreases by one when we pass downwards through a cusp of $D$.
\end{definition}

From now on, unless explicitly stated, fronts are equipped with a Maslov potential.

\begin{definition}[Grading of crossing]\label{dfn:grading_crossing}
    Let $\mu$ be a Maslov potential on a front diagram $D$. The \emph{grading} of a crossing $c$ in $D$ is defined as the difference $\mu(s_+)-\mu(s_-)$ where $s_+$ and $s_-$ are the over strand and under strand at $c$, respectively.
\end{definition}

\begin{definition}[Normal ruling]\label{dfn:normal_ruling}
    A \emph{normal ruling} of the front diagram $D$ of a Legendrian link $\Lambda \subset \R^3$ is a decomposition of $D$ into pairs of paths (called \emph{ruling paths}) where the two paths begin at the same left cusp and end at the same right cusp, and the pair of paths co-bound a topological disk (called a \emph{ruling disk}). The ruling paths are required to be smooth, except at cusps and certain crossings called \emph{switches}, and satisfy:
    \begin{enumerate}
        \item Any two ruling paths only intersect at crossings or cusps.
        \item Near a switch, the ruling disks must look like a diagram in \cref{fig:ruling_switches} and the crossing must have degree $0$.
    \end{enumerate}
    There are three types of crossings in a normal ruling:
    \begin{description}
        \item[Switches] The crossing has grading zero and the ruling disks are either nested or disjoint on both sides of the crossing; see \cref{fig:ruling_switches}.
        \item[Departures] To the left of the crossing the ruling disks are nested or disjoint and to the right of the crossing the ruling disks are interlaced; see \cref{fig:ruling_deps}.
        \item[Returns] To the left of the crossing the ruling disks are interlaced and to the right of the crossing the ruling disks are nested or disjoint; see \cref{fig:ruling_rets}.
    \end{description}
\end{definition}

\begin{figure}[!htb]
    \begin{subfigure}{\textwidth}
        \centering
        \includegraphics[width=.8\textwidth]{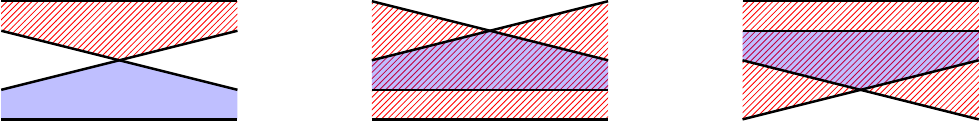}
        \caption{}\label{fig:ruling_switches}
    \end{subfigure}
    
    \begin{subfigure}{\textwidth}
        \centering
        \includegraphics[width=.8\textwidth]{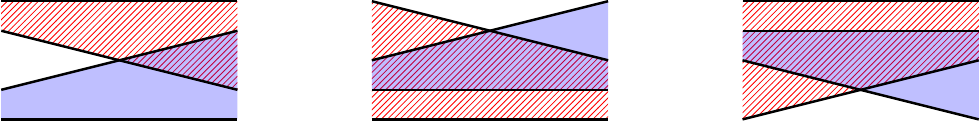}
        \caption{}\label{fig:ruling_deps}
    \end{subfigure}
    
    \begin{subfigure}{\textwidth}
        \centering
        \includegraphics[width=.8\textwidth]{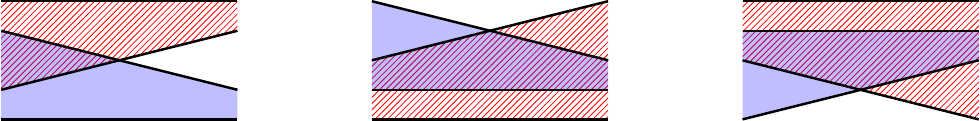}
        \caption{}\label{fig:ruling_rets}
    \end{subfigure}
    \caption{All the possible configurations of a normal ruling near a crossing. The top row gives all possible configurations near switches, the second row configurations near departures, and the third row near returns.}
    \label{fig:ruling}
\end{figure}
\begin{remark}
     We only consider $0$-graded normal rulings and call them ``normal rulings.'' We also warn the reader that while normal rulings propagate under Legendrian isotopies, the rules of propagation are, in general, delicate. This is the main reason why many arguments involving rulings later on are verified by hand in a special setting. 
\end{remark}
\begin{notation}
    Let $D$ be the front diagram of a Legendrian link $\Lambda \subset \R^3$. We denote the set of all normal rulings of $D$ by $\frakr(D)$. Given a normal ruling $\rho\in\frakr(D)$, let
    \begin{itemize}
        \item $s(\rho)$ be the number of crossings that are switches in $\rho$,
        \item $r(\rho)$ be the number of crossings that are returns in $\rho$, and
        \item $d(\rho)$ be the number of crossings that are departures in $\rho$.
    \end{itemize}
\end{notation}

\begin{definition}[Maximal normal ruling]
    Let $D$ be the front diagram of a Legendrian link $\Lambda \subset \R^3$. We say that $\rho \in \frakr(D)$ is a \emph{maximal normal ruling} if $s(\rho) \geq s(\rho')$ for all normal rulings $\rho' \in \frakr(D)$.
\end{definition} 

Arbitrary Legendrians may admit multiple maximal normal rulings, but not in our setting.

\begin{proposition}\label{prop:max_ruling}
    Let $\beta\in \Br_n^+$ and let $D$ be the front diagram of the $(-1)$-closure of $\beta\Delta$ in $\R^3$. There is a unique maximal normal ruling that we denote by $\rho_0 \in \frakr(D)$ with the property that $\rho_0$ has no departure at a crossing of $D$ corresponding to an Artin generator of $\beta$.
\end{proposition}

\begin{proof}
    There is a one-to-one correspondence between normal rulings before and after applying a number of Reidemeister 2-moves, canceling the crossings corresponding to Artin generators of $\Delta$; see \cref{fig:r2_normal_rulings}. Indeed, for any normal ruling of $D$, every crossing of $D$ corresponding to Artin generators of $\Delta$ must automatically be a departure, and the ruling disks persist under the Reidemeister 2-moves. After applying these Reidemeister 2-moves, we proceed by induction.

    \begin{figure}[!htb]
        \centering
        \includegraphics{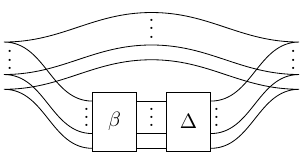}
        \hspace{5mm}
        \raisebox{8ex}{$\longrightarrow$}
        \hspace{5mm}
        \includegraphics{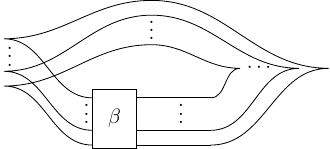}
        \caption{Before and after a number of Reidemeister 2-moves.}
        \label{fig:r2_normal_rulings}
    \end{figure}
    Assume $\beta = \sigma_{i_1} \cdots \sigma_{i_m}$ and label the crossings in the front diagram of $\beta$ from left to right by $c_1, \dots, c_m$. Consider any two strands in $D$ which go from a left cusp to a right cusp and are incident to $\beta\Delta$. They cross exactly once to the left of $\beta\Delta$ in $D$ and thus, in any normal ruling of $D$, this crossing must be a departure. If we consider $c_1$ (the leftmost crossing in $\beta\Delta$), it must then be the second intersection of the two ruling disks and so, in any normal ruling of $D$, the crossing must be a return. Assuming that $c_1 \cdots c_{p-1}$ are determined, the next crossing $c_p$ corresponds to the intersection of two ruling disks. If it is the second time the two ruling disks intersect (counted left to right), $c_p$ is forced to be a return. Otherwise, we define $c_p$ to be a switch. This is uniquely maximal.
\end{proof}

\subsection{Augmentation varieties}\label{sec:augmentation_variety}
    One method of finding invariants of Legendrian links is through the Chekanov--Eliashberg differential graded algebra (dga). The Chekanov--Eliashberg dga was first defined for Legendrian links in $\R^3$ over $\Z/2\Z$ by Chekanov \cite{ChekanovDGA} and also independently developed in generality by Eliashberg \cite{eliashberg1998invariants}; it fits in the larger framework of symplectic field theory developed by Eliashberg--Givental--Hofer \cite{EGH}. Many other versions and flavors of the Chekanov--Eliashberg dga have been introduced and studied since then \cite{ekholm2005the, ENS, ekholm2007legendrian, ekholm2017duality, an2020a, asplund2020chekanov}.
    
    \begin{notation}\label{notn:chekanov_eliashberg}
        Let $\beta \in \Br_n^+$ and let $D$ be the front diagram of the $(-1)$-closure of $\beta\Delta$ in $\R^3$ equipped with a Maslov potential $\mu$ and decorated with marked points $t_1, \ldots, t_n$. Let $(\A(D), \partial)$ denote the Chekanov--Eliashberg dga of $D$ with coefficients in $\Z[t_1^{\pm 1},\ldots,t_n^{\pm 1}]$. If $\Lambda \subset \R^3$ is a Legendrian link, we also use the notation $(\A(\Lambda), \partial)$ or $\A(\La)$ to denote the Chekanov--Eliashberg dga of the front projection of $\Lambda$. 
    \end{notation}
    \begin{remark}
        We do not give the full definition of $(\A(D), \partial)$, but we give a description of the generators of $\A(D)$ (for a full description, see e.g.\@ \cite{EtnyreNg}). They are in bijection with the crossings of $D$, right cusps of $D$, and the marked points $t_1, \ldots, t_n$. The grading of the generators in $\A(D)$ corresponding to crossings of $D$ are given as in \cref{dfn:grading_crossing}. The generators corresponding to the right cusps of $D$ have degree $1$, and the grading of each marked point $t_i^{\pm 1}$ is defined to be zero (because the rotation number of $D$ is equal to zero).
    \end{remark}

    The stable tame isomorphism class of $\A(D)$ is a powerful Legendrian isotopy invariant of $\Lambda$ whose precise definition we do not discuss. We encourage the reader to consult the survey by Etnyre--Ng \cite{EtnyreNg} and references therein. Our interest in $\A(D)$ mainly regards its augmentations.

    \begin{definition}[Augmentation]\label{dfn:augmentation}
        Let $D$ be the front diagram of a Legendrian link $\Lambda \subset \R^3$. A \emph{$\C$-valued augmentation} is a unital dga map $\epsilon\colon (\A(D),\partial)\to(\C,0)$, where $(\C,0)$ is the dga with zero differential and one generator in degree $0$.
    \end{definition}

    \begin{definition}\label{def:dg-homotopy}
        Let $D$ be the front diagram of a Legendrian link $\Lambda \subset \R^3$. Given two augmentations $\epsilon_1,\epsilon_2 \colon (\A(D), \partial) \to (\C,0)$, a \emph{dga homotopy} between $\epsilon_1$ and $\epsilon_2$ is a $\C$-linear map $\eta \colon (\A(D), \partial) \to (\C,0)$ of degree $1$ such that
        \begin{enumerate}
            \item $\epsilon_1(x)-\epsilon_2(x)=(\eta\circ\partial)(x)$ for any $x\in \A(D)$ and
            \item $\eta(xy)=\eta(x)\epsilon_2(y)+(-1)^{|x|}\epsilon_1(x)\eta(y)$ for all $x,y\in \A(D).$
        \end{enumerate}
    \end{definition}
    \begin{definition}[Augmentation variety]\label{def:aug_variety}
        Let $D$ be the front diagram of a Legendrian link $\Lambda \subset \R^3$. Denote the degree $0$ generators of $\A(D)$ by $a_1,\ldots,a_N$. The \emph{naive augmentation variety} is
        \[
        \V(D) \coloneqq \{((\epsilon(t_i))_{i=1}^n,(\epsilon(a_i))_{i=1}^N) \in (\C^\ast)^n \times \C^{N} \mid \epsilon \text{ is a $\C$-valued augmentation of }\A(D)\}.
        \]
        The \emph{augmentation variety} of $D$ is the algebraic variety (scheme)
        \[
            \Aug(D) \coloneqq \begin{cases}
                \varnothing & \text{if } \V(D) = \varnothing,\\
                \V(D)/{\simeq} &\text{otherwise,}
            \end{cases}
        \]
        where
        \[
        (\epsilon(t_1),\dots,\epsilon(t_c),\epsilon(a_1),\ldots,\epsilon(a_N)) \simeq (\phi(t_1),\dots,\phi(t_c),\phi(a_1),\ldots,\phi(a_N))
        \]
        if and only if $\epsilon$ and $\phi$ are dga homotopic.
    \end{definition}
    \begin{remark}\label{rmk:aug_naive_equal}
        If $\A(D)$ does not have any generators of negative degree, it is clear from the definition that two augmentations $\varepsilon$ and $\phi$ are dga homotopic if and only if $\varepsilon=\phi$, in which case $\Aug(D) = \V(D)$. In general, we can view $\Aug(D)$ as a functor $\operatorname{CRing}^{\textit{op}} \to \operatorname{Grpd},\, R \mapsto \V(D; R)/{\simeq}$ (where $\V(D; R)$ consists of all $R$-valued augmentations) and this defines a stack.
    \end{remark}

\begin{lemma}\label{thm:bijection_equiv}
    Let $\beta \in \Br_n^+$ and let $\{w_{ij}\}_{1 \leq i < j \leq n}$ denote the generators of $\A(\La(\Delta\beta))$ that correspond to the crossings of $\Delta$. The map
    \begin{align}\label{eq:equiv_factor_new}
        \V(\La(\Delta\beta)) &\longrightarrow \Aug(\La(\Delta\beta)) \times \C^{\binom n2} \\
        \epsilon &\longmapsto ([\epsilon],\{\epsilon(w_{ij})\}_{1 \leq i < j \leq n}) \nonumber
    \end{align}
    is an isomorphism. 
\end{lemma}
\begin{proof}
    Label the generators of $\A(\La(\Delta\beta))$ as follows: the crossings of $\beta\Delta$ are the degree $0$ generators, label them by $\{w_{ij}\}_{1 \leq i < j \leq n}\cup \{x_i\}_{1\leq i\leq r}$; the crossings to the left of $\Delta\beta$ are the degree $-1$ generators, label them by $\alpha_{ij}$, where $\alpha_{ij}$ is the crossing between the $i$-th and $j$-th cusp if the cusps are labeled $1, \ldots, n$ from bottom to top; the crossings to the right of $\Delta\beta$ are the degree $1$ generators, label them by $\kappa_{ij}$, where $\kappa_{ij}$ is the crossing between the $i$-th and $j$-th cusp. Note that there are $\binom n2$ degree $-1$ generators and $\binom n2$ degree $1$ generators.

    In $\A(\La(\Delta\beta))$, the differential of any degree $0$ generator $x$ is given by
    \begin{equation}\label{eq:diff_homotopy_1}
        \partial x = \begin{cases}
        \alpha_{ij} \pm \sum_{i < k < j} \alpha_{kj}w_{ik}, & x = w_{ij}, \\
        0, &\text{otherwise}.
    \end{cases}
    \end{equation}
    Because there are no generators of $\A(\La(\Delta\beta))$ with degree less than $-1$, if $\epsilon_1,\epsilon_2 \colon \A(\La(\Delta\beta)) \to \C$ are two augmentations, note that a dga homotopy between $\epsilon_1$ and $\epsilon_2$ is completely determined by the values on the generators $\alpha_{ij}$.

    We now show that for any augmentation $\epsilon \colon \A(\La(\Delta\beta)) \to \C$, there is a unique dga homotopy to the augmentation $\epsilon' \colon \A(\La(\Delta\beta)) \to \C$ satisfying
    \[
    \epsilon'(x) = \begin{cases}
        0, & x = w_{ij}, \\
        \epsilon(x), &\text{otherwise}.
    \end{cases}
    \]
    By \eqref{eq:diff_homotopy_1}, it follows that if $\eta$ is a dga homotopy between $\epsilon$ and $\epsilon'$, then
    \[
    \epsilon(x) = \begin{cases}
        \eta(\alpha_{ij}), & x = w_{ij}, \\
        \epsilon'(x), & \text{otherwise},
    \end{cases}
    \]
    which immediately shows that $\eta(\alpha_{ij})$ is uniquely determined by $\epsilon(w_{ij})$ for all $1 \leq i < j \leq n$, and hence $\eta$ is uniquely determined. It follows that the map
    \[
    \epsilon \longmapsto ([\epsilon],(\epsilon(w_{ij}))_{i,j})
    \]
    is an isomorphism.
\end{proof}

\begin{notation}\label{notn:backwards} If $\beta \in \Br_n^+$ is given by $\beta = \sigma_{i_1} \cdots \sigma_{i_m}$, 
    \begin{enumerate}
        \item we denote the reversed braid word $\ateb = \sigma_{i_m} \cdots \sigma_{i_1}$ by $\ateb \in \Br_n^+$ and
        \item we denote the ``upside-down'' braid word $\atebud = \sigma_{(n+1)-i_m} \cdots \sigma_{(n+1)-i_1}$ by $\atebud \in \Br_n^+$.
    \end{enumerate}
\end{notation}

    \begin{figure}[!htb]
        \centering
        \includegraphics{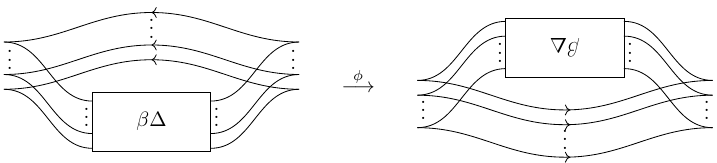}
        \caption{Transformation of the front of $\La(\beta\Delta)$ under the contactomorphism $\phi$ to the front diagram $(\Deltaud\atebud)\Lambdaud$.}
        \label{fig:front_transform}
    \end{figure}
\begin{lemma}\label{lem:reflect_iso}
    Let $\beta \in \Br_n^+$. There is an isomorphism
    \begin{equation}\label{eq:a-form_reflect_variety}
        \V(\La(\beta\Delta)) \cong \Aug(\beta\Delta) \times \C^{\binom n2}.
    \end{equation}
\end{lemma}
\begin{proof}
    First, we consider a Legendrian isotopy obtained by moving the teardrops in Ng's resolution of the $(-1)$-closure of $\beta\Delta$ from the right to the left as shown in \cref{fig:teardrop_isotopy_1}, followed by the contactomorphism $\phi(x,y,z) = (-x,y,-z)$. We end up with a Legendrian that is the Lagrangian $(-1)$-closure of $\atleD\ateb$, denoted by $\La_{\text{Lag}}(\atleD\ateb)$.
    \begin{figure}[!htb]
        \centering
        \includegraphics{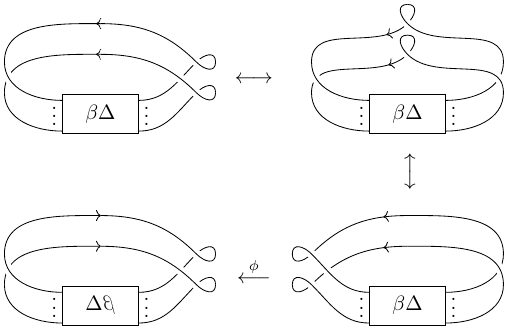}
        \caption{A Legendrian isotopy followed by the contactomorphism $\phi$ relating $\La(\beta\Delta)$ and $\La_{\text{Lag}}(\atleD\ateb)$.}
        \label{fig:teardrop_isotopy_1}
    \end{figure}
    The generators of $\A(\La(\beta\Delta))$ of degree $\pm 1$ correspond to crossings that are not in the braid word $\beta\Delta$. Under this transformation, they map to the degree $\pm 1$ generators of $\A(\La_{\text{Lag}}(\atleD\ateb))$. In the case where $\beta = \Delta\gamma$, it is simple to describe the bijection on the degree $0$ generators; in particular, the crossing $z_i$ maps to the crossing $z_{n-i}$.

    The contactomorphism of $\R^3$ defined by $\phi(x,y,z) = (-x,y,-z)$ transforms the front diagram of $\La(\beta\Delta)$ as depicted in \cref{fig:front_transform}. The result is the Legendrian submanifold that we denote by $(\Deltaud\atebud)\Lambdaud$.
    
    Therefore we have that 
    \[
        \V((\Deltaud\atebud)\Lambdaud) \cong \V(\La(\beta\Delta)) \quad\text{and}\quad \Aug((\Deltaud\atebud)\Lambdaud) \cong \Aug(\La(\beta\Delta)).
    \]
    Since the Lagrangian projection of $(\Deltaud\atebud)\Lambdaud$ is $\La_{\text{Lag}}(\atleD\ateb)$ as depicted in \cref{fig:teardrop_isotopy_1}, the proof of \cref{thm:bijection_equiv} now shows that there is an isomorphism
    \[
        \V((\Deltaud\atebud)\Lambdaud) \cong \Aug((\Deltaud\atebud)\Lambdaud) \times \C^{\binom n2},
    \]
    and hence we get the composition of isomorphisms
    \begin{equation*}
    \V(\La(\beta\Delta)) \cong \V((\Deltaud\atebud)\Lambdaud) \cong \Aug((\Deltaud\atebud)\Lambdaud) \times \C^{\binom n2} \cong \Aug(\La(\beta\Delta)) \times \C^{\binom n2}. \qedhere
    \end{equation*}
\end{proof}
\begin{remark}\label{rmk:nonexplicit_iso}
    In contrast to \cref{thm:bijection_equiv}, we do not have an explicit description of the isomorphism \eqref{eq:a-form_reflect_variety} for general $\beta$. If $\beta = \Delta\gamma$ it coincides with \eqref{eq:equiv_factor_new}, because in this case the isomorphism $\A(\La(\Delta\gamma\Delta)) \cong \A((\Deltaud\ammagud\Deltaud)\Lambdaud)$ is the identity on the degree $\pm 1$ generators, and acts by $z_i \mapsto z_{n-i}$ on the degree $0$ generators.
\end{remark}

It is useful to understand how normal rulings of $\La(\beta\Delta)$ and $(\Deltaud\atebud)\Lambdaud$ are related.

 \begin{lemma}\label{lma:rulings_upside_down}
    Let $\beta \in \Br_n^+$. There is a bijection
    \[
    \mathfrak{R}(\La(\beta\Delta)) \overset{\cong}{\longrightarrow} \mathfrak{R}((\Deltaud\atebud)\Lambdaud),
    \]
    sending a normal ruling $\rho \in \mathfrak{R}(\La(\beta\Delta))$ given by a sequence $(x_1,\ldots,x_m)$ of returns, switches and departures to the normal ruling $\ohrud \in \mathfrak{R}((\Deltaud\atebud)\Lambdaud)$ that is obtained by reversing the sequence $(x_1,\ldots,x_m)$ followed by changing all returns to departures, and vice versa.
\end{lemma}
\begin{proof}
    The contactomorphism $\phi$ yields a bijection between the ruling disks, and preserves the normality condition; see \cref{fig:front_transform}. By construction, the contactomorphism $\phi$ swaps the returns and departures.
\end{proof}

\subsection{Weaves}\label{sec:weaves}

A weave is a combinatorial diagram used to encode Legendrian surfaces in $\R^5$ with the standard contact structure, and were introduced by Casals--Zaslow \cite{CasalsZaslow20}. Roughly, it is a graph $\w$ in the unit disk $\D^2$ that encodes a Legendrian surface $L(\w) \subset \R^5$ by describing the singularities of its front projection in $\R^3$. The Lagrangian projection of $L(\w)$ to $\R^4$ with the standard symplectic structure is an exact (immersed) Lagrangian surface with Legendrian boundary.

\begin{definition}[$n$-weave]
An \emph{$n$-weave} $\w$ on the unit disk $\D^2$ is a set of $n-1$ embedded graphs $\{G_i\}_{i=1}^{n-1}$ with only trivalent vertices (see \cref{fig:weave-vertex1}) such that $G_i$ and $G_{i+1}$ intersect only at hexagonal vertices (see \cref{fig:weave-vertex2}), and $G_i$ and $G_j$ intersect only at smooth points on each graph when $|i-j|\geq 2$. 
\end{definition}

\begin{figure}[!htb]
    \centering
    \begin{subfigure}{0.45\textwidth}
        \centering
        \includegraphics[scale=0.75]{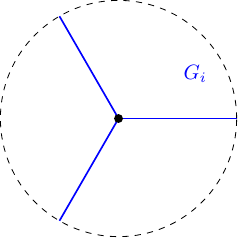}
        \caption{A trivalent vertex.}\label{fig:weave-vertex1}
    \end{subfigure}
    \begin{subfigure}{0.45\textwidth}
        \centering
        \includegraphics[scale=0.75]{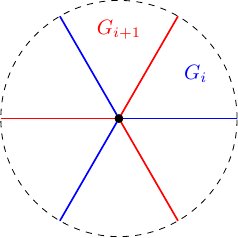}
        \caption{A hexagonal vertex.}\label{fig:weave-vertex2}
    \end{subfigure}
    \caption{The two types of vertices in a weave.}
\end{figure}

Given an $n$-weave $\w \subset \D^2$, we consider a singular surface in $\D^2 \times \R$ whose transverse self-intersection set is $\w$ and its projection onto $\D^2$ is an $n$-fold cover branched over the trivalent vertices of $\w$. There are local models for the Legendrian surface $L(\w)$ whose front projection is encoded by an $n$-weave, i.e., local models for trivalent and hexavalent vertices; see \cite[Definition 2.7]{CasalsZaslow20}. Note that in loc. cit., the Legendrian surface is referred to as a ``Legendrian weave'' and the corresponding graph is an ``$n$-graph.'' We have instead chosen to adopt the terminology of \cite{CGGS1,CGGS2,CGGLSS} and refer to the graph itself as a weave in order to better reflect our focus on braid varieties and their decompositions.

The restriction of the $n$-weave $\w$ to $\partial \D^2$ (denoted by $\w|_{\partial \D^2}$) is a collection of disjoint points, each one corresponding to $G_i$ for some $i\in \{1,\ldots,n-1\}$. We associate a positive braid $\beta(\w) \in \Br_n^+$ (that is well-defined up to cyclic permutation) to the restricted $n$-weave $\w|_{\partial \D^2}$ by traversing $\partial \D^2$ in the counterclockwise direction and replacing each point belonging to $G_i$ with the Artin generator $\sigma_i\in \Br_n^+$. The positive braid $\beta(\w)$ defines a Legendrian link $\Lambda_\partial(\w) \subset J^1S^1$ by taking the Legendrian lift of the closure of $\beta(\w)$ in $S^1 \times \R$; this is the $(-1)$-closure of $\beta(\w)$ in $J^1S^1$.

\begin{notation}
    Let $\w$ be an $n$-weave such that $\beta(\w) = \beta_1 \beta_2$ where the braid word $\beta_1 \subset \beta(\w)$ lies in $\partial \D^2_+ \coloneqq \partial \D^2 \cap \{z > 0\} \subset \partial \D^2$ and the braid word $\beta_2 \subset \beta(\w)$ lies in $\partial \D^2_- \coloneqq \partial \D^2 \cap \{z > 0\} \subset \partial \D^2$. We say that $\w$ is a \emph{weave from $\beta_1$ to $\beta_2$}, and write $\w \colon \beta_1 \to \beta_2$. We use a diffeomorphism $\Psi \colon \D^2 \setminus \{\pm 1\} \overset{\cong}{\to} (0,1) \times [0,1]$ that identifies $\partial \D^2_+$ with $(0,1) \times \{0\}$ and $\partial \D^2_-$ with $(0,1) \times \{1\}$. We call $\w|_{(0,1) \times [0,\epsilon)}$ for some small $\epsilon > 0$ the \emph{top} of the weave, and $\w|_{(0,1) \times (1-\epsilon,1]}$ the \emph{bottom} of the weave.
\end{notation}

Let $\beta_1,\beta_2,\beta_3, \beta_4 \in \Br_n^+$, and let $\w_1 \colon \beta_1 \to \beta_2$ and $\w_2 \colon \beta_3 \to \beta_4$ be two weaves. The \emph{horizontal composition of $\w_1$ and $\w_2$} is the weave $\w_1 \sqcup \w_2$ shown in \cref{fig:horizontal_comp}. For $\beta_2=\beta_3$, the \emph{vertical composition of $\w_1$ and $\w_2$} is the weave $\w_2 \circ \w_1$ shown in \cref{fig:vertical_comp}.

\begin{figure}[!htb]
    \centering
    \includegraphics[scale=0.9]{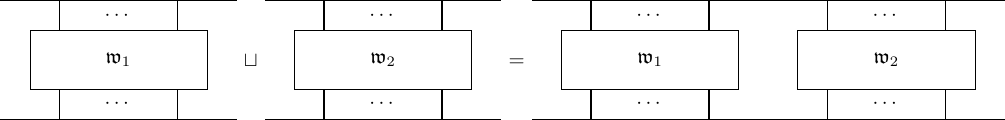}
    \caption{Horizontal composition of weaves.}
    \label{fig:horizontal_comp}
\end{figure}

\begin{figure}[!htb]
    \centering
    \includegraphics[scale=0.9]{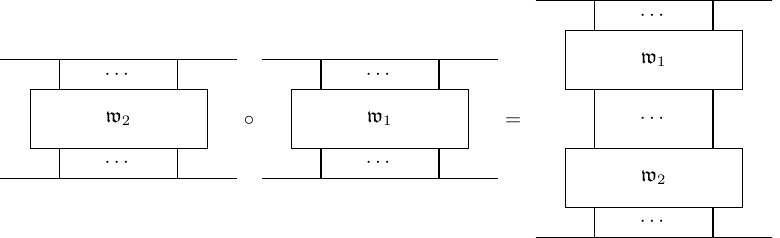}
    \caption{Vertical composition of weaves.}
    \label{fig:vertical_comp}
\end{figure}

Following \cite[Section 4]{CGGS1}, we now define sliced weaves, Demazure weaves, and simplifying weaves.

\begin{definition}[Sliced weave {\cite[Definition 4.2]{CGGS1}}]\label{dfn:sliced_weave}
    For a positive braid $\beta \in \Br_n^+$, a \emph{sliced weave} $\w$ is a weave $\beta \to \beta'$ that is defined as the vertical composition of weaves of the following six local models. The weave that:
    \begin{itemize}
        \item is trivial; \cref{fig:demazure1}.
        \item merges two adjacent strands into a trivalent vertex; \cref{fig:demazure2}.
        \item merges three adjacent strands into a hexagonal vertex; \cref{fig:demazure3} for the example with strands $G_iG_{i+1}G_i$. The case for $G_{i+1}G_iG_{i+1}$ is analogous.
        \item merges two adjacent strands into a \emph{cup}; \cref{fig:demazure4}.
        \item merges two adjacent strands into a \emph{cap}; \cref{fig:demazure5}.
        \item merges two distant strands (meaning $|i - j| > 1$) into a \emph{tetravalent vertex}; \cref{fig:demazure6}.
    \end{itemize}
\end{definition}
\begin{figure}[!htb]
    \centering
    \begin{subfigure}{0.32\textwidth}
        \centering
        \includegraphics{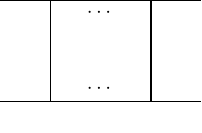}
        \caption{}\label{fig:demazure1}
    \end{subfigure}
    \begin{subfigure}{0.32\textwidth}
        \centering
        \includegraphics{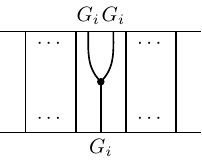}
        \caption{}\label{fig:demazure2}
    \end{subfigure}
    \begin{subfigure}{0.32\textwidth}
        \centering
        \includegraphics{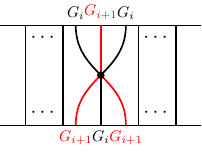}
        \caption{}\label{fig:demazure3}
    \end{subfigure}
    \begin{subfigure}{0.32\textwidth}
        \centering
        \includegraphics{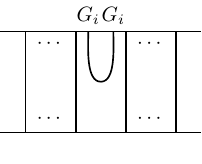}
        \caption{}\label{fig:demazure4}
    \end{subfigure}
    \begin{subfigure}{0.32\textwidth}
        \centering
        \includegraphics{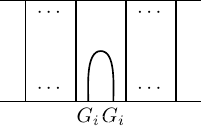}
        \caption{}\label{fig:demazure5}
    \end{subfigure}
    \begin{subfigure}{0.32\textwidth}
        \centering
        \includegraphics{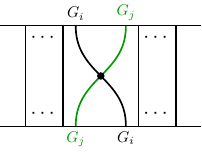}
        \caption{}\label{fig:demazure6}
    \end{subfigure}
    \caption{Local models for sliced weaves.}
\end{figure}

\begin{definition}[Demazure and simplifying weaves]
    Let $\beta \in \Br_n^+$.
    \begin{enumerate}
        \item A \emph{simplifying weave} is a sliced weave $\beta \to \delta(\beta)$ without any caps.
        \item A \emph{Demazure weave} is a simplifying weave $\beta \to \delta(\beta)$ without any cups.
    \end{enumerate}
\end{definition}

\begin{definition}
    We say that we perform a \emph{braid move} on a weave if we take the vertical composition with any one of the three local models depicted in \cref{fig:demazure1,fig:demazure3,fig:demazure6}.
\end{definition}

The Demazure product $\delta$ (\cref{def:demazure_prod}) allows us to construct a useful set of weaves.

\begin{notation}
    Let $\beta=\sigma_{i_1} \cdots \sigma_{i_r} \in \Br_n^+$. For $k\in \{1,\ldots,r\}$ we define $\beta^{\leq k}\coloneqq \sigma_{i_1}\dots \sigma_{i_k}$.     
\end{notation}

\begin{definition}[Right simplifying weave]\label{dfn:inductive_simplifying_weave}
    Suppose $\beta=\sigma_{i_1} \cdots \sigma_{i_r} \in \Br_n^+$. A \emph{right simplifying weave} is a weave $\w\colon \beta\to \delta(\beta)$ that is constructed by inductively defining braids $\beta_k'$ and weaves $\w(\beta^{\leq k}):\beta^{\leq k}\to\beta_k'$ (where $\beta_k'$ is not necessarily $\delta(\beta^{\leq k})$) as follows:
    \begin{itemize}
        \item Set $\beta_0'$ to be the empty word and $\w(\beta^{\leq 0})$ to be the empty weave.
        \item If $\delta(\beta_j'\sigma_{i_{j+1}})=\delta(\beta_j')s_{i_{j+1}}$, then we set $\beta_{j+1}'=\beta_j'\sigma_{i_{j+1}}$ and $\w(\beta^{\leq j}\sigma_{i_{j+1}})$ is obtained by the horizontal composition of $\w(\beta^{\leq j})$ and the trivial weave on a single strand belonging to $G_{i_{j+1}}$; see \cref{fig:inductive1}. 
        \item If $\delta(\beta_j'\sigma_{i_{j+1}}) = \delta(\beta_{j}')$, then apply braid moves to $\w(\beta^{\leq j})$ such that its rightmost strand belongs to $G_{i_{j+1}}$, and define $\w(\beta^{\leq j}\sigma_{i_{j+1}})$ to be the horizontal composition of $\w(\beta^{\leq j})$ and the trivial weave on a single strand belonging to $G_{i_{j+1}}$, followed by vertical composition with either of the following two local pictures:
            \begin{itemize}
                \item a weave that is trivial except that there is a trivalent vertex joining the two rightmost strands; see \cref{fig:inductive2}. In this case, we set $\beta_{j+1}'=\beta_j'$.
                \item a weave that is trivial except that there is a cup joining the two rightmost strands; see \cref{fig:inductive3}. In this case, we set $\beta_{j+1}'\sigma_{i_{j+1}}=\beta_j'$. 
    \end{itemize}
    \end{itemize}
\end{definition}
Note that the $\beta_k'$ depend on $\w$ and that when there is a choice of whether to introduce a trivalent vertex or a cup, these will in general produce distinct right simplifying weaves and different $\beta_k'$ braids will appear in the construction process. These choices may be constrained by the requirement that $\w$ is a weave from $\beta$ to $\delta(\beta)$.
\begin{definition}[Right inductive weave {\cite[Definition 4.5]{CGGLSS}}]\label{dfn:inductive_weave}
A \emph{right (left) inductive weave} $\w(\beta) \colon \beta \to \delta(\beta)$ is a right (left) inductive simplifying weave with no cups. 
\end{definition}

\begin{figure}[!htb]
    \centering
    \begin{subfigure}{0.3\textwidth}
        \centering
        \includegraphics[scale=0.9]{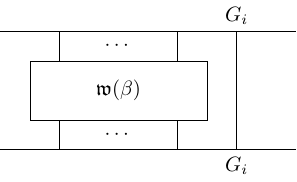}
        \caption{}\label{fig:inductive1}
    \end{subfigure}
    \begin{subfigure}{0.3\textwidth}
        \centering
        \includegraphics[scale=0.9]{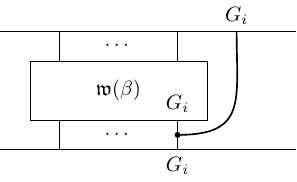}
        \caption{}\label{fig:inductive2}
    \end{subfigure}
    \begin{subfigure}{0.3\textwidth}
        \centering
        \includegraphics[scale=0.9]{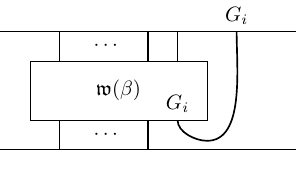}
        \caption{}\label{fig:inductive3}
    \end{subfigure}
    \caption{Defining local models for a right simplifying weave.}
\end{figure}

We now introduce an equivalence relation on the set of right simplifying weaves. By construction any right simplifying weave $\w$ is given by $\w_r \circ \cdots \circ \w_1$, where each $\w_i$ contains exactly one of the local models depicted in \cref{fig:inductive1,fig:inductive2,fig:inductive3}. There is a natural projection map $p$ from the set of right simplifying weaves to finite sequences $(x_1,\ldots,x_r)$ where each $x_i\in \{A,B,C\}$ indicates which local model appears in $\w_i$.

\begin{definition}[Inductive equivalence]\label{dfn:inductive_equiv}
    Let $\beta\in \Br_n^+$. We say that two right simplifying weaves $\w$ and $\w'$ from $\beta$ to the same braid $\delta(\beta)$ are \emph{inductive equivalent}, denoted by $\w \approx \w'$, if and only if $p(\w) = p(\w')$.
\end{definition}

\begin{remark}
    Two inductive equivalent right simplifying weaves are represented by the same sequence of trivalent vertices and cups, but may contain different sequences of hexavalent and tetravalent vertices.
\end{remark}

 \begin{question}\label{qst:inductive_equiv_implies_weave_equiv}
 Given any two right simplifying weaves that are inductive equivalent, must they then be weave equivalent?
\end{question}

 Inductive weaves with at most one trivalent vertex are weave equivalent~\cite[Theorem 4.17]{CGGS1}. See~\cite[Section 4.1]{CGGS1} for a definition of weave equivalence.

\subsection{Braid varieties}\label{sec:braid_variety}

In this section we define the braid variety $X(\beta)$ following Casals--Gorsky--Gorsky--Le--Shen--Simental \cite{CGGS1,CGGLSS}. We follow the conventions in \cite[Section 2.1]{CGGS1} (those in \cite[Section 3.3]{CGGLSS} are different); an equivalent definition of braid varieties is given in \cite[Section 6.3]{GLSS}. We focus on $A$-type reductive groups: $G = \GL(n, \C)$ and $B$ is the Borel subgroup of upper triangular matrices. We implicitly fix an identification of the Weyl group $W = S_n$ with the group of permutation matrices in $G = \GL(n,\C)$.

For any $i \in \{1,\dots,n-1\}$, define the braid matrix $B_i(z) \in G$ by
\begin{equation}\label{eq:braid_matrix}
    B_i(z) \coloneqq I_{i-1} \oplus B(z) \oplus I_{n-i-1}, \quad \text{where }\, B(z) \coloneqq \begin{pmatrix}0 & 1 \\ 1 & z\end{pmatrix},
\end{equation}
meaning that $B_i(z)$ is the $n\times n$ identity matrix with the $2\times 2$ matrix $B(z)$ inserted as a block with its upper left corner at position $(i,i)$. These matrices are closely related to the Bruhat decomposition of the general linear group $G = \GL(n, \C)$; see \cref{sec:braid_sheaf}.

\begin{notation}\label{notn:braid_matrix_word}
    Let $\beta = \sigma_{i_1}\cdots \sigma_{i_r} \in \Br_n^+$, and $z = (z_1,\ldots,z_r) \in \C^r$. We use the notation
    \[
        B_{\beta}(z) \coloneqq B_{i_r}(z_{r}) \cdots B_{i_1}(z_{1}).
    \]
\end{notation}

\begin{definition}[Braid variety {\cite[Section 3.3]{CGGLSS}}]\label{dfn:braid_var}
    Let $\beta = \sigma_{i_1}\cdots\sigma_{i_r} \in \Br_n^+$ and $\pi \in S_n$. Then the \emph{braid variety of $\beta$ and $\pi$} is defined as
    \[
    X(\beta, \pi) \coloneqq \{(z_1, \ldots, z_r) \in \C^r \mid B_{\beta}(z)\pi \in B\}.
    \] 
    In particular, we write $X(\beta) \coloneqq X(\beta, \delta(\beta))$.
\end{definition}
\begin{remark}
    By \cite[Theorem 3.7]{CeballosLabbeStump} or \cite[Lemma 3.4]{CGGLSS}, we know that if $\delta(\beta \sigma_i) = \delta(\beta) s_i$, then $X(\beta \sigma_i) = X(\beta)$. Hence, we focus on $\beta$ such that $\delta(\beta)=w_0$.
\end{remark}

\subsection{Braid varieties are augmentation varieties}
We show that the augmentation variety of the $(-1)$-closure of $\beta\Delta$ is isomorphic to the braid variety of $\beta \in \Br_n^+$ with $\delta(\beta) = w_0$. This is done via an intermediary Lagrangian projection of $\beta\Delta$ called the pigtail closure denoted by $\La_{\mathrm{pig}}(\beta \Delta)$; see \cref{fig:pigtail_closure}.

\begin{figure}[!htb]
    \centering
    \includegraphics{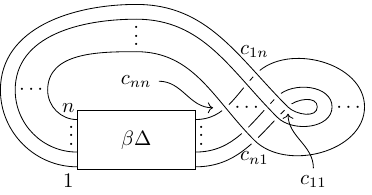}
    \caption{The Lagrangian projection of the Legendrian pigtail closure $\La_{\mathrm{pig}}(\beta\Delta)$.}
    \label{fig:pigtail_closure}
\end{figure}

\begin{proposition}\label{thm:augmentation_pig=braid variety}
    Let $\beta \in \Br_n^+$ be such that $\delta(\beta) = w_0$. There is an isomorphism 
    \[
    X(\beta) \cong \Aug(\Lambda_{\mathrm{pig}}(\beta\Delta))
    \]
    such that the braid matrix variables are mapped to the corresponding degree $0$ generators of $\A(\La_{\mathrm{pig}}(\beta\Delta))$.
\end{proposition}
\begin{proof}
     Let $\beta = \sigma_{i_1} \sigma_{i_2} \cdots \sigma_{i_r}$. Recall from \cref{notn:chekanov_eliashberg} that the Chekanov--Eliashberg dga of $\Lambda_{\mathrm{pig}}(\beta\Delta)$ is denoted by $\A(\Lambda_{\mathrm{pig}}(\beta\Delta))$. First, note that the degree 0 generators of $\A(\Lambda_{\mathrm{pig}}(\beta\Delta))$ correspond to the crossings of $\beta\Delta$, which we denote by $\{z_i\}_{1\leq i\leq r} \cup \{w_{ij}\}_{1\leq i<j \leq n}$. Second, note that the degree 1 generators of $\A(\Lambda_{\mathrm{pig}}(\beta\Delta))$ correspond to the crossings labeled by $\{c_{ij}\}_{1\leq i,j\leq n}$ in \cref{fig:pigtail_closure}. We place one marked point on each strand in between the braids $\beta$ and $\Delta$ and denote the corresponding generators of the Chekanov--Eliashberg dga $\{t_i\}_{1\leq i\leq n}$. Define $D(t) \coloneqq \diag(t_1,\dots,t_n)$. By \cite[Proposition 5.2]{CasalsNg}, the non-trivial differentials in the Chekanov--Eliashberg dga are given by 
     \[
     \partial C = I_n + B_\beta(z) D(t) B_\Delta(w),
     \]
     where $C = (c_{ij})_{i,j}$. Therefore, the augmentation variety of $\Lambda_{\mathrm{pig}}(\beta\Delta)$ is
     \[
     \V(\La_{\mathrm{pig}}(\beta\Delta)) = \{ (z, w, t) \in \C^r \times \C^{\binom n2} \times (\C^\ast)^n \mid I_n + B_\beta(z) D(t) B_\Delta(w) = 0 \},
     \]
     and since $\A(\Lambda_{\mathrm{pig}}(\beta\Delta))$ does not have any generators of negative degree, by \cref{rmk:aug_naive_equal}, $\V(\Lambda_{\mathrm{pig}}(\beta\Delta)) = \Aug(\Lambda_{\mathrm{pig}}(\beta\Delta))$. By \cite[Example 2.2]{CGGS1}, we can find algebraically independent polynomials of $\{w_{ij}\}_{1\leq i<j\leq n}$, which we denote by $\{u_{ji}\}_{1\leq i<j\leq n}$, such that
     \begin{align*}
         B_\Delta(w) = L(u) w_0 = \begin{pmatrix}
             1 & 0 & \dots & 0 \\ u_{21} & 1 & \dots & 0 \\
             \vdots & \vdots & \ddots & \vdots \\ u_{n1} & u_{n2} & \dots & 1
         \end{pmatrix}
         \begin{pmatrix}
             0 & \dots & 0 & 1 \\ 0 & \dots & 1 & 0 \\
             \vdots & \iddots & \vdots & \vdots \\ 1 & \dots & 0 & 0
         \end{pmatrix},
     \end{align*}
     where
     \[L(u)=\begin{pmatrix}
             1 & 0 & \dots & 0 \\ u_{21} & 1 & \dots & 0 \\
             \vdots & \vdots & \ddots & \vdots \\ u_{n1} & u_{n2} & \dots & 1
         \end{pmatrix}.\]
     Therefore, the augmentation variety $\Aug(\Lambda_{\mathrm{pig}}(\beta\Delta)) = \V(\Lambda_{\mathrm{pig}}(\beta\Delta))$ is defined by the equation
     \begin{align*}
         w_0B_\beta(z) = -L(u)^{-1} D(t)^{-1} = -\begin{pmatrix}
             1 & 0 & \dots & 0 \\ u_{21} & 1 & \dots & 0 \\
             \vdots & \vdots & \ddots & \vdots \\ u_{n1} & u_{n2} & \dots & 1
         \end{pmatrix}^{-1}
         \begin{pmatrix}
             t_1 & 0 & \dots & 0 \\ 0 & t_2 & \dots & 0 \\
             \vdots & \vdots & \ddots & \vdots \\ 0 & 0 & \dots & t_n
         \end{pmatrix}^{-1}
     \end{align*}
     and hence, in the augmentation variety, $w_0B_\beta(z)$ is always lower triangular. Equivalently, $B_\beta(z)w_0$ is always upper triangular.

     Conversely, for any $z \in \C^r$ such that $B_\beta(z)w_0$ is upper triangular, we know that the variables $\{u_{ji}\}_{1\leq i<j\leq n}$ and $\{t_i\}_{1\leq i\leq n}$ are uniquely determined. Thus, by \cite[Lemma 3.2]{CGGLSS}, the variables $\{w_{ij}\}_{1\leq i<j\leq n}$ are also uniquely determined. This shows that
     \[
     \Aug(\Lambda_{\mathrm{pig}}(\beta \Delta)) = \V(\Lambda_{\mathrm{pig}}(\beta\Delta)) = \{z \in \C^r \mid B_\beta(z)w_0 \text{ is upper triangular} \} \cong X(\beta),
     \]
     which is the sought isomorphism.
\end{proof}

\begin{lemma}\label{lem:isomorphism_naive_aug} 
    Let $\beta \in \Br_n^+$ be such that $\delta(\beta) = w_0$. There is a natural injection
    \[
    \V(\La_\mathrm{pig}(\beta\Delta))\hooklongrightarrow \V(\La(\beta\Delta)).
    \]
\end{lemma}
\begin{proof}
    By \cite[Figure 8]{CasalsNg}, there is a Legendrian isotopy from $\La_\mathrm{pig}(\beta\Delta)$ to $\La(\beta\Delta)$ inducing a dga map $\A(\La_\mathrm{pig}(\beta\Delta))\to \A(\La(\beta\Delta))$ that we now describe. First, we label the generators of $\A(\La_\mathrm{pig}(\beta\Delta))$ as in the proof of \cref{thm:augmentation_pig=braid variety}.
    
    Next, we label the generators of $\A(\La(\beta\Delta))$: the crossings of $\beta\Delta$ are the degree $0$ generators and are labeled $\{x_i\}_{1\leq i\leq r}\cup \{w_{ij}\}_{1\leq i<j\leq n}$; the crossings to the left of $\beta\Delta$ are degree $-1$ generators labeled $\{\alpha_{ij}\}_{1\leq i,j\leq n}$, where $\alpha_{ij}$ is the crossing between the $i$-th and $j$-th cusp if the cusps are labeled $1, \ldots, n$ from bottom to top; the crossings to the right of $\beta\Delta$ are degree $1$ generators labeled $\kappa_{ij}$, where $\kappa_{ij}$ is the crossing between the $i$-th and $j$-th cusp. Note that there are $\binom n2$ degree $-1$ generators and $\binom n2$ degree $1$ generators. The isotopy from $\La_\mathrm{pig}(\beta\Delta)$ to $\La(\beta\Delta)$ includes Reidemeister 2-moves that add and remove pairs of generators and Reidemeister 3-moves whose induced dga map is trivial. The induced dga map $\A(\La_\mathrm{pig}(\beta\Delta))\to \A(\La(\beta\Delta))$ is defined as follows on generators:
    \[
    z_i\longmapsto x_i, \quad
    w_{ij}\longmapsto w_{ij}, \quad
    c_{ij}\longmapsto \begin{cases}
        \kappa_{ij}, & i\geq j,\\
        0, & i < j.
    \end{cases}
    \]
    Following a similar argument to that of \cref{thm:augmentation_pig=braid variety}, we see that 
    \[
    \V(\Lambda(\beta\Delta)) = \{ (x, w, t) \in \C^r \times \C^{\binom n2} \times (\C^\ast)^n \mid I_n + B_\beta(z) D(t)  B_{\Delta}(w)\text{ is strictly lower triangular} \}.
    \]
    To verify this, we label the strands from bottom to top and note that by \cite{Kalman06} (see also, \cite[Section 5.1]{CasalsNg}) the entry $(i, j)$ of $B_\beta(z)$ counts paths beginning at the left of $\beta$ starting with the $i$-th strand and ending on the right of $\beta$ on the $j$-th strand. Consider the resulting matrix with entries $(i, j)$ for $1\leq i\leq j\leq n$. The entries that lie on or above the diagonal compute the differential of the degree $1$ generators of $\A(\La(\beta\Delta))$ and define $\V(\La(\beta\Delta))$.
    
    Again by \cite[Example 2.2]{CGGS1}, we can find algebraically independent polynomials of $\{w_{ij}\}_{1\leq i<j\leq n}$, which we denote by $\{u_{ji}\}_{1\leq i<j\leq n}$, such that
     \begin{align*}
         B_\Delta(w) = L(u) w_0 = \begin{pmatrix}
             1 & 0 & \dots & 0 \\ u_{21} & 1 & \dots & 0 \\
             \vdots & \vdots & \ddots & \vdots \\ u_{n1} & u_{n2} & \dots & 1
         \end{pmatrix}
         \begin{pmatrix}
             0 & \dots & 0 & 1 \\ 0 & \dots & 1 & 0 \\
             \vdots & \iddots & \vdots & \vdots \\ 1 & \dots & 0 & 0
         \end{pmatrix}.
     \end{align*}
    Therefore the variety $\V(\Lambda(\beta \Delta))$ is defined by the equation
    \[
    B_\beta(z)=L'w_0^{-1}L(u)^{-1}D(t)^{-1},
    \]
    where $L'$ is a lower unitriangular matrix given by subtracting $I_n$ from the strictly lower triangular matrix given in the equation defining $\V(\La(\beta\Delta))$. Conjugating $w_0^{-1}$ past $L'$, we obtain an upper unitriangular matrix $U'$ and we conclude that  $\V(\La(\beta\Delta))$ is cut out by the equation $w_0B_\beta(z)=U'L(u)^{-1}D(t)^{-1}$.
    Equivalently, $\V(\La(\beta\Delta))$ is defined by the condition that $B_\beta(z)w_0$ admits an $LU$ decomposition. 
\end{proof}

\begin{theorem}\label{thm:-1=braid vty_new}
    Let $\beta\in \Br_n^+$ such that $\delta(\beta) = w_0$. There is an isomorphism $\Aug(\La(\beta\Delta))\cong X(\beta)$. Moreover, if $\beta = \Delta\gamma$ for some $\gamma \in \Br_n^+$, the isomorphism is the identity map on the degree $0$ generators corresponding to $\gamma$ in the Chekanov--Eliashberg dgas.
\end{theorem}
\begin{proof}
    By \cref{thm:augmentation_pig=braid variety}, it suffices to prove that there is an somorphism $\Aug(\La(\beta\Delta)) \cong \Aug(\La_{\mathrm{pig}}(\La(\beta\Delta))$ with the desired properties.

    By \cref{lem:isomorphism_naive_aug} there is a natural injection $\V(\Lambda_{\mathrm{pig}}(\beta\Delta)) \hookrightarrow \V(\La(\beta\Delta))$. We have a commutative diagram of algebraic varieties
    \[
    \begin{tikzcd}
        \V(\Lambda_{\mathrm{pig}}(\beta\Delta)) \rar[equal] \ar[d, hook] & \Aug(\La_{\mathrm{pig}}(\beta\Delta)) \dar{\cong} \\
    \V(\Lambda(\beta\Delta)) \ar[r, two heads] & \Aug(\La(\beta\Delta))
    \end{tikzcd}.
    \]
    This implies that the coordinate functions $z_i$, for $ {n \choose 2} + 1\leq i \leq r$ on $\V(\Lambda(\beta\Delta))$ descend to the augmentation variety $\Aug(\La(\beta\Delta))$.
\end{proof}

\subsection{Braid varieties, flags and sheaves}\label{sec:braid_sheaf}
We give an alternative description of braid varieties in terms of flags using the moduli space of microlocal rank 1 sheaves, another invariant of Legendrians \cite{STZ_ConstrSheaves}. There is an equivalence between (the categories of) microlocal rank 1 sheaves and augmentations of Legendrian knots~\cite{NRSSZ}.

A complete flag $F$ is defined as a sequence of vector spaces $0 \subset F_1 \subset F_2 \subset \dots \subset F_{n-1} \subset \C^n$ where $\dim F_i = i$. For two complete flags $F(x)$ and $F(y)$, we say that they are in relative position $s_i \in S_n$ if $F_j(x) = F_j(y)$ for $j \neq i+1$, and if $F_{i+1}(x)$ and $F_{i+1}(y)$ are transverse in $F_{i+2}(x) = F_{i+2}(y)$. If $F(x)$ and $F(y)$ are in relative position $s_i \in S_n$ we use the notation $F(x) \xrightarrow{s_i} F(y)$. Complete flags are parametrized by the flag variety $G/B$, and we fix an identification between $xB \in G/B$ and complete flags for $x\in G$. Using the Bruhat decomposition $G = \bigsqcup_{w \in S_n} BwB$, we alternatively define two flags $xB$ and $yB$ to be in relative position $w \in S_n$ if $x^{-1}y \in BwB$; in this case we again use the notation $xB \xrightarrow{s_i} yB$. See e.g.\@ the survey \cite{Speyer23} for more details on the Bruhat decomposition and flags.

The proof of the following proposition is the same as the proof of \cite[Corollary 3.7]{CGGLSS} with the only exception that we use different braid matrices compared to \cite{CGGLSS}.

\begin{proposition}\label{prop:braid_variety}
    Let $\beta = \sigma_{i_1} \cdots \sigma_{i_r} \in \Br_n^+$ and $\pi \in S_n$. The braid variety of $\beta$ and $\pi$ is defined as
    \[
    X(\beta, \pi) \coloneqq \{(F_1, F_2, \dots, F_{r+1}) \in (G/B)^{r+1} \mid F_1 = B,\; F_k \xrightarrow{s_{i_k}} F_{k+1}, \; F_{r+1} = \pi B\}.
    \]
    We write $X(\beta) \coloneqq X(\beta, \delta(\beta))$.
\end{proposition}

Motivated by the result that augmentations are sheaves \cite{NRSSZ}, the isomorphism in \cref{thm:-1=braid vty_new} should lead to an isomorphism between the braid variety and a certain moduli space of microlocal rank $1$ sheaves, namely, microlocal rank $1$ sheaves for the $(-1)$-closure $\Lambda(\beta\Delta)) \subset \R^3$ satellited over the standard unknot or the cylindrical closure $\Lambda_{\mathrm{cyl}}(\Lambda(\beta\Delta))$ satellited over the outward conormal $\Lambda_{S^1} \subset S^*\R^2$ of the unit circle $S^1 \subset \R^2$.

\begin{notation}\label{notn:sheaves}
    Let $D$ be a front diagram of a Legendrian link $\varLambda \subset \R^3$ with a set of marked points $T \subset D$. Let $\mathcal{M}_1^\textit{fr}(D, T)$ denote the framed moduli space of microlocal rank 1 sheaves on $\R^2$ with singular support on $D$. We sometimes also use the notation $\mathcal M_1^\textit{fr}(\varLambda,T)$.
\end{notation}

The following result is essentially contained in \cite[Corollaries 6.3 and 6.6]{CasalsLi} or \cite[Lemma 4.3 and Proposition 4.4]{CasalsWeng}, but we provide a direct proof here for exposition.

\begin{proposition}\label{prop:sheaves_braid_var}
    Let $\beta \in \Br_n^+$ such that $\delta(\beta) = w_0$. Then there are isomorphisms
    \[
    X(\beta) \cong \mathcal{M}_1^\textit{fr}(\Lambda_\mathrm{cyl}(\beta\Delta), T) \cong \mathcal{M}_1^\textit{fr}(\Lambda(\beta\Delta), T),
    \]
    where $\mathcal{M}_1^\textit{fr}(\Lambda(\beta\Delta), T)$ is the framed moduli space of microlocal rank 1 sheaves on $\R^2$ with singular support on $\Lambda(\beta\Delta)$ with one base point per strand on the braid.
\end{proposition}
\begin{proof}
    Write $\beta = \sigma_{i_1} \dots \sigma_{i_{r}}$ and $\Delta = \sigma_{i_{r+1}} \dots \sigma_{i_{r+n(n-1)/2}}$ and $N \coloneqq r+1+n(n-1)/2$. Then, by \cite[Proposition 1.5]{STZ_ConstrSheaves}, \cite[Corollary 6.3]{CasalsLi}, or \cite[Lemma 4.3]{CasalsWeng}, we know that the moduli space of microlocal rank 1 sheaves on $\R^2$ with singular support on $\Lambda(\beta\Delta)$ is given by
    \[
    \mathcal{M}_1(\Lambda_\mathrm{cyl}(\beta\Delta)) = \{(F_1, \dots, F_N) \in (G/B)^N \mid F_k \xrightarrow{s_{i_k}} F_{k+1}, \; F_1 = F_N\}/G.
    \]
    By the argument in \cite[Proposition 6.5]{CasalsLi}, we know that the moduli space of the $(-1)$-closure is isomorphic to the moduli space of the cylindrical closure, where the flag on the top $n$-strands in the $(-1)$-closure is in relative position $w_0$ with the flag $F_1 = F_N$
    \[
    \mathcal{M}_1(\Lambda_\mathrm{cyl}(\beta\Delta)) \cong \mathcal{M}_1(\Lambda(\beta\Delta)).
    \]
    Consider Reidemeister 2-moves at the right cusps and the $\Delta$ crossings as in \cref{fig:r2_normal_rulings} which induce equivalences of the moduli spaces of sheaves as in \cite[Proposition 6.5]{CasalsLi}. The flag on the top $n$-strands in the $(-1)$-closure is now the opposite flag of $F_{r+1}$. Thus, we have that 
    \[
    \mathcal{M}_1(\Lambda(\beta\Delta)) = \{(F_1, \dots, F_{r+1}) \in (G/B)^{r+1} \mid F_k \xrightarrow{s_{i_k}} F_{k+1},\; F_1 = w_0F_{r+1}\}/G.
    \]
    Finally, given the framing by one base points per strand on the braid, we may assume that $F_1 = B$. Therefore we have
    \[
    \mathcal{M}_1^\textit{fr}(\Lambda_\mathrm{cyl}(\beta\Delta), T) \cong \mathcal{M}_1^\textit{fr}(\Lambda(\beta\Delta), T) = X(\beta).
    \]
\end{proof}

\section{Decompositions of braid and augmentation varieties}\label{sec:decomp_braid_varieties_mcs}

In this section we discuss various decompositions of braid varieties and augmentation varieties. In \cref{sec:weave_decomp} we describe the weave decomposition of the braid variety introduced by Casals--Gorsky--Gorsky--Simental \cite{CGGS1}. Braid varieties generalize open Richardson varieties, and in \cref{sec:deodhar_decomp} we show that the weave decomposition of braid varieties generalizes the well-known Deodhar decomposition of open Richardson varieties. After reviewing the definition of a Morse complex sequence in \cref{sec:mcs} we describe the ruling decomposition of the augmentation variety in \cref{sec:ruling_decomp} introduced by Henry--Rutherford \cite{HenryRutherford15}.

\subsection{Weave decomposition}\label{sec:weave_decomp}
    In this section we follow \cite[Section 5]{CGGS1} and associate to a weave $\w \colon \beta_1 \to \beta_2$ an algebraic variety $\mathfrak X(\w)$ which is used to construct a so-called ``weave decomposition" of the braid variety associated to $\beta_1$. We first address the case of (unframed) algebraic weaves and then address framed algebraic weaves.

    \subsubsection{Algebraic weaves}

    \begin{definition}[Algebraic weave {\cite[Definition 5.1]{CGGS1}}]\label{dfn:algebraic_weave}
        An \emph{algebraic weave} is a sliced weave $\w \colon \beta_1 \to \beta_2$ (see \cref{dfn:sliced_weave}) that has been decorated with dashed rays at every trivalent vertex, cup, and cap, that go to the right boundary of the weave; see \cref{fig:alg_weave1,fig:alg_weave2,fig:alg_weave3}. A dashed ray at a point $(b,r)$ is the subset $(b,1) \times \{r\} \subset (0,1) \times [0,1]$.

        A transverse intersection between a dashed ray and the weave is called a \emph{virtual vertex}.

        In an algebraic weave, a \emph{weave segment} is an edge of the weave that is not part of a dashed ray. A \emph{dashed segment} is an edge that is an edge of the weave that belongs to a dashed ray. Equip an algebraic $n$-weave $\w \colon \beta_1 \to \beta_2$  with the following:
    \begin{itemize}
        \item For each $i$-colored weave segment, a variable $z$. 
        \item For each dashed segment, an $n\times n$ invertible upper triangular matrix $V$.
    \end{itemize}
    Define
    \[
        \mathbb V^{\w} \coloneqq \C^{\text{weave segments}} \times \left(\mathbb C^{\binom{n}2} \times \left(\C^\ast\right)^n\right)^{\text{dashed segments}},
    \]
    where we identify $\mathbb C^{\binom n2} \times \C^\ast$ with the set of $n\times n$ invertible upper triangular matrices.
    \end{definition}
    \begin{remark}\label{rmk:dashed_rays_right}
        Our dashed rays go \emph{to the right}, which is the opposite of \cite[Definition 5.1]{CGGS1}, but the same as in \cite[Section 5.1.1]{CGGLSS}.
    \end{remark}
    \begin{figure}[!htb]
        \centering
        \begin{subfigure}{0.3\textwidth}
            \centering
            \includegraphics{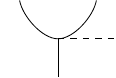}
            \caption{}\label{fig:alg_weave1}
        \end{subfigure}
        \begin{subfigure}{0.3\textwidth}
            \centering
            \includegraphics{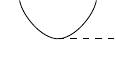}
            \caption{}\label{fig:alg_weave2}
        \end{subfigure}
        \begin{subfigure}{0.3\textwidth}
            \centering
            \includegraphics{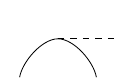}
            \caption{}\label{fig:alg_weave3}
        \end{subfigure}
        \caption{Dashed rays at a trivalent vertex, cup, and cap, respectively.}
    \end{figure}
     See \cref{fig:intro_weave_mcs,{fig:cluster_ex}} for two examples of algebraic weaves. The algebraic weave in \cref{fig:intro_weave_mcs} has $15$ weave segments and $3$ dashed segments, while the algebraic weave in \cref{fig:cluster_ex} has $5$ weave segments and $2$ dashed segments.

    \begin{definition}[Monodromy]\label{dfn:monodromy}
        Let $\w \colon \beta_1 \to \beta_2$ be an algebraic $n$-weave equipped with variables and invertible upper triangular matrices matrices associated to weave segments and dashed segments, respectively. Let $\tau \colon S^1 \to (0,1) \times [0,1]$ be an embedding that only has transverse intersections with $\w$ and its dashed rays.

        The \emph{monodromy} of $\w$ along $\tau$ is defined as the product of the following matrices, according to the orientation of $\tau$:
        \begin{enumerate}
            \item $B_i(z)$ if $\tau$ crosses an $i$-colored weave segment with variable $z$ \emph{from right to left}.
            \item $B_i(z)^{-1}$ if $\tau$ crosses an $i$-colored weave segment with variable $z$ \emph{from left to right}.
            \item $V$ if $\tau$ crosses a dashed segment with matrix $V$ \emph{from top to bottom}.
            \item $V^{-1}$ if $\tau$ crosses a dashed segment with matrix $V$ \emph{from bottom to top}.
        \end{enumerate}
    \end{definition}
    \begin{remark}\label{rmk:monodromy_diff}
        Items (1) and (2) in \cref{dfn:monodromy} differ from \cite[Definition 5.3(i) and (ii)]{CGGS1} to reflect the fact that our dashed rays go to the right instead of to the left; see \cref{rmk:dashed_rays_right}.
    \end{remark}
    
    \begin{notation}By construction of the weave $\w \colon \beta_1 \to \beta_2$ there is a single weave segment associated with each letter of the braid word $\beta_2$. If $\beta_2 = \sigma_{j_1} \cdots \sigma_{j_k}$, we let $z_{j_m}$ denote the variable corresponding to the $i_m$-colored weave segment corresponding to the letter $\sigma_{j_m}$.
    \end{notation}
     As in \cref{notn:braid_matrix_word} we use the following notation
    \begin{equation}\label{eq:braid_comb_matrix}
        B_{\beta_2}(z_1,\ldots,z_m) \coloneqq B_{i_m}(z_{m}) \cdots B_{i_1}(z_{1}),
    \end{equation}
    and note that it is equal to the monodromy of a path from right to left intersecting every weave segment at the bottom of the weave.
    \begin{definition}\label{dfn:weave_corr}
        Let $\w \colon \beta_1 \to \beta_2$ be an algebraic $n$-weave with variables and invertible upper triangular matrices associated to weave segments and dashed segments, respectively, as described above. Let $\pi \in S_n$ be represented by a permutation matrix in $\GL(n,\C)$. We define $\mathfrak X(\w,\pi)$ to be the closed subvariety of $\mathbb V^{\w}$ that is cut out by the following conditions:
        \begin{itemize}
            \item The monodromy around each closed loop around each vertex and virtual vertex of $\w$ is the identity matrix.
            \item The matrix $B_{\beta_2}(z_1,\ldots,z_m)\pi$ is upper triangular.
        \end{itemize}
        We use the notation $\mathfrak{X}(\w) \coloneqq \mathfrak{X}(\w,w_0)$.
    \end{definition}
    \begin{remark}
        Although \cref{dfn:monodromy} differs from \cite[Definition 5.3]{CGGS1} (see \cref{rmk:monodromy_diff}), our definition of $\mathfrak{X}(\w)$ in \cref{dfn:weave_corr} coincides with $\mathfrak M(\w)$ in \cite[Definition 5.4]{CGGS1}.
    \end{remark}
    \begin{remark}
        The $z$-variables on all the weave segments are uniquely determined by the $z$-variables at the top of the weave using the trivial monodromy condition. This immediately implies that we have an injection $\mathfrak{X}(\w) \hookrightarrow X(\beta_1)$. 
    \end{remark}

    The monodromy condition for algebraic weaves determines the relation between the $z$-variables of each weave segment in each layer as follows (cf.\@ \cite[pp.\@ 1153--1554]{CGGS1}):

\begin{lemma}\label{lma:weave_monodromy_vertex}
    For an algebraic weave with variables and invertible upper triangular matrices associated to weave segments and dashed segments, respectively, the monodromy around vertices and virtual vertices holds if and only if the following conditions hold:
\begin{enumerate}
    \item \label{cond:hexavalent} \textbf{At a hexavalent vertex colored by $i$ and $i+1$:} We have 
    \[
    B_i(z_3)B_{i+1}(z_2)B_i(z_1)=B_{i+1}(z_1)B_{i}(z_2- z_1z_3)B_{i+1}(z_3).
    \]
    \item \label{cond:distant} \textbf{At a distant crossing colored by $i$ and $k$:} We have
    \[
    B_i(z)B_k(w)=B_{k}(w)B_{i}(z), \qquad |i - k|\geq 2.
    \]
    \item \label{cond:trivalent}\textbf{At a trivalent vertex colored by $i$:} We have $z_2\neq 0$ and  
    \[
    B_i(z_2)B_{i}(z_1)=VB_i(z_1+z_2^{-1}), \qquad V=\begin{pmatrix}-z_2^{-1} & 1\\ 0 & z_2\end{pmatrix}.
    \]
    \item \label{cond:cup} \textbf{At a cup colored by $i$:} We have $z_2 = 0$ and
    \[
    B_i(0) B_i(z_1) = V, \qquad V = \begin{pmatrix}1 & z_1\\ 0 & 1\end{pmatrix}.
    \]
    \item \label{cond:virtual} \textbf{At a virtual vertex colored by $i$:} We have
    \[
    B_i(z)V = V'B_i(w), \qquad w = a_{ii}^{-1}(a_{i,i+1}+za_{i+1,i+1}), \; V = (a_{k,\ell})_{k \leq \ell},
    \]
    where $V'$ is completely determined by $V$ and $z$.
    \qed
\end{enumerate}
\end{lemma}

Whenever there is a trivalent vertex or a cup, by condition \eqref{cond:trivalent} and \eqref{cond:cup}, the trivial monodromy condition introduces an upper triangular matrix corresponding to the dashed ray. Then, we can push the upper triangular matrix all the way to the right following the dashed ray using condition \eqref{cond:virtual}.

\begin{lemma}\label{lma:tri_hexa_corresp}
    Let $\w_{\mathrm{tri}} \colon \sigma_1^2 \to \sigma_1$, $\w_{\mathrm{hexa}} \colon \sigma_1\sigma_2\sigma_1 \to \sigma_2\sigma_1\sigma_2$, and $\w_{\mathrm{cup}} \colon \sigma_1^2 \to 1$ be the algebraic weaves consisting of a single trivalent vertex, a single hexavalent vertex, and a single cup, respectively. Then
    \begin{align*}
        \mathfrak X(\w_{\mathrm{tri}}) &= \left\{(z_1,z_2,z_3,1,-z_2^{-1},z_2) \in \C^3 \times (\C \times (\C^\ast)^2) \;\middle|\; z_3 = z_1+z_2^{-1}, \, z_3 = 0\right\},\\
        \mathfrak X(\w_{\mathrm{hexa}}) &= \left\{(z_1,z_2,z_3,z_3,z_2-z_1z_3,z_1) \in \C^6 \;\middle|\; z_1 = z_2 = z_3 = 0 \right\},\\
        \mathfrak X(\w_{\mathrm{cup}}) &= \left\{(z_1,z_2,z_1,1,1) \in \C^2 \times (\C \times (\C^\ast)^2) \;\middle|\; z_2 = 0\right\}.
    \end{align*}
    \qed
\end{lemma}

\begin{figure}[!htb]
    \centering
    \includegraphics{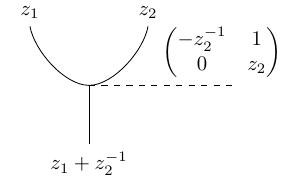}
    \hspace{5mm}
    \includegraphics{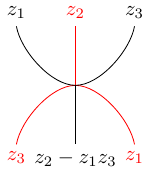}
    \hspace{5mm}
    \includegraphics{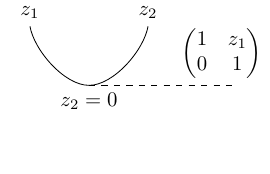}
    \caption{Variables associated to weave segments near a trivalent vertex, hexavalent vertex, and a cup, respectively.}
    \label{fig:tri_hexa}
\end{figure}

Let us summarize the constructions in this section as the construction of a functor.
\begin{definition}[Weave category]\label{dfn:weave_category}
    Let $n\in \Z_{\geq 1}$. Let $\mathfrak{W}_n$ be the category defined as follows:
    \begin{description}
        \item[Objects] $\Ob(\mathfrak{W}_n) = \Br_n^+$.
        \item[Morphisms] A morphism $\beta_1 \to \beta_2$ is an algebraic weave from $\beta_1$ to $\beta_2$.
        \item[Composition] The composition is defined as vertical composition of weaves (see \cref{fig:vertical_comp}).
    \end{description}
\end{definition}
Let $\mathfrak{C}$ denote the category whose objects are algebraic varieties, and a morphism $X \to Y$ is a correspondence, i.e., an algebraic variety $Z$ and algebraic maps $X \leftarrow Z \to Y$. Based on \cref{lma:tri_hexa_corresp}, we have the following theorem \cite[Theorem 1.3]{CGGS1}; see also \cite[Theorem 5.9]{CGGS1}.
\begin{theorem}[{\cite[Theorem 1.3]{CGGS1}}]\label{thm:functor_w_to_c}
    For any $n\in \mathbb Z_{\geq 1}$, there exists a functor
    \[
    \mathfrak X \colon \mathfrak W_n \longrightarrow \mathfrak C,
    \]
    that is defined on objects by $\mathfrak X(\beta) \coloneqq X(\beta)$ and on morphisms $\w \colon \beta_1 \to \beta_2$ by $\mathfrak X(\w)$.
    Moreover, if $\w \colon \beta_1 \to \beta_2$ is a simplifying algebraic weave with $c$ cups and $t$ trivalent vertices, then there is an isomorphism $\mathfrak{X}(\w) \cong \C^c \times (\C^\ast)^t \times X(\beta_1)$ which defines an injective composition
    \[
    \phi_\w \colon \C^c \times (\C^\ast)^t \times X(\beta_1) \overset{\cong}{\longleftarrow} \mathfrak X(\w) \hooklongrightarrow X(\beta_2). \eqno\qed
    \]
\end{theorem}

The following theorem describes the weave decomposition of the braid variety and is completely analogous to \cite[Theorem 5.35]{CGGS1}.

\begin{theorem}[Weave decomposition]\label{thm:weave_decomp}
    Let $\beta \in \Br_n^+$ and let $\w \colon \beta \to \Delta$ be a Demazure weave. There exists a tuple of simplifying weaves $(\w = \w_1,\w_2,\dots,\w_k)$ such that
    \begin{itemize}
        \item $X(\beta) = \bigcup_{i=1}^k \mathfrak X({\w_i})$,
        \item $\mathfrak X(\w_i) \subset X(\beta)$ are pairwise disjoint for $i\in \{1,\dots,k\}$, and
        \item $\mathfrak X(\w) \cong (\C^\ast)^{\ell(\beta)-\binom n2}$ is the unique piece of maximal dimension.
    \end{itemize}
    Call such a tuple $(\w_1,\dots,\w_k)$ a \emph{decomposing tuple} for $X(\beta)$.
    \qed
\end{theorem}
\begin{lemma}\label{lma:decomp_tuple_indep_equiv}
    Suppose $\beta \in \Br_n^+$ and that $(\w_1,\dots,\w_k)$ is a decomposing tuple for $X(\beta)$. If $\w_i \approx \w_i'$ (\cref{dfn:inductive_equiv}), then $(\w_1',\ldots,\w_k')$ is a decomposing tuple, and the two decompositions of $X(\beta)$ are in fact equal.
\end{lemma}
\begin{proof}
    By \cref{lma:tri_hexa_corresp}, the weave correspondence induced by hexavalent and tetravalent vertices induce isomorphisms of the braid varieties. Therefore, by the proof of \cite[Theorem 5.35]{CGGS1}, it is clear that each piece in the decomposition of $X(\beta)$ does not depend on the hexavalent and tetravalent vertices used in the definition of a right simplifying weave, but only the sequences of trivalent vertices and cups.
\end{proof}

\subsubsection{Framed algebraic weaves}\label{ssec:framed-weave}
Following \cite[Section 5]{CGGLSS}, we now consider the algebraic variety of framed weaves and framed flags. While unframed weaves and flags define the same moduli space and induce the same weave decomposition, framed weaves and flags have the advantage of providing the cluster structure on the varieties.

Given a simplifying algebraic weave $\w\colon \beta_1 \to \beta_2$ with $t$ trivalent and $c$ cups, by \cref{thm:functor_w_to_c}, we have an algebraic isomorphism
\[
\mathfrak{X}(\w) \cong \C^c \times (\C^\ast)^t \times X(\beta_1).
\]
In particular, for each Demazure weave $\w \colon \beta \to \Delta$ with $t$ trivalent vertices, we have an embedding of an algebraic torus
\[
(\C^\ast)^{t} \hooklongrightarrow X(\beta).
\]
Given a Demazure weave $\w \colon \beta \to \Delta$ one can define a quiver $Q_\w$ whose vertices are trivalent vertices of the weave and whose exchange matrix $\epsilon_\w$ is a certain intersection matrix (it is the intersection matrix of the Lusztig cycles associated to trivalent vertices \cref{dfn:lusztig_cycle}); see \cite[Section 4.6]{CGGLSS}. Let $\A(\varepsilon_\w)$ be the cluster algebra defined by the quiver with exchange matrix $\epsilon_\w$. Then by \cite[Theorem 1.1]{CGGLSS} there is an isomorphism
\[
\C[X(\beta)] \cong \A(\varepsilon_\w),
\]
and the monomials coming from the open toric chart $(\C^\ast)^{t} \rightarrow X(\beta)$ are the cluster coordinates associated to the initial seed determined by $Q_\w$ (in particular, they are regular on $X(\beta)$).

\begin{notation}\label{not:xi(u)}
    Let $a,b\in \C^\ast$. For $i < k$, define
    \[
    \chi_{i,k}(a,b) \coloneqq \diag(1,\ldots,1,a,1,\ldots,1,b,1,\ldots,1),
    \]
    where the elements $a$ and $b$ are the $i$-th and $k$-th diagonal elements, respectively. For $u\in \C^\ast$, define $\chi_i(u) \coloneqq \chi_{i,i+1}(-u^{-1},u)$. Thus, $u$ and $u^{-1}$ are the $i$-th and $(i+1)$-st diagonal elements, respectively.
\end{notation}

The following identities establish an isomorphism between $X(\beta)$ and a framed version of the braid variety $X_\textit{fr}(\beta)$ which is defined below in \cref{lem:framed-flag=flag}.

\begin{lemma}\label{lma:move_marked_points}
    The following identities hold:
    \begin{enumerate}
        \item $B_i(z)\chi_{i,i+1}(a,b) = \chi_{i,i+1}(a,b)B_i(ab^{-1}z)$,
        \item $B_i(z)\chi_{i,k}(a,b) = \chi_{i+1,k}(a,b)B_i(a^{-1}z)$ for $i < k+1$,
        \item $B_i(z)\chi_{i+1,k}(a,b) = \chi_{i,k}(a,b)B_i(az)$ for $i < k+1$,
        \item $B_k(z)\chi_{i,k}(a,b) = \chi_{i,k+1}(a,b)B_k(b^{-1}z)$ for $i < k+1$,
        \item $B_k(z)\chi_{i,k+1}(a,b) = \chi_{i,k}(a,b)B_k(bz)$ for $i < k+1$,
        \item $B_\ell(z)\chi_{i,k}(a,b) = \chi_{i,k}(a,b)B_\ell(z)$ for $i,k \not\in \{\ell,\ell+1\}$.\qed
    \end{enumerate}
\end{lemma}

\begin{lemma}[{\cite[Lemma 3.13]{CGGLSS}}]\label{lem:framed-flag=flag}
    Let $\beta = \sigma_{i_1}\dots\sigma_{i_r} \in \Br_n^+$. For any fixed collection $(u_1, \dots, u_r) \in (\C^\ast)^r$, there is an isomorphism
    \[X_\textit{fr}(\beta) \coloneqq \{({z}_1, \dots, {z}_r) \in \C^r \mid  \chi_{i_r}(u_r) B_{i_r}({z}_r) \cdots \chi_{i_1}(u_1)B_{i_1}({z}_1) \delta(\beta) \in B\} \cong X(\beta),
    \]
    where we identify the Weyl group $W = S_n$ with the permutation matrices in $G = \GL(n, \C)$.
    \qed
\end{lemma}

\begin{definition}[Framed algebraic weave]
    A \emph{framed algebraic weave} is an algebraic weave $\w \colon \beta_1 \to \beta_2$ together with the following:
    \begin{itemize}
        \item For each $i$-colored weave segment, a pair of variables $({z},u) \in \C \times \C^\ast$. The first entry is called a \emph{${z}$-variable}, and the second entry is called a \emph{$u$-variable}. 
        \item For each dashed segment, an $n\times n$ unipotent matrix $V$.
\end{itemize}
\end{definition}
Similarly to \cref{sec:weave_decomp}, we define
\[
    \mathbb V^{\w}_\textit{fr} \coloneqq (\C \times \C^\ast)^{\text{weave segments}} \times \left(\mathbb C^{\binom{n}2}\right)^{\text{dashed segments}},
\]
where $\mathbb C^{\binom{n}2}$ is identified with the set of $n \times n$ unipotent matrices.

\begin{definition}[Framed monodromy]\label{dfn:framed_monodromy}
    Let $\w \colon \beta_1 \to \beta_2$ be a framed algebraic weave. Let $\tau \colon S^1 \to (0,1) \times [0,1]$ be an embedding such that all intersections with $\w$ and its dashed rays are transverse.
    
    The \emph{framed monodromy} of $\w$ along $\tau$ is defined as the product of the following matrices, according to the orientation of $\tau$:
    \begin{enumerate}
        \item $\chi_i(u)B_i({z})$ if $\tau$ crosses an $i$-colored weave segment \emph{from right to left}.
        \item $B_i({z})^{-1}\chi_i(u)^{-1}$ if $\tau$ crosses an $i$-colored weave segment \emph{from left to right}.
        \item $V$ if $\tau$ crosses a dashed segment with matrix $V$ \emph{from top to bottom}.
        \item $V^{-1}$ if $\tau$ crosses a dashed segment with matrix $V$ \emph{from bottom to top}.
    \end{enumerate}
\end{definition}

\begin{definition}\label{dfn:framed_weave_corr}
    Let $\w \colon \beta_1 \to \beta_2$ be a framed algebraic weave. Let $\pi \in S_n$ be represented by a permutation matrix in $\GL(n,\C)$. We define $\mathfrak X_\textit{fr}(\w,\pi)$ to be the closed subvariety of $\mathbb V^{\w}_\textit{fr}$ that is cut out by the following conditions:
    \begin{itemize}
        \item The framed monodromy around each closed loop around each vertex and virtual vertex of $\w$ is the identity matrix.
        \item The matrix $B_{\beta_2}({z}_1, \ldots, {z}_r)\pi$ is upper triangular, where ${z}_1,\ldots,{z}_r$ are the ${z}$-variables associated to the weave segments at the bottom of the weave after setting all the $u$-variables associated with the weave segments at the top of the weave equal to $1$.
    \end{itemize}
    We use the notation $\mathfrak{X}_\textit{fr}(\w) \coloneqq \mathfrak{X}_\textit{fr}(\w,w_0)$.
\end{definition}
\begin{remark}
    The $z$- and $u$-variables on all the weave and dashed segments are uniquely determined by the $z$-variables at the top of the weave using the framed monodromy condition. This immediately implies that there exists an injective map $\mathfrak{X}_\textit{fr}(\w) \hookrightarrow X(\beta_1)$. 
\end{remark}

Similarly to the monodromy condition of algebraic weaves, the framed monodromy condition for framed algebraic weaves determines the relation between the $z$-variables assigned to each weave segment in each layer.

\begin{lemma}\label{lem:framed-weave=weave}
    Let $\w \colon \beta_1 \to \beta_2$ be a framed algebraic weave. Then there is an algebraic isomorphism $\mathfrak{X}_\textit{fr}(\w) \cong \mathfrak{X}(\w)$ that fits into the following commutative diagram
    \[\begin{tikzcd}
        X_\textit{fr}(\beta_2) \ar[d,"\cong"] & \mathfrak{X}_\textit{fr}(\w) \ar[d, "\cong"] \ar[l,"\pi" above] \ar[r] & X(\beta_1) \ar[d, "="] \\
        X(\beta_2) & \mathfrak{X}(\w) \ar[l,"\pi" above] \ar[r] & X(\beta_1).
    \end{tikzcd}\]
\end{lemma}
\begin{proof}
    The proof is similar to \cref{lem:framed-flag=flag} which is analogous to the proof of \cite[Lemma 3.13]{CGGLSS}. Using the identities in \cref{lma:move_marked_points} shows that any product of the form
    \[
    \chi_{i_r}(u_r)B_{i_r}({z}_r) \cdots \chi_{i_1}(u_1)B_{i_1}({z}_1)
    \]
    is equal to a product of the form
    \[
    \diag(f_1,\ldots,f_n)B_{i_r}(g_r{z}_r) \cdots B_{i_1}(g_1{z}_1),
    \]
    where $f_i$ and $g_i$ are monomials in $\C[u_1^{\pm1},\ldots,u_r^{\pm1}]$. Comparing the framed and unframed monodromy conditions defines an isomorphism $\mathfrak{X}_\textit{fr}(\w) \cong \mathfrak{X}(\w)$ which, on each layer of the weave, restricts to an isomorphism $X_\textit{fr}(\beta_2) \cong X(\beta_2)$.
\end{proof}

\subsection{Cluster variables}\label{sec:cluster_variables}
It is shown in Casals--Gorsky--Gorsky--Le--Shen--Simental \cite{CGGLSS} and Galashin--Lam--Sherman-Bennett--Speyer \cite{GLSS} independently that the braid variety $X(\beta)$ is a cluster $\A$-variety (equivalently, a cluster $K_2$-variety); in other words, that the coordinate ring $\C[X(\beta)]$ is a cluster algebra. (Recently, it was shown that the two cluster structures agree \cite{CGGSBS}.) In this section, we give a description of how the cluster algebra structure (in particular, the cluster variables) is related to the weave decomposition of the variety. First, we define Lusztig cycles in a Demazure weave as in \cite[Section 4]{CGGLSS}.

\begin{definition}[Lusztig cycle {\cite[Definitions 4.8 and 4.12]{CGGLSS}}]\label{dfn:lusztig_cycle}
    Let $\w$ be a Demazure weave. For each trivalent vertex $v \in V(\w)$, the \emph{Lusztig cycle} $\gamma_v$ associated to $v$ is a weighted collection of edges, encoded by a map $\gamma_v \colon E(\w) \to \mathbb{N}$ such that the following holds.
    \begin{itemize}
        \item For the edges $e$ above or on the same layer of the trivalent vertex $v \in V(\w)$, we have $\gamma_v(e) = 0$ except for the outgoing edge $e_{vo}$ for the trivalent vertex where $\gamma_v(e_{vo}) = 1$.
         \item For any trivalent vertex $v' \in V(\w)$ below $v \in V(\w)$ with incoming edges $e_l, e_r$ and outgoing edge $e_o$, we have
        \[
            \gamma_v(e_o) = \min(\gamma_v(e_l), \gamma_v(e_r)).
        \]
        \item For any hexavalent vertex $v' \in V(\w)$ below $v \in V(\w)$ with incoming edges $e_1, e_2, e_3$ and outgoing edges $e_1', e_2', e_3'$, we have
        \begin{align*}
            \gamma_v(e_1') &= \gamma_v(e_2) + \gamma_v(e_3) - \min(\gamma_v(e_1), \gamma_v(e_3)) \\
            \gamma_v(e_3') &= \gamma_v(e_1) + \gamma_v(e_2) - \min(\gamma_v(e_1), \gamma_v(e_3)) \\
            \gamma_v(e_2') &= \min(\gamma_v(e_1), \gamma_v(e_3)).
        \end{align*}
    \end{itemize}
\end{definition}

We now explain how the cluster variables of the initial seed associated with a right inductive weave $\w(\beta)$ on $\beta \in \Br_n^+$ are computed. Our notations of $s$-variables follows \cite[Theorem 5.19]{CGGLSS}.

\begin{definition}[{{\cite[Definition 5.19]{CGGLSS}}}]\label{dfn:s-variable_weave}
    Let $\w$ be a right simplifying weave and $v \in V(\w)$ be a trivalent vertex. Then $s_v \in \C[\mathfrak{X}(\w)]$ is the ${z}$-variable assigned to the northeast incoming framed weave segment $e_{ne}$ of the trivalent vertex $v$ after setting all $u$-variables at the top of the weave equal to one.
\end{definition}

\begin{definition}[{\cite[Definition 5.18]{CGGLSS}}]
    Let $\beta \in \Br_n^+$ such that $\delta(\beta) = w_0$, and let $\w(\beta)$ be a right inductive weave. Given two trivalent vertices $v, v' \in V(\w(\beta))$ we say that \emph{$v'$ covers $v$} if $\gamma_{v'}(e_{nw}) \neq 0$ where $e_{nw}$ is the northwest incoming edge of $v$.
\end{definition}

The following result is analogous to \cite[Theorem 5.19]{CGGLSS}, which is stated for left inductive weaves.

\begin{theorem}\label{thm:cluster_inductive}
    Let $\w$ be a right inductive weave.
    \begin{enumerate}
        \item For the topmost trivalent vertex $v \in V(\w)$, the $s_v$-variable agrees with the corresponding ${z}$-variable on the edge $e_{ne}$ at the top layer of the framed weave.
        \item For any trivalent vertex $v \in V(\w)$, the cluster variable $A_v(\w)$ satisfies the inductive formula
        \[
        A_v = s_v \cdot \prod_{v' \text{ covers } v} A_{v'}^{\gamma_{v'}(e_{nw})}. \eqno\qed
        \]
    \end{enumerate}
\end{theorem}

\begin{example}\label{ex:cluster-hopf}
    Let $\beta = \sigma_1^3$. Then 
    \[
        X(\beta) = \{(z_1, z_2, z_3) \in \C^3 \mid B_1(z_3)B_1(z_2)B_1(z_1)w_0 \in B\}.
    \]
    We compute the cluster variables $A_v(\w)$ for a right inductive weave $\w$. Using the framed monodromy conditions from \cref{dfn:framed_weave_corr}, we find the framed weave with trivial monodromy to be as in \cref{fig:cluster_ex}.
    After setting $u_1 = u_2 = u_3 = 1$, \cref{thm:cluster_inductive} gives $A_1 = s_1 = z_2$, $s_2 = z_3 - z_2^{-1}$, and $A_2 = s_2A_1 = z_2 z_3-1$.
    \begin{figure}[!htb]
        \centering
        \includegraphics{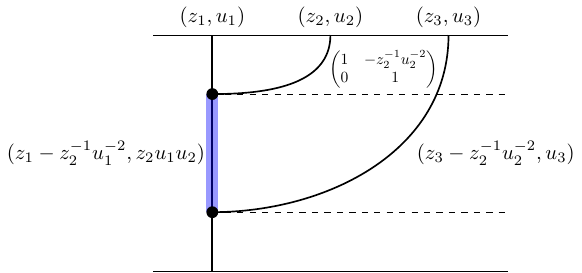}
        \caption{A framed right inductive weave with trivial monodromy. The highlighted edge is the only Lusztig cycle that covers.}
        \label{fig:cluster_ex}
    \end{figure}
\end{example}

\subsection{Morse complex sequences}\label{sec:mcs}
Morse complex sequences, which were first introduced in unpublished work of Pushkar{\cprime}, and in print by Henry \cite[Section 3]{Henry}, are a combinatorial version of generating families (families of Morse functions that generate Legendrians). They are a finite sequence of chain complexes over a commutative ring $R$ together with chain maps relating consecutive chain complexes constrained by the front diagram. Our presentation follows \cite{Henry,HenryRutherford15}.

\begin{definition}[Morse complex sequence]\label{dfn:mcs}
    A graded \emph{Morse complex sequence (MCS)} over a commutative ring $R$ for a front diagram $D$ equipped with a Maslov potential, is a quadruple
    \[
        C = (\{(C_\ell, d_\ell)\}_{\ell=0}^m, \{x_\ell\}_{\ell=0}^m, H, \{\varphi_\ell\}_{\ell=0}^{m-1})
    \]
    consisting of the following
    \begin{enumerate}
        \item $H$ is a set whose elements are called \emph{handleslide marks}. A handleslide mark is a vertical line segment in the $xz$-plane avoiding crossings and cusps of $D$ with endpoints on two strands with the same Maslov potential of $D$. A handleslide mark also comes equipped with an element $r \in R$.
        \item $\{x_\ell\}_{\ell=0}^m$ is a strictly increasing sequence so that $D \subset \{x_0 \leq x \leq x_m\}$, and such that for each $\ell \in \{0,\ldots,m-1\}$ the tangle $D_{\ell} \coloneqq D \cap \{x_\ell \leq x \leq x_{\ell+1}\}$ only contains a single cusp, crossing, or handleslide mark (see \cref{fig:mcs_diff1,fig:mcs_diff2,fig:mcs_diff3,fig:mcs_diff4}).
        \item For each $\ell \in \{0,\ldots,m\}$, $C_\ell$ is a free graded $R$-module generated by $e(\ell)_1, \ldots, e(\ell)_{s_{\ell}}$, labeled by the points $D \cap \{x = x_\ell\}$ from top to bottom. The grading of each $e(\ell)_i$ is defined to be the value of the Maslov potential of the strand it is labeled by. The differential $d_\ell$ is defined by the following local contributions for all $\ell \in \{0,\ldots,m-1\}$:
        \begin{enumerate}
            \item If $D_\ell$ contains a crossing of the $k$-th and $(k+1)$-st strands (see \cref{fig:mcs_diff1}), then
            \[
            \langle d_{\ell} e(\ell)_k,e(\ell)_{k+1}\rangle = \langle d_{\ell+1} e(\ell+1)_k,e(\ell+1)_{k+1}\rangle = 0.
            \]
            \item If $D_\ell$ contains a right cusp between the $k$-th and $(k+1)$-st strands (see \cref{fig:mcs_diff2}), then
            \[
            \langle d_{\ell} e(\ell)_k,e(\ell)_{k+1}\rangle = -s_k, \quad s_k \in R^\ast.
            \]
            \item If $D_\ell$ contains a left cusp between the $k$-th and $(k+1)$-st strands (see \cref{fig:mcs_diff3}), then
            \[
            \langle d_{\ell+1} e(\ell+1)_k,e(\ell+1)_{k+1}\rangle = 1.
            \]
        \end{enumerate}
        \item For each $\ell \in \{0,\ldots,m-1\}$, $\varphi_\ell \colon (C_\ell,d_\ell) \to (C_{\ell+1},d_{\ell+1})$ is a quasi-isomorphism defined by the tangle $D_\ell$, depending on its kind:
        \begin{enumerate}
            \item If $D_\ell$ contains a crossing of the $k$-th and $(k+1)$-st strands (see \cref{fig:mcs_diff1}), then $\varphi_\ell$ is defined on the set of generators as follows
            \[
            \varphi_\ell(e(\ell)_i) = \begin{cases}
                e(\ell+1)_{k+1}, & i = k, \\
                e(\ell+1)_k, & i = k+1, \\
                e(\ell+1)_i, & \text{otherwise}.
            \end{cases}
            \]
            \item If $D_\ell$ contains a right cusp between the $k$-th and $(k+1)$-st strands (see \cref{fig:mcs_diff2}), then $\varphi_\ell$ is defined on the set of generators as follows
            \[
            \varphi_\ell(e(\ell)_i) = \begin{cases}
                0, & i \in \{k,k+1\}, \\
                e(\ell+1)_{i}, & i < k, \\
                e(\ell+1)_{i-2}, & i > k+1.
            \end{cases}
            \]
            \item If $D_\ell$ contains a left cusp between the $k$-th and $(k+1)$-st strands (see \cref{fig:mcs_diff3}), then $\varphi_\ell$ is defined on the set of generators as follows
            \[
            \varphi_\ell(e(\ell)_i) = \begin{cases}
                e(\ell+1)_{i}, & i \leq k, \\
                e(\ell+1)_{i+2}, & i \geq k+1. 
            \end{cases}
            \]
            \item If $D_\ell$ contains a handleslide mark between the $j$-th and $k$-th strands equipped with the element $r\in R$ (see \cref{fig:mcs_diff4}), then $\varphi_\ell$ is defined on the set of generators as follows
            \[
            \varphi_\ell(e(\ell)_i) = \begin{cases}
                e(\ell+1)_j+re(\ell+1)_k, & i = j, \\
                e(\ell+1)_i, & \text{otherwise}.
            \end{cases}
            \]
        \end{enumerate}
    \end{enumerate}
\end{definition}

\begin{figure}[!htb]
    \centering
    \begin{subfigure}{0.45\textwidth}
        \centering
        \includegraphics{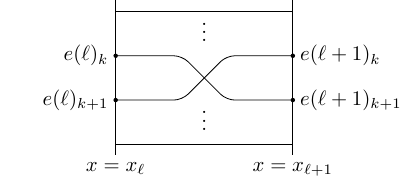}
        \caption{}\label{fig:mcs_diff1}
    \end{subfigure}
    \begin{subfigure}{0.45\textwidth}
        \centering
        \includegraphics{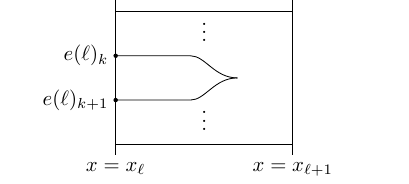}
        \caption{}\label{fig:mcs_diff2}
    \end{subfigure}
    \begin{subfigure}{0.45\textwidth}
        \centering
        \includegraphics{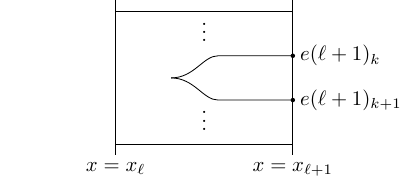}
        \caption{}\label{fig:mcs_diff3}
    \end{subfigure}
    \begin{subfigure}{0.45\textwidth}
        \centering
        \includegraphics{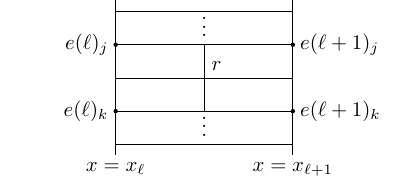}
        \caption{}\label{fig:mcs_diff4}
    \end{subfigure}
    \caption{The different kinds of tangles $D_\ell$ used to define $d_\ell$ and $\varphi_\ell$.}
\end{figure}

\begin{remark}\label{rmk:mcs_difference}
    Our \cref{dfn:mcs}(4)(d) differs by a minus sign compared to \cite[Definition 4.1(5)(b)]{HenryRutherford15}. Consequently our MCS moves \cref{fig:mcs_local3,fig:mcs_local5} differ from \cite[Figures 13(d) and 13(e)]{HenryRutherford15} by a minus sign. This will be important when comparing Morse complex sequences with algebraic weaves in \cref{sec:monodromy_mcs}; see \cref{rmk:mcs_matrix_order}. Our definition of SR-form MCS (see \cref{dfn:sr-form_mcs} below) consequently also differs from the definition in Henry--Rutherford by suitable signs; see \cref{rmk:sr-form_diff}.
\end{remark}

\begin{remark}\label{rmk:mcs_marked_pts}
 Because we work with the front diagram of the $(-1)$-closure of $\beta\Delta$ for some $\beta \in \Br_n^+$ with one marked point per strand (as opposed to one marked point per component), in \cref{dfn:mcs}(3)(b), we set $\langle d_{\ell} e(\ell)_k,e(\ell)_{k+1}\rangle = -s_k$ for $s_k \in R^\ast$ for every right cusp. This is in contrast with \cite[Definition 4.1]{HenryRutherford15}, where only some right cusps are marked.
\end{remark}

\begin{definition}[Simple left cusp]
    Let $C$ be an MCS over $R$ for a front diagram $D$ equipped with a Maslov potential. A \emph{simple left cusp} for $C$ is a left cusp in $D$ (see \cref{fig:mcs_diff3}) such that
    \[
    \langle d_{\ell+1}e(\ell+1)_j,e(\ell+1)_k\rangle = 0 = \langle d_{\ell+1}e(\ell+1)_j,e(\ell+1)_{k+1}\rangle \quad \forall j < k.
    \]
\end{definition}
A simple right cusp is defined analogously, although we only consider simple left cusps in this paper. The main motivation for introducing simple left cusps is the following.
\begin{proposition}[{\cite[Proposition 4.2]{HenryRutherford15}}]\label{prop:handleslides_determine_mcs}
    An MCS for a front diagram $D$ such that all of its left cusps are simple, is uniquely determined by its handleslide set $H$.
    \qed
\end{proposition}

It is useful to study the equivalence class of MCSs under a number of ``MCS moves.'' For computational purposes in \cref{sec:mcs_and_alg_weaves}, we include additional MCS moves involving marked points that are depicted in \cref{fig:mcs_moves_marked_pts}.

\begin{definition}[{\cite[Definition 2.6]{HenryRutherford15Equiv}}]\label{defn:MCSmoves}
    Let $D$ be a front diagram equipped with a Maslov potential. Let $C$ and $C'$ be two MCSs for $D$ such that all of its left cusps are simple with respect to both $C$ and $C'$. We say that $C$ and $C'$ are equivalent if they are related by a finite sequence of MCS moves, depicted in \cref{fig:mcs_moves,fig:mcs_moves_marked_pts}. 
\end{definition}

\begin{figure}[!htb]
    \centering
    \begin{subfigure}{0.45\textwidth}
        \centering
        \includegraphics{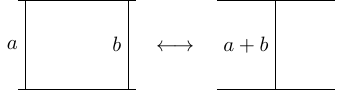}
        \caption{}\label{fig:mcs_local1}
    \end{subfigure}
    \begin{subfigure}{0.45\textwidth}
        \centering
        \includegraphics{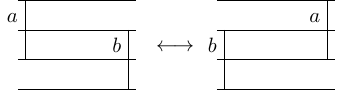}
        \caption{}\label{fig:mcs_local2}
    \end{subfigure}
    \begin{subfigure}{0.45\textwidth}
        \centering
        \includegraphics{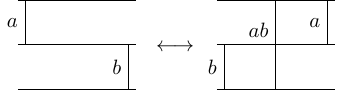}
        \caption{}\label{fig:mcs_local3}
    \end{subfigure}    
    \begin{subfigure}{0.45\textwidth}
        \centering
        \includegraphics{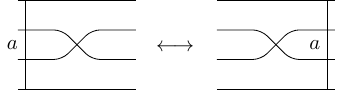}
        \caption{}\label{fig:mcs_local4}
    \end{subfigure}
    \begin{subfigure}{0.45\textwidth}
        \centering
        \includegraphics{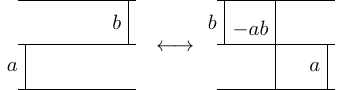}
        \caption{}\label{fig:mcs_local5}
    \end{subfigure}
    \begin{subfigure}{0.45\textwidth}
        \centering
        \includegraphics{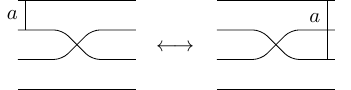}
        \caption{}\label{fig:mcs_local6}
    \end{subfigure}
    \begin{subfigure}{0.45\textwidth}
        \centering
        \includegraphics{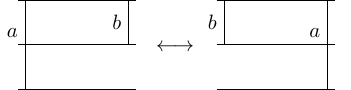}
        \caption{}\label{fig:mcs_local7}
    \end{subfigure}
    \begin{subfigure}{0.45\textwidth}
        \centering
        \includegraphics{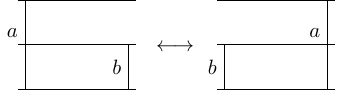}
        \caption{}\label{fig:mcs_local8}
    \end{subfigure}
    \begin{subfigure}{0.45\textwidth}
        \setcounter{subfigure}{8}
        \centering
        \includegraphics{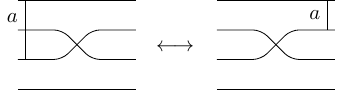}
        \caption{}\label{fig:mcs_local_extra1}
    \end{subfigure}
    \begin{subfigure}{0.45\textwidth}
        \centering
        \includegraphics{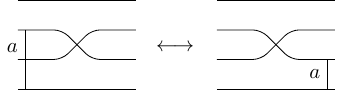}
        \caption{}\label{fig:mcs_local_extra2}
    \end{subfigure}
    \begin{subfigure}{0.45\textwidth}
        \centering
        \includegraphics{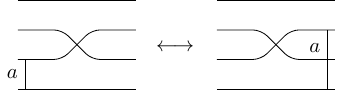}
        \caption{}\label{fig:mcs_local_extra3}
    \end{subfigure}
    \begin{subfigure}{0.45\textwidth}
        \centering
        \includegraphics{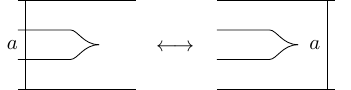}
        \caption{}\label{fig:mcs_local_cusp1}
    \end{subfigure}
    \begin{subfigure}{0.45\textwidth}
        \centering
        \includegraphics{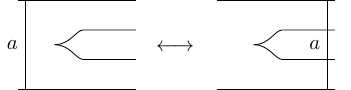}
        \caption{}\label{fig:mcs_local_cusp2}
    \end{subfigure}
    \begin{subfigure}{0.45\textwidth}
        \centering
        \includegraphics{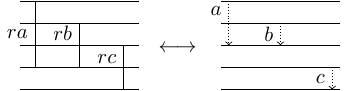}
        \caption{}\label{fig:mcs_local13}
    \end{subfigure}
    \caption{MCS moves.}
    \label{fig:mcs_moves}
\end{figure}

\begin{figure}[!htb]
    \centering
    \begin{subfigure}{0.45\textwidth}
        \centering
        \includegraphics{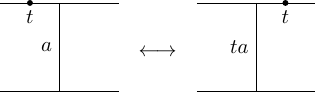}
        \caption{}\label{fig:mcs_local_marked1}
    \end{subfigure}
    \begin{subfigure}{0.45\textwidth}
        \centering
        \includegraphics{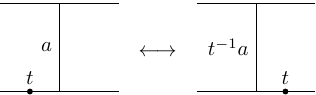}
        \caption{}\label{fig:mcs_local_marked2}
    \end{subfigure}
    \begin{subfigure}{0.45\textwidth}
        \centering
        \includegraphics{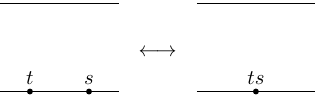}
        \caption{}\label{fig:mcs_local_marked3}
    \end{subfigure}
    \begin{subfigure}{0.45\textwidth}
        \centering
        \includegraphics{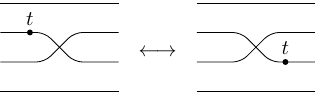}
        \caption{}\label{fig:mcs_local_marked4}
    \end{subfigure}
    \begin{subfigure}{0.45\textwidth}
        \centering
        \includegraphics{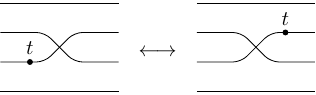}
        \caption{}\label{fig:mcs_local_marked5}
    \end{subfigure}
    \begin{subfigure}{0.45\textwidth}
        \centering
        \includegraphics{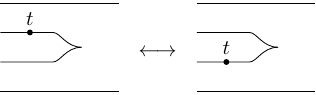}
        \caption{}\label{fig:mcs_local_marked6}
    \end{subfigure}
    \caption{MCS moves involving marked points.}
    \label{fig:mcs_moves_marked_pts}
\end{figure}

There are two special forms of MCSs called the A-form and SR-form that are closely related to augmentations and normal rulings, respectively. 
\begin{definition}[A-form MCS {\cite[Definition 5.1]{HenryRutherford15}}]\label{dfn:a_form_mcs}
    A (graded) \emph{A-form MCS} for a front diagram $D$ equipped with a Maslov potential is a graded MCS for $D$ with simple left cusps such that the set of handleslide marks consists of one handleslide mark immediately to the left of every crossing between the two crossing strands, equipped with an arbitrary element of $R$.
\end{definition}

\begin{definition}[SR-form MCS {\cite[Definition 4.4]{HenryRutherford15}}]\label{dfn:sr-form_mcs}
    Let $D$ be a front diagram equipped with a Maslov potential. A graded \emph{SR-form MCS} for $D$ with respect to a normal ruling $\rho \in \mathfrak{R}(D)$ is a graded MCS for $D$ with simple left cusps such that the set of handleslide marks consists of the following:
    \begin{description}
        \item[Near switches] One handleslide mark immediately to the left and one handleslide mark immediately to the right of every switch equipped with the elements $r \in R^\ast$ and $-r^{-1} \in R^\ast$, respectively. Moreover, there is one more handleslide mark between the two ``companion strands" of the two strands involved in the switch, depending on whether the ruling disks are nested or disjoint (see the first row of \cref{fig:ruling}), equipped with the element $-ar^{-1}b^{-1}$, where $a$ and $b$ are certain coefficients in the differential $d_\ell$; see \cref{fig:mcs1,fig:mcs2,fig:mcs3}.
        \item[Near returns] One handleslide mark immediately to the left of the return equipped with the element $r \in R$. Moreover, one more handleslide mark between the two ``companion strands'' of the two strands involved in the return, depending on whether the ruling disks are nested or disjoint (see the third row of \cref{fig:ruling}), equipped with the element $-a^{-1}rb$ or $-arb^{-1}$, depending on the type of return, where $a$ and $b$ are certain coefficients in the differential $d_\ell$; see \cref{fig:mcs4,fig:mcs5,fig:mcs6}.
        \item[Near departures] One handleslide mark immediately to the left of the departure equipped with the element $0 \in R$.
    \end{description}
\end{definition}

\begin{figure}[!htb]
    \centering
    \begin{subfigure}{0.32\textwidth}
        \centering
        \includegraphics{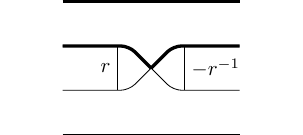}
        \caption{}\label{fig:mcs1}
    \end{subfigure}
    \begin{subfigure}{0.32\textwidth}
        \centering
        \includegraphics{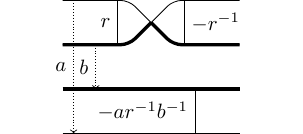}
        \caption{}\label{fig:mcs2}
    \end{subfigure}
    \begin{subfigure}{0.32\textwidth}
        \centering
        \includegraphics{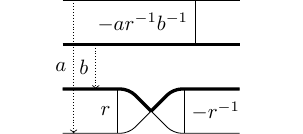}
        \caption{}\label{fig:mcs3}
    \end{subfigure}
    
    \vspace{1cm}
    
    \begin{subfigure}{0.32\textwidth}
        \centering
        \includegraphics{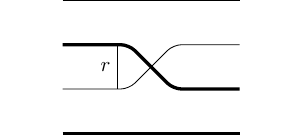}
        \caption{}\label{fig:mcs4}
    \end{subfigure}
    \begin{subfigure}{0.32\textwidth}
        \centering
        \includegraphics{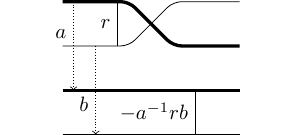}
        \caption{}\label{fig:mcs5}
    \end{subfigure}
    \begin{subfigure}{0.32\textwidth}
        \centering
        \includegraphics{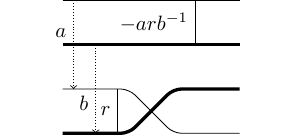}
        \caption{}\label{fig:mcs6}  
    \end{subfigure}
    \caption{Local models for the SR-form MCS associated to a normal ruling. The thickness of the lines indicate the boundaries of the two ruling disks. A dotted arrow between strands $j$ and $k$ for $j < k$ with label $x$ means $\langle d_\ell e(\ell)_k, e(\ell)_j\rangle = x$.}
\end{figure}
\begin{remark}
    A handleslide mark equipped with $0 \in R$ is equivalent to the absence of that handleslide mark.
\end{remark}
\begin{remark}\label{rmk:sr-form_diff}
    The elements decorating the handleslide marks between the companion strands in the SR-form MCS differ from those appearing in \cite[Definition 4.4]{HenryRutherford15} by a minus sign; see \cref{rmk:mcs_difference}.
\end{remark}

  Equivalent SR-form Morse complex sequences belong to the same normal ruling:

\begin{proposition}[{\cite[Corollary 4.3]{HenryRutherford15Equiv}}]\label{prop:SRmcs-equiv}
    Let $D$ be a nearly plat front diagram of a Legendrian link $\Lambda$ in $\R^3$. Suppose $C$ and $C'$ are equivalent SR-form Morse complex sequences. Then there exists $\rho \in \mathfrak{R}(D)$ such that $C, C' \in \MCS^\rho(D)$.
    \qed
\end{proposition}

\begin{notation}
    Let $D$ be a front diagram equipped with a Maslov potential.
    \begin{itemize}
        \item Let $\MCS^A(D)$ denote the set of A-form MCSs and let $\widehat{\MCS}{}^A(D)$ denote its set of equivalence classes.
        \item Let $\MCS^\rho(D)$ denote the set of SR-form MCSs with respect to the normal ruling $\rho$ and let $\widehat{\MCS}{}^\rho(D)$ denote its set of equivalence classes.
    \end{itemize}
    To avoid certain edge cases, we also define $\widehat{\MCS}{}^A(D) \coloneqq \varnothing$ if $\MCS^A(D) = \varnothing$ and $\widehat{\MCS}{}^\rho(D) \coloneqq \varnothing$ if $\MCS^\rho(D) = \varnothing$ (cf.\@ \cref{def:aug_variety}).
\end{notation}

\begin{theorem}[{\cite[Theorem 5.2]{HenryRutherford15}}]\label{thm:hr_Amcs_aug}
    Let $D$ be a nearby plat front diagram of a Legendrian link $\Lambda\subset \R^3$. Then there is an isomorphism $\MCS^A(D) \cong \V(D)$.
    \qed
\end{theorem}
\begin{remark}\label{rmk:aform_aug_iso}
    The isomorphism in \cref{thm:hr_Amcs_aug} for the $(-1)$-closure of $\beta\Delta$ for some $\beta \in \Br_n^+$ (see \cref{fig:closure}) is explicitly given by defining $\epsilon \colon \A(\La(\beta\Delta)) \to \C$ by $\epsilon(a) = -x_a$, where $x_a$ is the handleslide mark associated with the crossing of $\La(\beta\Delta)$ that corresponds to the degree $0$ generator $a$.
\end{remark}

\begin{theorem}[{\cite[Theorem 1.1]{HenryRutherford15Equiv}}]\label{thm:hr_Amcs_aug_equiv}
    Let $D$ be a nearly plat front diagram of a Legendrian link $\Lambda$ in $\R^3$. Then there exists an isomorphism
    \[
    \widehat\MCS{}^A(D) \cong \Aug(D). \eqno\qed
    \]
\end{theorem}
\begin{remark}
    Both \cref{prop:SRmcs-equiv} and \cref{thm:hr_Amcs_aug_equiv} are originally proved over $\Z/2\Z$ in \cite{HenryRutherford15Equiv}. However, with the sign conventions we define (following \cite{HenryRutherford15}), the proofs go through with coefficients in any commutative ring.
\end{remark}

\subsection{Ruling decomposition}\label{sec:ruling_decomp}
In this section we discuss the ruling decomposition of the augmentation variety introduced by Henry--Rutherford \cite{HenryRutherford15}.

\begin{theorem}[{\cite[Theorem 6.1]{HenryRutherford15}}]\label{thm:one-to-one_A_SR}
    Let $D$ be a nearly plat front diagram of a Legendrian link $\Lambda$ in $\R^3$. There exists an isomorphism 
    \begin{equation}\label{eq:A_SR}
        \MCS^A(D) \cong \bigsqcup_{\rho \in \mathfrak{R}(D)} \MCS^\rho(D),
    \end{equation}
    where $\mathfrak{R}(D)$ is the set of all normal rulings of $D$.
    \qed
\end{theorem}

\begin{remark}\label{rmk:A_SR_algorithm}
    Let us give a brief description of the bijection in \cref{thm:one-to-one_A_SR}.
    
    First we consider the map taking an A-form MCS $\mathcal C_1$ that corresponds to an augmentation $\epsilon$ to an SR-form MCS with respect to $\rho$. Namely, the map constructs a sequence of MCSs $\mathcal{C}_1,\ldots, \mathcal C_N$ such that $\mathcal C_N$ is an SR-form MCS. By induction assume that $\mathcal{C}_i$ is in SR-form to the left of the $i$-th crossing $q_i$. If $q_{i-1}$ is a switch, there are two handleslide mark to the right of $q_{i-1}$ labeled by $-x_{i-1}^{-1}$ and $x_{i-1}^{-1}$; use MCS moves to push all handleslide marks between $q_{i-1}$ and $q_i$ to the right of $q_i$, keeping the handleslide mark labeled by $-x_{i-1}^{-1}$ fixed. If $q_{i-1}$ is not a switch, use MCS moves to push all handleslide marks between $q_{i-1}$ and $q_i$ to the right of $q_i$. Next, if $q_i$ is a return or departure, this is $\mathcal C_{i+1}$. If $q_i$ is a switch, then there is a handleslide mark to the left of $q_i$ between the crossing strands with coefficient $x_i$. Add two handleslide marks between the crossing strands to the right of $q_i$ with coefficients $-x_i^{-1}$ and $x_i^{-1}$. We iterate this process until we have passed all crossings. If there are any handleslide marks remaining to the right of the rightmost crossing we simply remove them. It turns out that the final MCS $\mathcal{C}_N$ is in SR-form with respect to $\rho$.

    We now describe the reverse map taking an SR-form MCS $\mathcal C'_1$ with respect to $\rho$ to an A-form MCS. As above, we construct a sequence of MCSs $\mathcal C'_1,\ldots,\mathcal C'_N$ where $\mathcal C'_N$ is an A-form MCS. Again, by induction, suppose that $\mathcal{C}'_i$ is in A-form to the left of the $i$-th crossing $q_i$. First, use MCS moves to push the handleslide marks between $q_{i-1}$ and $q_i$ to the right of $q_i$. If $q_i$ is a return or departure of $\rho$, this is $C'_{i+1}$. If $q_i$ is a switch, then there were handleslide marks to the left and right of $q_i$ between the crossing strands with coefficients $x_i$ and $-x_i^{-1}$, respectively, before pushing handleslide marks to the right. Thus, after pushing the handleslide marks to the right of $q_i$, there will still be a handleslide mark to the left of $q_i$. We iterate this process until we have passed all crossings. If there are any handleslide marks (including the ones between the companion strands) remaining to the right of the rightmost crossing we simply remove them. The final MCS $\mathcal{C}'_N$ will be in A-form MCS with respect to $\epsilon$.
\end{remark}

\begin{theorem}[{\cite[Theorems 4.6 and 4.12]{HenryRutherford13}}]\label{thm:sr_bijection}
   Let $D$ be a front diagram of the $(-1)$-closure of $\beta\Delta$ for some $\beta \in \Br_n^+$ with $\delta(\beta) = w_0$. Let $\rho$ be a normal ruling for $D$. There is an isomorphism
   \begin{equation}\label{eq:sr_handleslide_marks}
    \eta_\rho \colon \C^{r(\rho)} \times (\C^\ast)^{s(\rho)} \overset{\cong}{\longrightarrow} \MCS^\rho(D)
    \end{equation}
    that is given by sending a tuple $((x_i)_{1\leq i\leq r(\rho)},(z_j)_{1\leq j\leq s(\rho)})$ to the SR-form MCS with handleslide marks $x_i$ at returns of $\rho$, and $z_j$ (and $-z_j^{-1}$) at switches of $\rho$.
    \qed
\end{theorem}

As a consequence of \cref{thm:hr_Amcs_aug,thm:one-to-one_A_SR,thm:sr_bijection} Henry and Rutherford obtain the following ruling decomposition of the naive augmentation variety. The only difference between the statement here and \cite{HenryRutherford15} is that we consider the front diagram $D$ with one marked point per right cusp instead of one marked point per component. From now on we let $R = \C$.

\begin{theorem}[{\cite[Theorem 3.4]{HenryRutherford15}}]\label{thm:HenryRutherford}
Suppose $D$ is a nearly plat front diagram of a Legendrian link $\Lambda \subset \R^3$. There exists a decomposition of the naive augmentation variety of $D$ into a disjoint union
\[\V(D)=\bigsqcup_{\rho\in\frakr(D)} \MCS^\rho(D),\]
where $\frakr(D)$ is the set of all normal rulings of $D$. Each set $\MCS^\rho(D)$ is the image of an injective regular map
\[
\phi_\rho\colon(\C^\ast)^{s(\rho)}\times\C^{r(\rho)} \overset{\cong}{\longrightarrow} \MCS^\rho(D) \hooklongrightarrow \V(D),
\]
where the first map is $\eta_\rho$ in \eqref{eq:sr_handleslide_marks}.
\qed
\end{theorem}

Combining the decomposition of the naive augmentation variety $\V(D)$ described in \cref{thm:HenryRutherford} with \cref{prop:SRmcs-equiv,thm:hr_Amcs_aug_equiv}, we obtain a ruling decomposition for the augmentation variety $\Aug(\Lambda)$.

\begin{theorem}\label{thm:HR_decomp_equiv}
    Let $D$ be a nearly plat front diagram of a Legendrian link $\Lambda$ in $\R^3$. There exists an isomorphism
    \begin{equation}\label{eq:A_SR_equiv}
        \phi \colon \widehat\MCS{}^A(D) \overset{\cong}{\longrightarrow} \bigsqcup_{\rho \in \mathfrak{R}(D)} \widehat\MCS{}^\rho(D),
    \end{equation}
    where $\mathfrak{R}(D)$ is the set of normal rulings of $D$, such that the following diagram commutes:
    \[
        \begin{tikzcd}[sep=scriptsize]
            \MCS^A(D) \dar \rar{\text{\eqref{eq:A_SR}}} & \bigsqcup_{\rho \in \mathfrak{R}(D)} \MCS^\rho(D) \dar \\
            \widehat{\MCS}{}^A(D) \rar{\text{\eqref{eq:A_SR_equiv}}} & \bigsqcup_{\rho \in \mathfrak{R}(D)} \widehat{\MCS}{}^\rho(D)
        \end{tikzcd},
    \]
    where the vertical arrows are the quotient maps.
    \qed
\end{theorem}

\begin{lemma}\label{lem:a-form_bijection_equiv}
    Let $\beta \in \Br_n^+$. There is an isomorphism
    \begin{equation}\label{eq:a-form_reflect}
        \MCS^A(\La(\Delta\beta)) \overset{\cong}{\longrightarrow} \widehat{\MCS}{}^A(\La(\Delta\beta)) \times \C^{\binom n2},
    \end{equation}
    that is defined by sending an A-form MCS to the pair consisting of its equivalence class and the tuple $(x_{i})_{1\leq i\leq \binom n2} \in \C^{\binom n2}$ of handleslide marks of the crossings of $\Delta$.
\end{lemma}
\begin{proof}
    This follows from \cref{thm:bijection_equiv}, \cref{thm:hr_Amcs_aug}, and \cref{thm:hr_Amcs_aug_equiv}.
\end{proof}

\begin{lemma}\label{lem:sr_form_bijection_equiv}
    Let $\beta\in \Br_n^+$ be such that $\delta(\beta) = w_0$. For any normal ruling $\rho \in \mathfrak{R}(\La(\Delta\beta))$ there is an isomorphism
    \[
        \C^{r(\rho)-\binom n2} \times (\C^\ast)^{s(\rho)} \overset{\cong}{\longrightarrow} \widehat{\MCS}{}^\rho(\La(\Delta\beta))
    \]
    that is defined by sending $((x_i)_{{\binom n2}+1\leq i\leq r(\rho)},(z_j)_{1\leq j\leq s(\rho)})$ to the equivalence class of the SR-form MCS with handleslide marks $x_i$ at returns of $\rho$, and $z_j$ (and $-z_j^{-1}$) at switches of $\rho$.
\end{lemma}
\begin{proof}
    From \cref{thm:one-to-one_A_SR,thm:HR_decomp_equiv} we obtain the composition of isomorphisms
    \begin{align*}
        \C^{r(\rho)} \times (\C^\ast)^{s(\rho)} &\overset{\text{\eqref{eq:sr_handleslide_marks}}}{\longrightarrow} \MCS^\rho(\La(\Delta\beta)) \overset{\text{\eqref{eq:A_SR}}}{\longrightarrow} \MCS^A(\La(\Delta\beta)) \\
        &\overset{\text{\eqref{eq:a-form_reflect}}}{\longrightarrow} \widehat{\MCS}{}^A(\La(\Delta\beta)) \times \C^{\binom n2} \overset{\text{\eqref{eq:A_SR_equiv}} \times \id}{\longrightarrow} \widehat{\MCS}{}^\rho(\La(\Delta\beta)) \times \C^{\binom n2}.
    \end{align*}
    For any normal ruling $\rho$, every crossing of $D$ corresponding to Artin generators of $\Delta$ must automatically be a departure; see \cref{prop:max_ruling}. Then, by the description of the isomorphisms \eqref{eq:sr_handleslide_marks} and \eqref{eq:a-form_reflect}, the result follows.
\end{proof}

\begin{remark}
    In contrast to \cref{lem:a-form_bijection_equiv}, we do not have an explicit description of the isomorphism \eqref{eq:a-form_reflect} for general $\beta$; see \cref{rmk:nonexplicit_iso}.
\end{remark}

\begin{corollary}\label{cor:sr_bijection_upside-down}
    Let $\beta \in \Br_n^+$. For any normal ruling $\rho$ of $\La(\beta\Delta)$, there is an isomorphism
    \begin{equation}\label{eq:sr_bijection_upside-down}
        \MCS^\rho(\La(\beta\Delta)) \cong \MCS^\ohrud((\Deltaud\atebud)\Lambdaud).
    \end{equation}
\end{corollary}
\begin{proof}
    This is a consequence of the isomorphism $\MCS^A((\Deltaud\atebud)\Lambdaud) \cong \MCS^A(\La(\beta\Delta))$ from the proof of \cref{lem:a-form_bijection_equiv}.
\end{proof}
\begin{lemma}\label{lma:iso_-1closures}
    Let $\beta \in \Br_n^+$ be such that $\delta(\beta) = w_0$. For any normal ruling $\rho$ of $\La(\beta\Delta)$, there is an isomorphism
    \begin{equation}\label{eq:sr-form_factor_general}
        \MCS^\rho(\La(\beta\Delta)) \cong \widehat{\MCS}{}^\rho(\La(\beta\Delta)) \times \C^{\binom n2}.
    \end{equation}
\end{lemma}
\begin{proof}
    In view of \cref{lma:rulings_upside_down}, the result follows by combining \cref{thm:one-to-one_A_SR}, \cref{thm:HR_decomp_equiv}, and \cref{lem:a-form_bijection_equiv}.
\end{proof}
\begin{lemma}\label{lma:sr-form_equiv_factors}
    Let $\beta \in \Br_n^+$. For any normal ruling $\rho$ of $\La(\beta\Delta)$, there is an isomorphism
    \[
        \eta_\rho \colon \C^{r(\rho)-\binom n2} \times (\C^\ast)^{s(\rho)} \overset{\cong}{\longrightarrow} \widehat{\MCS}{}^\rho(\La(\beta\Delta)).
    \]
\end{lemma}
\begin{proof}
    By \cref{lma:rulings_upside_down} it is clear that there is a one-to-one correspondence between returns of $\rho$ and returns of $\ohrud$, as well as between switches of $\rho$ and switches of $\ohrud$. Then we have a composition of isomorphisms
    \begin{align*}
        \C^{r(\rho)} \times (\C^\ast)^{s(\rho)} &\cong \C^{r(\ohrud)} \times (\C^\ast)^{s(\ohrud)} \overset{\text{\eqref{eq:sr_handleslide_marks}}}{\longrightarrow} \MCS^\ohrud((\Deltaud\atebud)\Lambdaud) \\
        &\overset{\text{\eqref{eq:sr_bijection_upside-down}}}{\longrightarrow} \MCS^\rho(\La(\beta\Delta)) \overset{\text{\eqref{eq:sr-form_factor_general}}}{\longrightarrow} \widehat{\MCS}{}^\rho(\La(\beta\Delta)) \times \C^{\binom n2}. \qedhere
    \end{align*}
\end{proof}

\begin{remark}
In \cref{thm:sr_bijection}, we defined $\eta_\rho \colon \C^{r(\rho)} \times (\C^\ast)^{s(\rho)} \xrightarrow{\cong} \MCS^\rho(D)$. From context, it will always be clear to which map $\eta_\rho$ we are referring.
\end{remark}

\begin{theorem}\label{thm:HenryRutherford_dg-homotopy}
Let $D$ be a nearly plat front diagram of the $(-1)$-closure of $\beta\Delta$ for some $\beta \in \Br_n^+$. There exists a decomposition of the augmentation variety of $D$ into a disjoint union
\[\Aug(D)=\bigsqcup_{\rho\in\frakr(D)} \widehat{\MCS}{}^\rho(D),\]
where $\frakr(D)$ is the set of all normal rulings of $D$, and each set $\widehat{\MCS}{}^\rho(D)$ is identified as the image of the injective regular map
\[\phi_\rho\colon\C^{r(\rho)-\binom n2}\times(\C^\ast)^{s(\rho)} \xrightarrow{\cong} \widehat{\MCS}{}^\rho(D) \hooklongrightarrow \Aug(D).\]
Moreover, the following diagram commutes
\[
\begin{tikzcd}[column sep=scriptsize]
   (\C^\ast)^{s(\rho)}\times\C^{r(\rho)} \dar{\cong}[swap]{\eta_\rho} \drar{\phi_\rho} & \\
   \MCS^\rho(D) \rar[hook] \dar & \V(D) \dar \\
   \widehat{\MCS}{}^\rho(D) \rar[hook] & \Aug(D) \\
   (\C^\ast)^{s(\rho)-\binom n2}\times\C^{r(\rho)} \uar{\eta_\rho}[swap]{\cong} \urar[swap]{\phi_\rho}
\end{tikzcd}.
\]
\end{theorem}
\begin{proof}
    This follows from \cref{thm:HR_decomp_equiv} and \cref{lma:sr-form_equiv_factors}. By \cref{thm:HR_decomp_equiv}, the injective regular map $\phi_\rho$ in the statement guarantees, by construction, that the diagram commutes.
\end{proof}

\subsection{Rulings, weaves and sheaves}\label{sec:rulings_weaves_sheaves}
    Weaves and rulings may also be described via the moduli space of microlocal rank 1 sheaves, which, under the equivalence between sheaves and augmentations \cite{NRSSZ}, recovers results in the previous subsections.

    First, we recall the result by Shende--Treumann--Zaslow \cite[Theorem 1.7]{STZ_ConstrSheaves} that rulings induce decompositions on the moduli space of microlocal rank 1 sheaves for Legendrian rainbow closures of positive braids, and generalize it to $(-1)$-closures of positive braids. Roughly speaking, the ruling decomposition comes from the filtration of microlocal rank 1 sheaves whose associated graded sheaves are rank 1 sheaves supported on ruling disks \cite[Section 5.2]{STZ_ConstrSheaves}. Recall the notation $\mathcal M^\textit{fr}_1(D, T)$ from \cref{notn:sheaves}.
\begin{theorem}\label{thm:ruling_sheaf_decomp}
    Let $D$ be a front diagram of the $(-1)$-closure of $\beta\Delta$ for some $\beta \in \Br_n^+$ with $\delta(\beta) = w_0$. There exists a decomposition of the moduli space of microlocal rank 1 sheaves singularly supported on the Legendrian link:
    \[
    \mathcal M^\textit{fr}_1(D, T) = \bigsqcup_{\rho \in \mathfrak R(D)}\mathcal M_1^\textit{fr}(D, T)^\rho,
    \]
    where $\frakr(D)$ is the set of all normal rulings of $D$, and each set $\mathcal M_1^\textit{fr}(D, T)^\rho$ is identified as the image of the injective regular map
    \[
    \phi_\rho'\colon(\C^\ast)^{s(\rho)-\binom n2}\times\C^{r(\rho)} \overset{\cong}{\longrightarrow} \mathcal M_1^\textit{fr}(D, T)^\rho \hooklongrightarrow \mathcal M^\textit{fr}_1(D, T).
    \]
\end{theorem}
\begin{proof}
    Since the framing is determined by one base point per strand in the braid, we assume that the flag on the left of the braid is $0 \subset \C \subset \C^2 \subset \dots \subset \C^n$. Thus, for any simple sheaf $\mathcal F \in \mathcal M_1^\textit{fr}(D, T)$, we have a filtration $R_\bullet$ such that
    \[
    R_i\mathcal F(U) = \mathcal F(U) \cap \C^i, \quad 0\leq i \leq n.
    \] 
    Define the ruling disks to be the supports of the associated graded sheaves. Then the result follows from the same argument in \cite[Theorem 1.7 or Proposition 6.31]{STZ_ConstrSheaves}; the only difference is that here we work with general $(-1)$-closures of positive braids, which affects the proof of \cite[Proposition 6.28]{STZ_ConstrSheaves} and \cite[Proposition 6.31]{STZ_ConstrSheaves}. 
    
    However, as in \cref{prop:sheaves_braid_var}, we can consider Reidemeister 2 moves at the right cusps and the $\Delta$ crossings as in \cref{fig:r2_normal_rulings} that induce equivalences of the moduli spaces of sheaves as in \cite[Proposition 6.5]{CasalsLi}. Then, the regions above the right cusps determine a complete flag $\C^n$. We can then follow the same argument in \cite[Proposition 6.28]{STZ_ConstrSheaves} to construct the filtration of the sheaf by intersecting with each subspace in the complete flag. Note that the complement of the boundary of the innermost ruling disk in the $(-1)$-closure of $\beta\Delta$ is still a $(-1)$-closure. This shows that we can follow the same argument in \cite[Proposition 6.31]{STZ_ConstrSheaves} to construct the decomposition of the moduli space.
\end{proof}

    It is not obvious from the descriptions that the ruling decomposition of the augmentation variety agrees with the ruling decomposition of the moduli space of sheaves under the isomorphism \cref{thm:-1=braid vty_new,prop:sheaves_braid_var}. We provide the proof in \cref{sec:weaves_aug_sheaves} using the weave decomposition of the moduli space of sheaves.

    We now prove the folklore result that the moduli space of algebraic weaves $\mathfrak{X}(\w)$ is naturally isomorphic to the moduli space of microlocal rank 1 sheaves with singular support on the associated Legendrian weave. This essentially follows from \cite[Theorem 5.3]{CasalsZaslow20}. Recall from \cref{sec:weaves} that a weave $\w$ yields a Legendrian $L(\w) \subset \R^5$. Similar to \cref{notn:sheaves} we denote by $\mathcal M_1^\textit{fr}(L(\w), T)$ the framed moduli space of microlocal rank 1 sheaves on $\D^2 \times \R$ with singular support in the Legendrian $L(\w) \subset \R^5$. We let $T$ be the set consisting of one base point on each sheet of the weave.

\begin{proposition}\label{prop:sheaf-on-weave}
    Let $\w$ be a weave in $\D^2$. There is an isomorphism 
    \[\mathfrak{X}(\w) \cong \mathcal M_1^\textit{fr}(L(\w), T).\]
\end{proposition}
\begin{proof}
    Denote the components of the complement of $\w$ inside $\D^2$ by weave regions. By \cite[Theorem 5.1]{CasalsZaslow20}, the moduli space of microlocal rank 1 sheaves on $\D^2 \times \R$ with singular support in the associated Legendrian weave $L(\w)$ is given by
    \[
    \mathcal{M}_1(L(\w)) = \{(F_\alpha) \in (G/B)^{\text{weave regions}} \mid F_{\alpha} \xrightarrow{s_i} F_{\alpha'} \text{ if $\alpha, \alpha'$ share a $G_i$-edge}\}/G.
    \]
    When $F_1, F_2$, and $F_3$ are pairwise in relative $s_i$-position, $F_2 = B_i(z_1)F_1$, $F_3 = B_i(z_2)F_2$ and $F_3 = B_i(z_3)F_1$, there is an upper triangular $V$ such that $B_i(z_2)B_i(z_1) = VB_i(z_3)$. This implies the trivial monodromy conditions. Moreover, for any upper triangular matrix $V$, there exists an upper triangular matrix $V'$ such that $B_i(z)V = V'B_i(w)$, and the associated flags of $B_i(z)$, $B_i(z)V$, and $V'B_i(w)$ are the same. Consider the framing by one base points per strand which determines the trivialization of a flag $F_\alpha = B$. We conclude that there is an injection
    \[
    \mathcal{M}_1^\textit{fr}(L(\w), T) \hooklongrightarrow \C^{ \text{weave segments}} \times \left( \C^{\binom n2} \times \left(\C^\ast\right)^n\right)^{\text{dashed segments}}
    \]
    whose image has trivial monodromy and thus is exactly $\mathfrak X(\w)$. This finishes the proof.
\end{proof}

\subsection{Deodhar decomposition}\label{sec:deodhar_decomp}
Braid varieties generalize open Richardson varieties \cite[Corollary 4.6]{CGGS2}. Deodhar constructed a decomposition of the open Richardson varieties by locally closed subvarieties \cite{Deodhar}, and this is generalized to all braid varieties in \cite{GalashinLamTrinhWilliams}. In this subsection, we discuss the Deodhar decomposition and explain the relationship between the Deodhar decomposition \cite{Deodhar,GalashinLamTrinhWilliams} and the weave decomposition of braid varieties defined in \cref{sec:weave_decomp}. See also~\cite{MarshRietsch,WebsterYakimov} for results on Deodhar decompositions and \cite{Speyer23} for a survey on Richardson varieties.

We let $G \coloneqq \GL(n, \C)$, let $B \coloneqq B_+ \subset G$ be the subgroup of upper triangular matrices, and let $B_-\subset G$ be the subgroup of lower triangular matrices. Recall from \cref{sec:braid_variety} that $G/B_+$ is the flag variety and recall our notation for complete flags.

\begin{definition}[Schubert cell]
    Let $w \in S_n$. The \emph{Schubert cell} associated with $w$ is $X^\circ_w \coloneqq B_- w B_+ / B_+$. The \emph{opposite Schubert cell} associated with $w$ is $X^{\circ w} \coloneqq B_+ w B_+/B_+$.
\end{definition}
\begin{definition}[Open Richardson variety]\label{dfn:richardson_variety}
    Let $u, w \in S_n$ such that $u \leq w$ in the Bruhat order, that is, the reduced word for $u$ is a reduced subword of $w$. The \emph{open Richardson variety} associated with $u$ and $w$ is
    \[
    R^\circ_{u,w} \coloneqq X^\circ_w \cap X^{\circ u} = \{F \in G/B_+ \mid B_+ \xrightarrow{w} F,\; B_- \xrightarrow{w_0u} F\}.
    \]
    More generally, for a positive braid $\beta \in \Br_n^+$ of length $r$, the braid Richardson variety associated with $w$ and $\beta=\sigma_{i_1}\cdots\sigma_{i_r}$ is 
    \[
    R^\circ_{u,\beta} = \big\{(F_1, \dots, F_{r+1}) \in (G/B_+)^r \mid F_k \xrightarrow{s_{i_k}} F_{k+1},\; F_1 = B_+, \; B_- \xrightarrow{w_0u} F_{r+1} \big\}.
    \]
\end{definition}

For the following proposition, the first part is \cite[Corollary 4.6]{CGGS2}, and the second part is essentially \cite[Proposition 2.8]{CGGSBS} by identifying $\beta$ as the double braid word with only negative indices.

\begin{proposition}\label{prop:braid_richardson}
    Let $u, w \in S_n$ such that $u \leq w$ in the Bruhat order. Let $\beta(w)$ and $\beta(u^{-1}w_0)$ be the positive braid lifts of $w$ and $u^{-1}w_0$, respectively. Then there is an isomorphism of algebraic varieties
    \[
    X(\beta(w) \beta(u^{-1}w_0)) \cong R^\circ_{u,w}.
    \]
    For a positive braid $\beta \in \Br_n^+$ with $\delta(\beta) = w_0$, there is an isomorphism of algebraic varieties
    \[
    X(\beta) \cong R^\circ_{w_0,\beta}. \eqno\qed
    \]
\end{proposition}

Viewing the open Richardson variety $R^\circ_{u,w}$ as a subvariety in the Schubert cell $X^\circ_w$, the Deodhar decomposition on $R^\circ_{u,w}$ is then induced from the Schubert stratification on the Schubert cell $X^\circ_w$. Such decompositions generalize to all braid Richardson varieties $R^\circ_{u,\beta}$. We follow the conventions in \cite[Definition 4.1]{GalashinLamTrinhWilliams}, \cite[Section 4]{Speyer23} and \cite[Section 3.6]{GLSS}.

\begin{definition}[Distinguished sequence]
    Let $\beta = \sigma_{i_1}\cdots \sigma_{i_r} \in \Br_n^+$. A \emph{distinguished sequence} for $u \in S_n$ is a tuple $\mathfrak{v} = (v_0, \dots, v_{r})$ such that $v_0 = e$, $v_{r} = u$, and such that for each $j\in \{1,\dots,r\}$, exactly one of the following conditions holds:
    \begin{enumerate}
        \item $v_j = v_{j-1}$ and $v_j > v_j s_{i_j}$,
        \item $v_j = v_{j-1}s_{i_j}$ and $v_j < v_{j-1}$,
        \item $v_j = v_{j-1}s_{i_j}$ and $v_j > v_{j-1}$.
    \end{enumerate}
    A \emph{positive distinguished expression} for $u \in S_n$ is a distinguished sequence $(v_0, \dots, v_r)$ such that for any $j\in \{1,\ldots,r\}$, $v_j$ satisfies (1) or (3).
\end{definition}

\begin{example}
Let $\beta=\sigma_1^3$. The only distinguished sequences for $u=s_1$ are $(e,s_1,s_1,s_1)$ and $(e,s_1,e,s_1)$.
\end{example}

\begin{remark}
    Our definition differs from \cite[Definition 4.1]{GalashinLamTrinhWilliams}, \cite[Lemma 4.1]{Speyer23} or \cite[Definition 3.1 or 3.16]{GLSS} in Case (1). The difference comes from reading the braid from left to right or from right to left: a distinguished sequence $(v_1, \dots, v_r)$ for $u$ and $\beta$ in our definition corresponds to a distinguished sequence $(v_r^{-1}u, \dots, v_1^{-1}u)$ for $u$ and the reverse braid $\ateb$ (again identified as a double braid word with only negative indices).
\end{remark}

\begin{definition}[Deodhar pieces {\cite[Section 7]{GalashinLamTrinhWilliams}} {\cite[Section 4.1]{Speyer23} \cite[Definitions 3.1 and 3.16]{GLSS}}]\label{dfn:deodhar_piece}
    Let $\mathfrak{v} = (v_0, \dots, v_r)$ be a distinguished sequence for $u \in S_n$. Then, the \emph{Deodhar piece} associated with $(v_0, \dots, v_r)$ in $R^{\circ}_{u,\beta}$ is defined as
    \[
    D(v_0, \dots, v_r) \coloneqq \{ (F_1, \dots, F_{r+1}) \in (G/B_+)^{r+1} \mid F_1 = B_+,\; F_{i_j} \xrightarrow{v_j} F_1\; \forall j \in \{1,\dots,r\}\}.
    \]
\end{definition}

\begin{lemma}[Deodhar {\cite[Theorem 7.4]{GalashinLamTrinhWilliams} \cite[Lemma 4.5]{Speyer23}}]\label{lem:deodhar-decomposition}
    Let $\beta \in \Br_n^+$ and $(v_0, \dots, v_r)$ be a distinguished sequence for $u \in S_n$. Let $t$ be the number of indices $j \in \{2,\dots,r\}$ such that $v_j = v_{j-1}$, and $c$ be the number of indices $j \in \{2,\dots,r\}$ such that $v_j = v_{j-1}s_{i_j}$ and $v_j < v_{j-1}$. 
    
    \begin{enumerate}
    \item There is an algebraic isomorphism
    \[
        D(v_0, \dots, v_r) \cong (\C^\ast)^t \times \C^c.
    \]
    \item The natural map $D(v_0, \dots, v_r) \to R_{u,\beta}^\circ$ is injective.
    \end{enumerate}
\end{lemma}
\begin{proof}
    We proceed by inducting on the length of the distinguished sequences. Denote the sequence of flags in $D(v_0, \dots, v_r)$ by $B_+, x_1B_+, \dots, x_rB_+$ and as $x_{j-1}B_+$ and $x_jB_+$ are in relative position $s_{i_j}$, we can write $x_jB_+ = B_{i_j}(z_j)x_{j-1}B_+$.
    
    In Case (1), $v_j = v_{j-1}$. Then if the flags satisfy $x_{j-1}B_+ \xrightarrow{v_{j-1}} B_+$, we know that $x_jB_+ \xrightarrow{v_j} B_+$ if and only if $z_j \in \C^\ast$. Therefore,
    \[
    D(v_0, \dots, v_{j}) \cong D(v_0, \dots, v_{j-1}) \times \C^\ast.
    \]
    In Case (2), $v_j = v_{j-1}s_{i_j}$ and $v_j < v_{j-1}$. Then if $x_{j-1}B_+ \xrightarrow{v_{j-1}} B_+$, we know that $x_jB_+ \xrightarrow{v_j} B_+$ for any $z_j \in \C$. Therefore,
    \[
    D(v_0, \dots, v_{j}) \cong D(v_0, \dots, v_{j-1}) \times \C.
    \]
    Finally, in Case (3), $v_j = v_{j-1}s_{i_j}$ and $v_j > v_{j-1}$. Then if $x_{j-1}B_+ \xrightarrow{v_{j-1}} B_+$, we know that $x_jB_+ \xrightarrow{v_j} B_+$ only if $z_j = 0$. Therefore,
    \[
    D(v_0, \dots, v_{j}) \cong D(v_0, \dots, v_{j-1}).
    \]
\end{proof}

    When $s_{i_1}\cdots s_{i_r}$ is a reduced word for $w \in S_n$, the Schubert cell associated with $w$ satisfies
    \[
    X^\circ_w = \{B_{i_j}(z_j) \cdots B_{i_1}(z_1) B_+ \in G/B_+ \mid (z_1, \dots, z_r) \in \C^r\}.
    \]
    
    Schubert cells form a decomposition of the flag variety, therefore the open Richardson variety $R^\circ_{u,w}$ can be decomposed into Deodhar pieces. Moreover, the Deodhar pieces that appear in the decomposition of $R^\circ_{u,w}$ are exactly the Deodhar pieces of a distinguished subsequence for $u \in S_n$; see \cite[Lemma 4.1]{Speyer23}. In general, the braid Richardson variety $R^\circ_{u,\beta}$ can also be decomposed into Deodhar pieces. We omit the proof of the following theorem since it follows from the identification with the weave decomposition in \cref{thm:deodhar=weave}.

\begin{theorem}[Deodhar decomposition {\cite[Theorem 7.4]{GalashinLamTrinhWilliams}}]\label{thm:deodhar}
    Let $u \in S_n$ and $\beta \in \Br_n^+$. Let $\mathfrak{v}$ be a positive distinguished expression of $u$ and $\beta$. Then the collection $(\mathfrak{v} = \mathfrak{v}_0, \dots, \mathfrak{v}_{\ell(\beta)})$ of distinguished sequences satisfies that
    \begin{enumerate}
        \item $R_{u,\beta}^\circ = \bigcup_{i=1}^{\ell(\beta)} D(\mathfrak{v}_i)$,
        \item $D(\mathfrak{v}_i) \subset R_{u,\beta}^\circ$ are pairwise disjoint for $i \in \{0, \dots, \ell(\beta)\}$, and
        \item $D(\mathfrak{v}) \cong (\C^\ast)^{\ell(\beta) - {n\choose 2}}$ is the unique piece of maximal dimension. \qed
    \end{enumerate}
\end{theorem}

    Now, we prove that the Deodhar decomposition for $R^\circ_{w_0, \beta}$ coincides with the weave decomposition of the braid variety $X(\beta)$ obtained from \cref{thm:weave_decomp}.

\begin{lemma}\label{lma:right_inductive_distinguished_sequences}
    Let $\pi \in S_n$ and $\beta \in \Br_n^+$ such that $\delta(\beta) = \pi$. Then there exists a bijective map $\psi$ from the set of distinguished sequences for $\pi \in S_n$ to the set of inductive equivalence classes of right simplifying algebraic weaves $\beta \to \delta(\beta)$.
\end{lemma}
\begin{proof}
    Given a distinguished sequence $(v_0, \dots, v_{r})$ for $u \in S_n$ and $\beta \in \Br_n^+$, we construct a right simplifying algebraic weave $\w(v_0, \dots, v_{r}) \colon \beta \to \beta(u)$ inductively as follows:
    \begin{enumerate}
        \item Let $\w(\id) \colon \beta \to \beta$ be the trivial weave.
        \item If $v_j = v_{j-1}$ and $v_j > v_j s_{i_j}$, we consider the subweave of $\w(\sigma_{i_1} \cdots \sigma_{i_{j-1}})$ on the left of the strand $s_{i_j}$ and apply some braid moves such that its $(j-1)$-th strand is in $G_{i_j}$, and join the $(j-1)$-st and $j$-th strand by a trivalent vertex and get $\w(\sigma_{i_1} \cdots \sigma_{i_{j}})$; see \cref{fig:inductive2}.
        \item If $v_j = v_{j-1}s_{i_j}$ and $v_j < v_{j-1}$, we consider the subweave of $\w(\sigma_{i_1} \cdots \sigma_{i_{j-1}})$ on the left of the strand $s_{i_j}$ and apply some braid moves such that its $(j-1)$-st strand is in $G_{i_j}$, and join the $(j-1)$-st and $j$-th strand by a cup and get $\w(\sigma_{i_1} \cdots \sigma_{i_{j}})$; see \cref{fig:inductive3}.
        \item If $v_j = v_{j-1}s_{i_j}$ and $v_j > v_{j-1}$, then we perform a horizontal composition with $\sigma_{i_j}$. 
    \end{enumerate}
    At the end of this process, one obtains a weave $\w\colon \beta\rightarrow \beta'$ where $\beta'$ is an equivalent braid word for $\delta(\beta)$. We obtain a weave $\w\colon \beta\rightarrow \delta(\beta)$ by vertically composing it with a weave $\beta'\rightarrow \delta(\beta)$ that only contains hexavalent and tetravalent vertices.
    
    There is an obvious projection map $\pi$ from the set of inductive equivalence classes of right simplifying algebraic weaves to the set of distinguished sequences, and one can check that the map $(v_0,\ldots,v_{r}) \mapsto \w(v_0,\ldots,v_{r})$ defined above is an inverse of $\pi$, showing that it is bijective. Moreover, if $t$ is the number of indices $j\in \{1,\dots,r\}$ such that $v_j = v_{j-1}$, and $m$ is the number of indices $j\in \{1,\dots,r\}$ such that $v_j < v_{j-1}$, then the simplifying weave has $m$ cups and $t$ trivalent vertices. Moreover, since our construction goes from left to right of the Demazure product $\delta(\beta)$, the simplifying weaves are right inductive with respect to $\beta$.
\end{proof}

\begin{theorem}\label{thm:deodhar=weave}
    Let $\beta \in \Br_n^+$ such that $\delta(\beta) = w_0$. For every distinguished sequence $\mathfrak{v}$ of $w_0$ there exists a right simplifying weave $\mathfrak{w}_{\mathfrak{v}}$ such that the Deodhar decomposition of $R^\circ_{w_0,\beta}$ agrees with the weave decomposition of $X(\beta)$ given by this collection of right simplifying weaves $\mathfrak{w}_{\mathfrak{v}}$ under the isomorphism $R^\circ_{w_0,\beta} \cong X(\beta)$ from \cref{prop:braid_richardson}.

\end{theorem}
\begin{proof}
    Since both the Deodhar decomposition of $R^\circ_{\pi,\beta}$ and weave decomposition of $X(\beta, \pi)$ are determined by induction on the length of the braid $\beta$, we can match up the decompositions by induction on the length of $\beta$.
    
    For the Deodhar decomposition, in the proof of \cref{lem:deodhar-decomposition} (see \cite[Lemma 4.5]{Speyer23}) it is shown that there are three cases:
    \begin{enumerate}
        \item $D(v_0, v_2, \dots, v_{r-1}, v_r) = D(v_0, v_2, \dots, v_{r-1}) \times \C^\ast$ if $v_j = v_{j-1}$ and $v_j > v_j s_{i_j}$, 
        \item $D(v_0, v_2, \dots, v_{r-1}, v_r) = D(v_0, v_2, \dots, v_{r-1}) \times \C$ if $v_j = v_{j-1}s_{i_j}$ and $v_j < v_{j-1}$, and
        \item $D(v_0, v_2, \dots, v_{r-1}, v_r) = D(v_0, v_2, \dots, v_{r-1})$ if $v_j = v_{j-1}s_{i_j}$ and $v_j > v_{j-1}$.
    \end{enumerate}
    For the weave decomposition of $X(\beta, \pi)$ by a collection of (right simplifying) weaves, it is shown in the proof of \cref{lma:tri_hexa_corresp} (see \cite[Theorem 5.9]{CGGS1}) that there are three cases:
    \begin{enumerate}
        \item $\mathfrak{X}(\w_r) = \mathfrak{X}(\w_{r-1}) \times \C^\ast$ if $\w_r$ is defined by adding a trivalent to $\w_{r-1}$,
        \item $\mathfrak{X}(\w_r) = \mathfrak{X}(\w_{r-1}) \times \C$ if $\w_r$ is defined by adding a cup to $\w_{r-1}$, and
        \item $\mathfrak{X}(\w_r) = \mathfrak{X}(\w_{r-1})$ if $\w_r$ is weave equivalent to $\w_{r-1}$.
    \end{enumerate}
    By \cref{lma:right_inductive_distinguished_sequences}, for a distinguished sequence, the right simplifying weave is constructed inductively so that the three cases appear exactly when the corresponding cases appear in the distinguished sequence. Consider a subsequence $(v_0, \dots, v_r)$ and let $\ell(r)$ be the length of a reduced word of $v_r$. Let $\beta_r$ be the subword of $\beta$ consisting of the first $r$ crossings and let $\w_{r} \colon \beta_{r} \to \beta(v_{r})$ be the associated weave constructed in \cref{lma:right_inductive_distinguished_sequences}. Let $B_+, x_1B_+, \dots, x_{r-1}B_+, yB_+$ be a sequence of flags in $D(v_0, \dots, v_{r-1})$. 
    
    We prove by induction that there is an isomorphism $D(v_0, \dots, v_{r-1}, v_r) \cong \mathfrak{X}(\w_{r})$ such that the corresponding sequence of flags on the top slice of the weave $\mathfrak{X}(\w_{r})$ is
    \[
    B_+,x_{1}B_+,\dots,x_{r-1}B_+,yB_+.
    \]
    To that end, suppose that $D(v_0, \dots, v_{p}) \cong \mathfrak{X}(\w_{p})$ for any $p\in \{1,\ldots,r-1\}$. In Case~(1), we know that $\ell(r) = \ell(r-1)$ and $x_{r-1}B_+ \neq yB_+$. Hence, the additional flag $x_rB_+$ in $D(v_0, \dots, v_{r-1}, v_r)$ can be any flag that is transverse to $x_{r-1}B_+$ and $yB_+$, and the additional flag $x_{r}B_+$ in $\mathfrak{X}(\w_{r})$ after adding a trivalent vertex can also be any flag that is transverse to $x_{r-1}B_+$ and $yB_+$, both of which are parametrized by $\C^\ast$. In Case~(2), we know that $\ell(r) = \ell(r-1) + 1$ and $x_{r-1}B_+ = yB_+$. Hence, the additional flag $x_rB_+$ in $D(v_0, \dots, v_{r-1}, v_r)$ can be any flag that is transverse to $x_{r-1}B_+$, and the additional flag $x_{r}B_+$ in $\mathfrak{X}(\w_{r})$ after adding a cup can also be any flag that is transverse to $x_{r-1}B_+$, both of which are parametrized by $\C$. In Case~(3), we know that $\ell(r) = \ell(r-1) - 1$, and the additional flag $x_rB_+$ in $D(v_0, \dots, v_{r-1}, v_r)$ must be $yB_+$, while no additional flag is introduced in $\mathfrak{X}(\w_{r})$. This shows $D(v_0,\dots,v_r) \cong \mathfrak{X}(\w_r)$ and hence completes the proof.
\end{proof}

\section{Comparing the ruling and weave decompositions}\label{sec:mcs_and_alg_weaves}
    The goal of this section is to prove our main theorem: The ruling decomposition of the augmentation variety described in \cref{sec:ruling_decomp} agrees with a corresponding weave decomposition of $X(\beta)$ described in \cref{sec:weave_decomp} via the isomorphism $\Aug(\La(\beta\Delta)) \cong X(\beta)$ in \cref{thm:-1=braid vty_new}. To compare the ruling and weave decompositions we define a category $\mathfrak{B}_n$ in \cref{sec:mcs_cat} whose set of objects is $\Br_n^+$ and morphisms are sequences of braids. In \cref{sec:monodromy_mcs} we bring the ruling and weave decompositions onto equal footing by showing that there is a commutative diagram
    \[
    \begin{tikzcd}[sep=scriptsize]
        \mathfrak{B}_n \ar[dr,swap,"\mathfrak M"] \ar[rr,"\mathfrak{A}"] && \mathfrak{W}_n \ar[dl,"\mathfrak{X}"] \\
        & \mathfrak{C}&
    \end{tikzcd},
    \]
    where $\mathfrak M$ is constructed via Morse complex sequences allowing us to express the weave decomposition of $X(\beta)$ solely in terms of Morse complex sequences. 
    
\subsection{A category of braids and the weave category}\label{sec:mcs_cat}
    The goal of this section is to define a category of braids $\mathfrak B_n$ that is equivalent to the weave category $\mathfrak W_n$. The main reason we introduce $\mathfrak B_n$ is to highlight the distinction between the functor $\mathfrak X$ constructed using trivial monodromy weaves and a functor $\mathfrak M$ constructed using Morse complex sequences.

\begin{definition}[Braid moves]
    Let $\beta, \beta' \in \Br_n^+$. We call the following moves \emph{braid moves}. We say that $\beta$ and $\beta'$ are related by a:
    \begin{enumerate}
        \item \emph{distant crossings move} if $\beta = \gamma_L\sigma_i\sigma_j\gamma_R$ and $\beta' = \gamma_L\sigma_j\sigma_i\gamma_R$ for $|i-j| > 1$; see \cref{fig:braid_distant}.
        \item \emph{hexavalent move} if $\beta = \gamma_L\sigma_i\sigma_{i+1}\sigma_i\gamma_R$ and $\beta' = \gamma_L\sigma_{i+1}\sigma_{i}\sigma_{i+1}\gamma_R$ for some $i$; see \cref{fig:braid_hexavalent}.
        \item \emph{trivalent move} if $\beta = \gamma_L\sigma_i^2\gamma_R$ and $\beta' = \gamma_L\sigma_{i}\gamma_R$ for some $i$; see \cref{fig:braid_trivalent}.
        \item \emph{cup move} if $\beta = \gamma_L\sigma_i^2\gamma_R$ and $\beta' = \gamma_L\gamma_R$ for some $i$ (respectively, $\beta$ and $\beta'$ are related by a \emph{cap move} if $\beta = \gamma_L\gamma_R$ and $\beta' = \gamma_L\sigma_i^2\gamma_R$ for some $i$); see \cref{fig:braid_cup}.
    \end{enumerate}
\end{definition}
\begin{figure}[!htb]
    \centering
    \begin{subfigure}{\textwidth}
        \centering
        \includegraphics{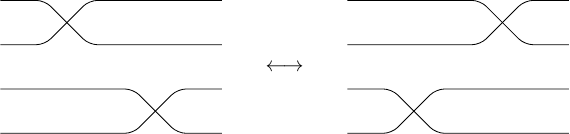}
        \caption{The distant crossings move $\sigma_i\sigma_j \leftrightarrow \sigma_j\sigma_i$ for $|i-j| > 1$.}\label{fig:braid_distant}
    \end{subfigure}
    \begin{subfigure}{\textwidth}
        \centering
        \includegraphics{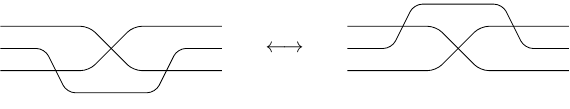}
        \caption{The hexavalent move $\sigma_i\sigma_{i+1}\sigma_i \leftrightarrow \sigma_{i+1}\sigma_i\sigma_{i+1}$.}\label{fig:braid_hexavalent}
    \end{subfigure}
    \begin{subfigure}{\textwidth}
        \centering
        \includegraphics{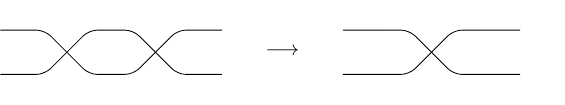}
        \caption{The trivalent move $\sigma_i^2 \to \sigma_i$.}\label{fig:braid_trivalent}
    \end{subfigure}
    \begin{subfigure}{\textwidth}
        \centering
        \includegraphics{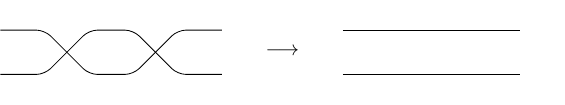}
        \caption{The cup move $\sigma_i^2 \to 1$.}\label{fig:braid_cup}
    \end{subfigure}
    \caption{Braid moves}
\end{figure}

\begin{definition}[Braid category]\label{dfn:braid_category}
    Let $n \in \Z_{\geq 1}$. The braid category $\mathfrak{B}_n$ is defined as follows:
    \begin{description}
        \item[Objects] $\Ob(\mathfrak{B}_n) = \Br_n^+$.
        \item[Morphisms] A morphism $\beta \to \beta'$ consists of a finite sequence of braids $(\beta \eqqcolon \beta_1,\ldots, \beta_q \coloneqq \beta')$ such that $\beta_{j-1}$ is related to $\beta_j$ via exactly one braid move.
        \item[Composition] The composition of two morphisms $\beta\to \beta'$ and $\beta'\to \beta''$ is given by concatenation of the corresponding sequences of braids and braid moves.
    \end{description}
\end{definition}

Recall the definition of the weave category $\mathfrak{W}_n$; see \cref{dfn:weave_category}.

\begin{theorem}\label{thm:ruling_to_weaves}
    Let $n\in \Z_{\geq 1}$. There is an equivalence $\mathfrak{A} \colon \mathfrak{B}_n \longrightarrow \mathfrak{W}_n$.
\end{theorem}
\begin{proof}
    The functor $\mathfrak{A}$ is defined as the identity map on objects. To define it on morphisms, let $\beta,\beta' \in \Br_n^+$. Given a morphism $\beta \to \beta'$, i.e., a sequence of braids $(\beta_1,\ldots,\beta_q)$, such that adjacent braids are related by braid moves. We construct an algebraic weave $\beta \to \beta'$ as the vertical concatenation of the algebraic weaves $\w_\ell \colon \beta_{\ell} \longrightarrow \beta_{\ell+1}$ for all $\ell \in \{1,\ldots,q\}$ where each is a single braid move, depending on how $\beta_{\ell}$ is related to $\beta_{\ell+1}$. It is clear that composition is respected. For any $\beta,\beta'\in \Br_n^+$, we may construct the inverse to $\mathfrak{A}_{\beta,\beta'} \colon \mathfrak{B}_n(\beta,\beta') \to \mathfrak{W}_n(\beta,\beta')$ by vertically decomposing a simplifying algebraic weave $\beta \to \beta'$ into algebraic weaves $\w_\ell \colon \beta_\ell \to \beta_{\ell+1}$ for all $\ell \in \{1,\ldots,q\}$ and assigning to each $\w_\ell$ the corresponding braid move between the braids $\beta_\ell$ and $\beta_{\ell+1}$.
\end{proof}
\begin{remark}\label{rmk:ruling_caps}
    Let $\mathfrak{B}_n^{\cup} \subset \mathfrak{B}_n$ denote the wide subcategory generated by braid operations with no caps, and $\mathfrak{W}_n^{\cup} \subset \w_n$ the wide subcategory generated by simplifying algebraic weaves (i.e., algebraic weaves with no caps). Then it is clear that $\mathfrak{A}$ restricts to an equivalence of categories $\mathfrak{A} \colon \mathfrak{B}_n^\cup \xrightarrow{\cong} \w_n^\cup$.
\end{remark}

\subsection{Monodromy of Morse complex sequences}\label{sec:monodromy_mcs}
The next goal is to define a functor $\mathfrak{B}_n \to \mathfrak{C}$, where $\mathfrak{C}$ is the category of algebraic correspondences. This functor is defined on the set of objects as $\beta \mapsto X(\beta)$. To define the functor on morphisms, we need additional discussion regarding Morse complex sequences, first introduced in \cref{sec:mcs}. More precisely, we discuss how (A-form) Morse complex sequences change under braid operations. These can be equivalently described using Morse complex 2-families on Legendrian surfaces introduced by Rutherford--Sullivan \cite{rutherford2018generating}; see \cref{rmk:mcs_monodromy=mcf_axiom,rmk:mcs_monodromy_at_vertex,rmk:mcs_morphism_are_mcf}. We also define an MCS analog of framed algebraic weaves from \cref{ssec:framed-weave}, and discuss their monodromy varieties.

\begin{remark}
    Throughout this section, we consider Morse complex sequences on a front diagram coming from a braid with Maslov potential on every strand equal to $0$. A Legendrian $(-1)$-closure also has strands with Maslov potential equal to $1$, but both the A-form MCS and the SR-form MCS are completely determined by the handleslide marks on the braid.
\end{remark}

\subsubsection{Morse complex sequences and their monodromy varieties}
\begin{definition}[MCS braid moves]\label{dfn:mcs_braid_moves}
We call the following four moves \emph{MCS braid moves}. 
\begin{enumerate}
    \item An \emph{MCS distant crossings move} is the local modification of the front diagram of an MCS defined by \cref{fig:mcs_distant}.
    \item An \emph{MCS hexavalent move} is the local modification of the front diagram of an MCS defined by \cref{fig:mcs_hexavalent}.
    \item An \emph{MCS trivalent move} is the local modification of the front diagram of an MCS defined by \cref{fig:mcs_trivalent}.
    \item An \emph{MCS cup move} is the local modification of the front diagram of an MCS defined by \cref{fig:mcs_cup}.
\end{enumerate}
\end{definition}

\begin{figure}[!htb]
    \centering
    \begin{subfigure}{\textwidth}
        \centering
        \includegraphics{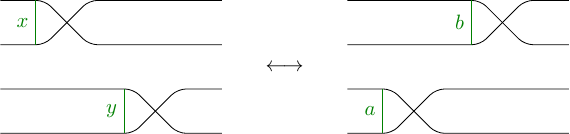}
        \caption{MCS distant crossings move.}\label{fig:mcs_distant}
    \end{subfigure}
    \begin{subfigure}{\textwidth}
        \centering
        \includegraphics{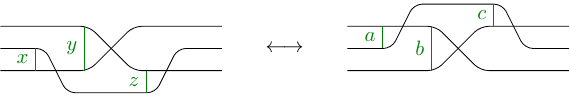}
        \caption{MCS hexavalent move.}\label{fig:mcs_hexavalent}
    \end{subfigure}
    \begin{subfigure}{\textwidth}
        \centering
        \includegraphics{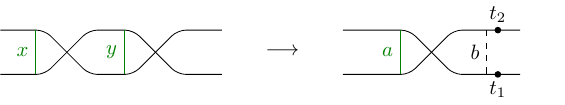}
        \caption{MCS trivalent move.}\label{fig:mcs_trivalent}
    \end{subfigure}
    \begin{subfigure}{\textwidth}
        \centering
        \includegraphics{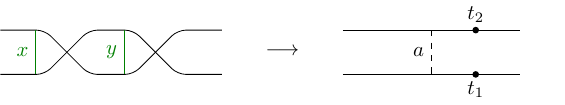}
        \caption{MCS cup move.}\label{fig:mcs_cup}
    \end{subfigure}
    \caption{MCS Braid moves}
    \label{fig:mcs_braid_moves}
\end{figure}

\begin{remark}
     We colloquially call the handleslide marks indicated with dashed lines on the right side of \cref{fig:mcs_trivalent,fig:mcs_cup} ``dashed handleslide marks'' and call the labeled dots ``marked points.'' 
\end{remark}

Given a Morse complex sequence $C$ as in \cref{dfn:mcs}, when all the strands in the front diagram have Maslov potential 0, the chain complexes $C_\ell$ are vector spaces, and the chain maps $\varphi_\ell$ are realized by the following matrices.

\begin{notation}\label{notn:matrices_mcs_rep}
    Let $\delta_{i,j}(z)$ denote the $n \times n$ matrix with $(\delta_{i,j}(z))_{i,j} = z$ and $(\delta_{i,j}(z))_{k,\ell} = 0$ for $(k,\ell) \neq (i,j)$. Define $S_{i,j}(z) \coloneqq I + \delta_{i,j}(z)$. Let $D_i(t) \coloneqq \diag(1,\ldots,1,t,1,\ldots,1)$ be the $n \times n$ diagonal matrix with $i$-th element equal to $t$. Let $P_i$ denote the $n \times n$ block matrix 
    \[
        P_i \coloneqq I_{i-1} \oplus \begin{pmatrix}0 & 1 \\ 1 & 0\end{pmatrix} \oplus I_{n-i-1}.
    \]
    Recall the definition of the braid matrix $B_i(z)$ from \eqref{eq:braid_matrix}:
    \[
        B_i(z) \coloneqq I_{i-1} \oplus \begin{pmatrix}0 & 1 \\ 1 & z\end{pmatrix} \oplus I_{n-i-1}.
    \]
\end{notation}

\begin{lemma}\label{lma:mcs_braid_matrix}
    We have $P_iS_{i,i+1}(z) = B_i(z)$.
    \qed
\end{lemma}

\begin{definition}[Monodromy of a Morse complex sequence]\label{dfn:monodromy_mcs}
    The \emph{monodromy} of a Morse complex sequence $C$ is a product of matrices as follows:
    \begin{itemize}
        \item Associate to each handleslide mark between the $i$-th and $k$-th strands (with $i \leq k$) labeled by $a\in R$, the matrix $S_{i,k}(a)$.
        \item Associated to each crossing $\sigma_i$, the matrix $P_i$.
        \item Associated to each marked point $t$ on strand $i$, the matrix $D_i(t)$.
    \end{itemize}
    When reading $C$ from left to right, the monodromy $\mu(C)$ is defined as the product of these matrices in the reversed order.
\end{definition}
\begin{remark}\label{rmk:mcs_matrix_order}
    In \cref{dfn:monodromy_mcs}, the convention to read matrices in the opposite order is \emph{different} from the convention in \cite[Section 3d]{Henry}. The MCS moves depicted in \cref{fig:mcs_local3,fig:mcs_local5} differ from those depicted in \cite[Figures 13(d) and 13(e)]{HenryRutherford15} by a sign; see \cref{rmk:mcs_difference}.
\end{remark}
\begin{remark}\label{rmk:monodromy_matrix_composition}
    The monodromy of an MCS encodes the matrices of the quasi-isomorphisms in the MCS associated to a front diagram $D$ away from the cusps. Without loss of generality, we assume that the front diagrams $D \subset \{x_0 \leq x \leq x_m\}$ we consider always have their cusps contained in the slices $\{x_0 \leq x \leq x_1\}$ and $\{x_{m-1} \leq x \leq x_m\}$. If $C = (C_0,C_1,\ldots,C_{m-1},C_m)$ is an MCS for $D$, with chain maps $\varphi_{\ell} \colon C_\ell \to C_{\ell+1}$ for $\ell \in \{0,\ldots,m-1\}$, we have
    \[
    \mu(C) = [\varphi_{m-2}\circ \cdots \circ \varphi_1].
    \]
\end{remark}
\begin{definition}[Monodromy of an MCS (braid) move]
    If $C_1$ and $C_2$ are two Morse complex sequences, the \emph{monodromy} of an MCS (braid) move $C_1 \to C_2$ is defined to be the matrix $\mu(C_1)\mu(C_2)^{-1}$. An MCS (braid) move whose monodromy is the identity matrix is said to have \emph{trivial monodromy}.
\end{definition}
\begin{remark}\label{rmk:mcs_monodromy=mcf_axiom}
    In view of \cref{rmk:monodromy_matrix_composition}, the trivial monodromy condition for an MCS (braid) move implies that the composition of quasi-isomorphisms of the chain complexes from the left to the right in $C_1$ and $C_2$ are chain homotopic (in fact equal). This is the condition in the definition of Morse complex 2-families \cite[Definition 4.1(2) and Axiom 4.3]{rutherford2018generating}.
\end{remark}
\begin{lemma}\label{lma:monodromy_comp_prod}
    If $C_1 \to C_2 \to C_3$ is the composition of two MCS (braid) moves of Morse complex sequences $C_1$, $C_2$, and $C_3$. The monodromy of the composition $C_1 \to C_3$ is equal to the product of the monodromies of $C_1 \to C_2$ and $C_2\to C_3$.
    \qed
\end{lemma}

\begin{lemma}[Monodromy of MCS moves]
    The monodromy of each of the MCS moves depicted in \cref{fig:mcs_local1,fig:mcs_local2,fig:mcs_local3,fig:mcs_local4,fig:mcs_local5,fig:mcs_local6,fig:mcs_local7,fig:mcs_local8,fig:mcs_local_extra1,fig:mcs_local_extra2,fig:mcs_local_extra3} and in \cref{fig:mcs_local_marked1,fig:mcs_local_marked2,fig:mcs_local_marked3,fig:mcs_local_marked4,fig:mcs_local_marked5} are given by the following matrices where $a, b,t$ and $x, y, z,s$ are variables (from left to right) on the left and right of each figure, respectively: 
    \begin{enumerate}[label=(18\Alph*)]
        \item $\mu_RS_{i,j}(b)S_{i,j}(a)\mu_L\left(\mu'_RS_{i,j}(x)\mu'_L\right)^{-1}$ for strands $i<j$,
        \item $\mu_RS_{i,k}(b)S_{j,\ell}(a)\mu_L\left(\mu'_RS_{j,\ell}(y)S_{i,k}(x)\mu'_L\right)^{-1}$ for strands $i<j<k<\ell$,
        \item $\mu_RS_{i,j}(b)S_{j,k}(a)\mu_L\left(\mu'_RS_{j,k}(z)S_{i,k}(y)S_{i,j}(x)\mu'_L\right)^{-1}$ for strands $i<j<k$,
        \item $\mu_RP_{i}S_{j,k}(a)\mu_L\left(\mu'_RS_{j,k}(x)P_{i}\mu'_L\right)^{-1}$ for strands $j<i<i+1<k$,
        \item $\mu_RS_{j,k}(b)S_{i,j}(a)\mu_L\left(\mu'_RS_{i,j}(z)S_{i,k}(y)S_{j,k}(x)\mu'_L\right)^{-1}$ for strands $i<j<k$,
        \item $\mu_RP_{i}S_{i+1,k}(a)\mu_L\left(\mu'_RS_{i,k}(x)P_{i}\mu'_L\right)^{-1}$ for strands $j<i<i+1<k$,
        \item $\mu_RS_{j,k}(b)S_{i,k}(a)\mu_L\left(\mu'_RS_{i,k}(y)S_{j,k}(x)\mu'_L\right)^{-1}$ for strands $i<j<k$,
        \item $\mu_RS_{i,j}(b)S_{i,k}(a)\mu_L\left(\mu'_RS_{i,k}(y)S_{i,j}(x)\mu'_L\right)^{-1}$ for strands $i<j<k$,
        \item $\mu_RP_{i}S_{i,k}(a)\mu_L\left(\mu'_RS_{i+1,k}(x)P_{i}\mu'_L\right)^{-1}$ for strands $j<i<i+1<k$,
        \item $\mu_RP_{i}S_{j,i+1}(a)\mu_L\left(\mu'_RS_{j,i}(x)P_{i}\mu'_L\right)^{-1}$ for strands $j<i<i+1<k$,
        \item $\mu_RP_{i}S_{j,i}(a)\mu_L\left(\mu'_RS_{j,i+1}(x)P_{i}\mu'_L\right)^{-1}$ for strands $j<i<i+1<k$,
    \end{enumerate}
    \begin{enumerate}[label=(19\Alph*)]
        \item $\mu_RS_{i,j}(a)D_{j}(t)\mu_L\left(\mu'_RD_{j}(s)S_{i,j}(x)\mu'_L\right)^{-1}$ for strands $i<j$,
        \item $\mu_RS_{i,j}(a)D_{i}(t)\mu_L\left(\mu'_RD_{i}(s)S_{i,j}(x)\mu'_L\right)^{-1}$ for strands $i<j$,
        \item $\mu_RD_{i}(s)D_{i}(t)\mu_L\left(\mu'_RD_{i}(r)\mu'_L\right)^{-1}$,
        \item $\mu_RP_iD_{i+1}(t)\mu_L\left(\mu'_RD_{i}(s)P_i\mu'_L\right)^{-1}$ if the crossings strands are $i$ and $i+1$,
        \item $\mu_RP_iD_{i}(t)\mu_L\left(\mu'_RD_{i+1}(s)P_i\mu'_L\right)^{-1}$ if the crossings strands are $i$ and $i+1$
    \end{enumerate}
    where $\mu_R$, $\mu'_R$, $\mu_L$, and $\mu'_L$ are the monodromies to the right and left in the Morse complex sequences outside of the local model.
    \qed
\end{lemma}
The following corollary is essentially equivalent to the matrix equations in the proof of \cite[Proposition 3.8]{Henry} encoding MCS moves.
\begin{corollary}\label{cor:mcs_matrix}
    The monodromy of each of the MCS moves depicted in \cref{fig:mcs_local1,fig:mcs_local2,fig:mcs_local3,fig:mcs_local4,fig:mcs_local5,fig:mcs_local6,fig:mcs_local7,fig:mcs_local8,fig:mcs_local_extra1,fig:mcs_local_extra2,fig:mcs_local_extra3} and \cref{fig:mcs_local_marked1,fig:mcs_local_marked2,fig:mcs_local_marked3,fig:mcs_local_marked4,fig:mcs_local_marked5} is equal to the identity. \qed
\end{corollary}

\begin{lemma}[Monodromy of MCS braid moves]
    The \emph{monodromies} of each of the MCS braid moves, with labels as in \cref{fig:mcs_braid_moves}, are given by the following matrices :
    \begin{description}
        \item[MCS distant crossings move]
        For strands $i+1 < j$, we have
        \begin{align*}
            &\mu_RP_iS_{i,i+1}(y)P_{j}S_{j,j+1}(x)\mu_L\left(\mu'_RP_{j}S_{j,j+1}(b)P_iS_{i,i+1}(a)\mu'_L\right)^{-1} \\
            &\qquad = \mu_RB_i(y)B_{j}(x)\mu_L\left(\mu'_RB_{j}(b)B_i(a)\mu'_L\right)^{-1}.
        \end{align*}
        \item[MCS hexavalent move]
        \begin{align*}
            &\mu_RP_{i}S_{i,i+1}(z)P_{i+1}S_{i+1,i+2}(y)P_{i}S_{i,i+1}(x)\mu_L\left(\mu'_RP_{i+1}S_{i+1,i+2}(c)P_iS_{i,i+1}(b)P_{i+1}S_{i+1,i+2}(a)\mu'_L\right)^{-1}\\
            &\qquad = \mu_RB_{i}(z)B_{i+1}(y)B_{i}(x)\mu_L\left(\mu'_RB_{i+1}(c)B_i(b)B_{i+1}(a)\mu'_L\right)^{-1}.
        \end{align*}
        \item[MCS trivalent move]
        \begin{align*}
            &\mu_RP_iS_{i,i+1}(y)P_iS_{i,i+1}(x) \mu_L\left(\mu'_RD_{i}(t_1)D_{i+1}(t_2)S_{i,i+1}(b)P_iS_{i,i+1}(a)\mu'_L\right)^{-1} \\
            &\qquad = \mu_RB_i(y)B_i(x) \mu_L\left(\mu'_RD_{i}(t_1)D_{i+1}(t_2)S_{i,i+1}(b)B_i(a)\mu'_L\right)^{-1}.
        \end{align*}
        \item[MCS cup move]
        \begin{align*}
            &\mu_RP_iS_{i,i+1}(y)P_iS_{i,i+1}(x) \mu_L\left(\mu'_RD_{i}(t_1)D_{i+1}(t_2)S_{i,i+1}(a)\mu'_L\right)^{-1} \\
            &\qquad = \mu_RB_i(y)B_i(x)\mu_L\left(\mu'_RD_{i}(t_1)D_{i+1}(t_2)S_{i,i+1}(a)\mu'_L\right)^{-1}.
        \end{align*}
    \end{description}
    \qed
\end{lemma}

The trivial monodromy condition determines the following relations for MCS braid moves.

\begin{lemma}\label{lma:monodromy_trivalent_etc_id}
    The monodromy of each of the MCS distant crossings, hexavalent, trivalent and cup moves is equal to the identity if and only if each handleslide mark outside of the local modification are equal to the corresponding handleslide mark after the move, and additionally:
    \begin{description}
        \item[MCS distant crossings move] $a = y$ and $b = x$; see \cref{fig:mcs_distant_monoless}.
        \item[MCS hexavalent move] $a = z$, $b = y-xz$, and $c = x$; see \cref{fig:mcs_hexavalent_monoless1,fig:mcs_hexavalent_monoless2},
        \item[MCS trivalent move] $a = x+y^{-1}$, $b = -y$, $t_1 = -y^{-1}$, and $t_2 = y$; see \cref{fig:mcs_trivalent_monoless}.
        \item[MCS cup move] $a = x$, $y = 0$, $t_1 = 1$, and $t_2 = 1$; see \cref{fig:mcs_cup_monoless}. \qed
    \end{description}
\end{lemma}
\begin{remark}\label{rmk:mcs_monodromy_at_vertex}
    It is not obvious that in the MCS braid moves, the handleslide mark can be arranged in 1-parameter families with bifurcations and endpoints that satisfy the condition in Morse complex 2-families \cite[Axiom 4.2]{rutherford2018generating}. We explain the correspondence in \cref{sec:weaves_aug_sheaves}.
\end{remark}

\begin{figure}[!htb]
    \centering
    \begin{subfigure}{0.49\textwidth}
        \hspace*{0.7mm}
        \includegraphics{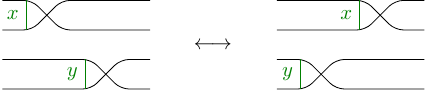}
        \caption{}\label{fig:mcs_distant_monoless}
    \end{subfigure}
    \begin{subfigure}{0.49\textwidth}
        \hspace*{1mm}
        \includegraphics{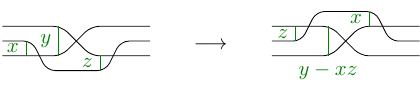}
        \caption{}\label{fig:mcs_hexavalent_monoless1}
    \end{subfigure}
    \begin{subfigure}{0.49\textwidth}
        \hspace*{1mm}
        \includegraphics{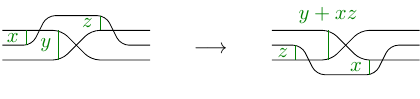}
        \caption{}\label{fig:mcs_hexavalent_monoless2}
    \end{subfigure}
    \begin{subfigure}{0.49\textwidth}
        \hspace*{1mm}
        \includegraphics{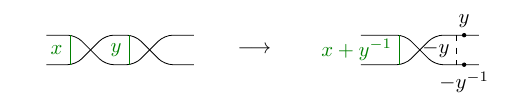}
        \caption{}\label{fig:mcs_trivalent_monoless}
    \end{subfigure}
    \begin{subfigure}{0.49\textwidth}
        \hspace*{1mm}
        \includegraphics{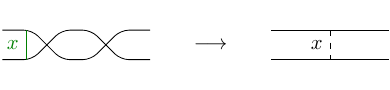}
        \caption{}\label{fig:mcs_cup_monoless}
    \end{subfigure}
    \caption{MCS moves with trivial monodromy.}
    \label{fig:mcs_monoless}
\end{figure}

Let $\mathfrak{m} \colon \beta \to \beta'$ be a morphism in the braid category, i.e., a sequence of braids $(\beta_1,\ldots,\beta_q)$ from $\beta$ to $\beta'$ and a sequence of braid moves. Let $p \in \{2,\ldots,q\}$. If $\beta_p$ is obtained from $\beta_{p-1}$ by a trivalent move or a cup move, then $\beta_{p-1} = \gamma_{p-1,L}\sigma_i^2\gamma_{p-1,R}$. Define $r(\beta_p) \coloneqq \ell(\gamma_{p-1,R})+1$. Otherwise, $r(\beta_p)\coloneqq 0$. Let $R \coloneqq \sum_{j=1}^q r(\beta_j)$. Assume there are $c$ cup moves, $h$ hexavalent moves, and $d$ distant crossings move in $\mathfrak{m}$, let $L \coloneqq \ell(\beta_1) + R + 3h + 2d - c$, and define
\begin{equation}\label{eq:total_sp_braids}
    \mathbb V^{\mathfrak{m}} \coloneqq \C^L \times \left(\C^{\binom n2} \times \left(\C^\ast\right)^n\right)^R.
\end{equation}
\begin{remark}
    The space $\mathbb V^{\mathfrak{m}}$ parametrizes the space of all $\C$-valued associated A-form MCS sequences for $\mathfrak{m}$.
\end{remark}

\begin{definition}[Formal A-form MCS]\label{dfn:formal_a-form_mcs}
    A (graded) \emph{formal A-form MCS} for a front diagram $D$ equipped with a Maslov potential is a graded A-form MCS where each handleslide mark is considered to be a formal variable valued in $R$.
\end{definition}
\begin{notation}
    For $\beta \in \Br_n^+$ with $\delta(\beta) = w_0$, we denote by $A(\beta)$ the formal A-form MCS associated to the $(-1)$-closure of $\beta\Delta$.
\end{notation}
\begin{definition}[Associated A-form MCS sequence]\label{dfn:assoc_a-form_mcs_sequence}
    Let $\mathfrak{m} \colon \beta \to \beta'$ be a morphism in $\mathfrak{B}_n$ with braid sequence $(\beta_1,\ldots,\beta_q)$. The \emph{associated A-form MCS sequence} is the sequence $A(\mathfrak{m}) \coloneqq (A(\beta_1),\ldots,A(\beta_q))$ of formal A-form MCSs such that for each $j\in \{2,\ldots,q\}$ exactly one of the following holds:
        \begin{enumerate}
            \item $A(\beta_{j-1})$ is related to $A(\beta_j)$ via an MCS distant crossings move.
            \item $A(\beta_{j-1})$ is related to $A(\beta_j)$ via an MCS hexavalent move.
            \item $A(\beta_{j-1})$ is related to $A(\beta_j)$ via an MCS trivalent move or an MCS cup move followed by some number of MCS moves that take the dashed handleslide marks and the marked points to the right of the rightmost crossing of $\beta_j$.
        \end{enumerate}
\end{definition}
\begin{definition}[Monodromy of a braid]
    Let $\mathfrak{m} \colon \beta \to \beta'$ be a morphism in $\mathfrak{B}_n$ given by a braid sequence $(\beta_1,\ldots,\beta_q)$. For any $\ell\in \{1,\ldots,q\}$ we define the \emph{monodromy} of $\beta_\ell = \sigma_{\ell_1} \cdots \sigma_{\ell_{r_\ell}}$, denoted by $\mu(\beta_\ell)$, to be the matrix
    \[
    \mu(\beta_\ell) \coloneqq P_{\ell_{r_\ell}}S_{\ell_{r_\ell},\ell_{r_\ell}+1}(z_{r_\ell}) \cdots P_{\ell_{1}}S_{\ell_{1},\ell_{1}+1}(z_{1}) = B_{\ell_{r_\ell}}(z_{r_\ell}) \cdots B_{\ell_1}(z_1).
    \]
\end{definition}

\begin{definition}[Variety of $\mathfrak{m}$]\label{dfn:variety_of_morphism}
    Let $\beta,\beta'\in \Br_n^+$ and suppose $\mathfrak{m} \colon \beta \to \beta'$ is a morphism in $\mathfrak{B}_n$. Let $\pi \in S_n$ be represented by a permutation matrix in $\GL(n,\C)$. We define $\mathfrak M(\mathfrak{m},\pi)$ to be the algebraic affine subvariety of $\mathbb V^{\mathfrak{m}}$ cut out by the following three conditions:
    \begin{enumerate}
        \item In $A(\mathfrak{m})$, the monodromy of each MCS distant crossings, hexavalent, trivalent, and cup moves relating $A(\beta_j)$ and $A(\beta_{j-1})$ for each $j$ equals the identity matrix.
        \item The monodromy of each MCS move required to move the dashed handleslide marks and the marked points to the right of the rightmost crossing of $\beta_j$ equals the identity matrix. 
        \item The matrix $\mu(\beta')\pi$ is upper triangular.
    \end{enumerate}
    We use the notation $\mathfrak M(\mathfrak{m}) \coloneqq \mathfrak M(\mathfrak{m},w_0)$.
\end{definition}
\begin{remark}\label{rmk:mcs_morphism_are_mcf}
    The attentive reader will notice that points in $\mathfrak M(\mathfrak{m})$ in fact define Morse complex 2-families on the Legendrian surface associated to $\mathfrak{m}$ \cite{rutherford2018generating}. We explain this relation in \cref{sec:weaves_aug_sheaves} and use it to show that ``representations are sheaves.''
\end{remark}

\begin{example}
    See \cref{fig:morphism_weave} for an example of an A-form MCS sequence with trivial monodromy associated to a morphism $\sigma_1\sigma_2\sigma_1^2\sigma_2 \to \Delta$ in $\mathfrak{B}_3$.

   \begin{figure}[!htb]
        \centering
        \includegraphics{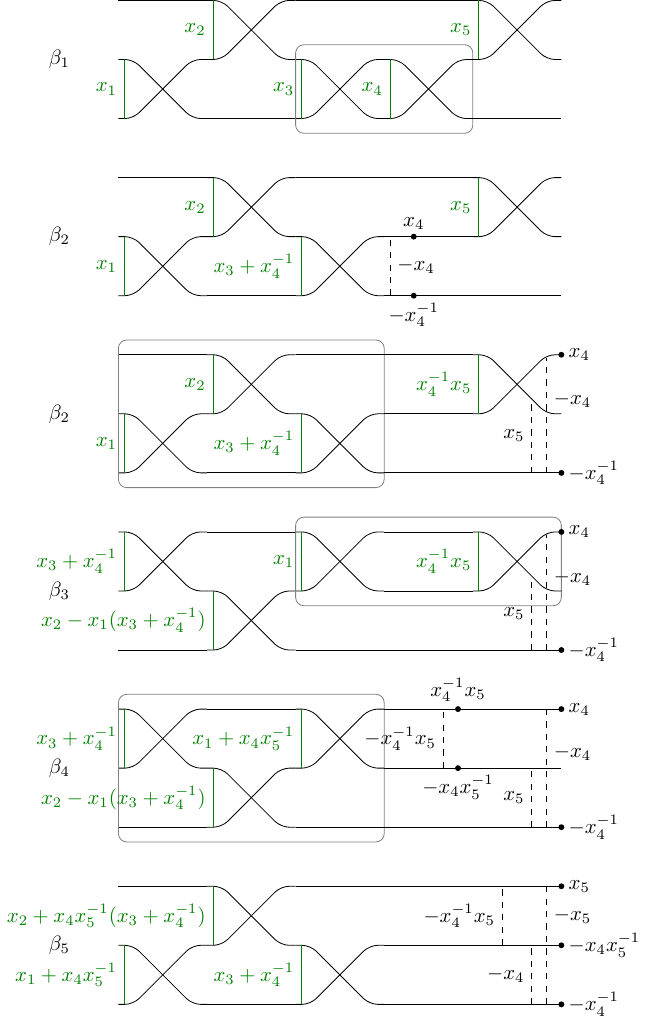}
        \caption{The associated A-form MCS sequence with trivial monodromy for the morphism $\sigma_1 \sigma_2 \sigma_1^2 \sigma_2 \rightarrow \Delta$ with braid sequence $(\beta_1=\sigma_1 \sigma_2 \sigma_1^2 \sigma_2,\beta_2 \ldots, \beta_5=\Delta)$.}
        \label{fig:morphism_weave}
    \end{figure}
\end{example}

\begin{lemma}\label{lma:functoriality}
    Let $\mathfrak{m} \colon \beta \to \beta'$ be a morphism in $\mathfrak{B}_n^\cup$ that has $c$ cup moves, $t$ trivalent moves, and no cap moves.
    \begin{enumerate}
        \item There is an injective map $\phi_{\mathfrak{m}} \colon \mathfrak M(\mathfrak{m}) \to X(\beta)$.
        \item There is an isomorphism $\mathfrak M(\mathfrak{m}) \cong \C^c \times (\C^\ast)^t \times X(\beta')$.
        \item If $\beta'=\Delta$, there is an injective composition
        \[
        \phi_{\mathfrak{m}} \colon \C^c \times (\C^\ast)^t \overset{\cong}{\longleftarrow} \mathfrak M(\mathfrak{m}) \longrightarrow X(\beta).
        \]
    \end{enumerate}
\end{lemma}

\begin{proof}
    \begin{enumerate}
        \item There are obvious projection maps
        \[
        \C^{\ell(\beta)} \overset{\pi}{\longleftarrow} \mathbb V^{\mathfrak{m}} \overset{\pi'}{\longrightarrow} \C^{\ell(\beta')}
        \]
        defined by picking out the variables corresponding to the handleslide marks of $A(\beta)$ and $A(\beta')$ that correspond to crossings of $\beta$ and $\beta'$, respectively. Associate to $A(\beta)$ and $A(\beta')$ the corresponding monodromies $\mu(\beta)$ and $\mu(\beta')$, respectively.
        
        We define a map $\mathfrak M(\mathfrak{m}) \to X(\beta)$ by showing that the image of $\mathfrak M(\mathfrak{m}) \hookrightarrow \mathbb V^{\mathfrak{m}} \xrightarrow{\pi} \C^{\ell(\beta)}$ belongs to $X(\beta)\hookrightarrow \C^{\ell(\beta)}$. By \cref{lma:monodromy_comp_prod} and the assumption that the monodromy of each MCS (braid) move associated with $A(\mathfrak{m})$ is equal to the identity, it follows that the total monodromy of the composition is equal to the identity, so
        \[
        \mu(\beta)\left(V\mu(\beta')\right)^{-1} = I \Longleftrightarrow V^{-1}\mu(\beta) = \mu(\beta').
        \]
        Here $V$ is the monodromy associated with all the residual dashed handleslide marks and marked points that are located the right of the rightmost crossing of $\beta$; it is, in particular, upper triangular. Therefore the condition that $\mu(\beta')w_0$ is upper triangular implies that $V^{-1}\mu(\beta)w_0$ is upper triangular and, since $V^{-1}$ is upper triangular, it implies that $\mu(\beta)w_0$ is upper triangular. Thus the image of $\mathfrak M(\mathfrak{m})$ under the map $\pi$ belongs to $X(\beta)$, yielding the map
        \[
        \mathfrak M(\mathfrak{m}) \longrightarrow X(\beta).
        \]
        For injectivity we consider each MCS (braid) move of $A(\mathfrak{m})$ which, by assumption, have trivial monodromy. If $A(\beta_{j-1})$ is related to $A(\beta_j)$ by an MCS braid move, it follows from \cref{lma:monodromy_trivalent_etc_id} that the handleslide mark variables in $A(\beta_{j-1})$ uniquely determine the handleslide marks of $A(\beta_j)$ and the variables associated with the dashed handleslide marks and the marked points. For any other MCS move that involves moving dashed handleslide marks and marked points to the right in $A(\beta_j)$, it follows from \cref{cor:mcs_matrix} that the variables before the move uniquely determine the variables after the move. Thus, points in $X(\beta)$ uniquely determine points in $\mathfrak{M}(\mathfrak{m})$.
        \item Consider the sequence of braids $(\beta \eqqcolon \beta_1,\ldots,\beta_q \coloneqq \beta')$ and go through the MCS (braid) moves in the associated A-form MCS sequence in the reverse order. Similar to (1), there is a map $\mathfrak M(\mathfrak{m}) \to X(\beta')$. Picking a point in $X(\beta')$ is equivalent to picking values of the handleslide marks of $A(\beta_q)$ corresponding to crossings of $\beta_q$ subject to the condition that $\mu(\beta_q)w_0$ is upper triangular. 
        
        We proceed backwards along the MCS (braid) moves determined by $\mathfrak{m}$. It is clear from \cref{cor:mcs_matrix,lma:monodromy_trivalent_etc_id} that moving through every MCS (braid) move (except for the MCS trivalent and MCS cup moves) backwards, uniquely determines the variables.
        The only two cases left to consider are the MCS trivalent and cup moves. First, for the MCS trivalent move we choose $y\in \C^\ast$ arbitrarily, create canceling dashed handleslide marks and marked points, and move through the MCS trivalent move backwards as depicted in \cref{fig:mcs_trivalent_backwards}. Afterwards, move the newly created dashed handleslide mark and marked points all the way to the right of the rightmost crossing of $\beta_j$ (or until they merge with other handleslide marks).
        \begin{figure}[!htb]
            \centering
            \includegraphics{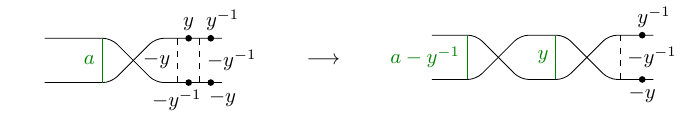}
            \caption{Choosing $y\in \C^\ast$ arbitrarily determines the handleslide marks before the MCS trivalent move.}
            \label{fig:mcs_trivalent_backwards}
        \end{figure}
        Thus, the handleslide marks after the backwards MCS trivalent move are determined by the condition that the monodromy is the identity; see \cref{lma:monodromy_trivalent_etc_id}.
        
        Similarly, for the MCS cup move, we choose $x\in \C$ arbitrarily, create canceling dashed handleslide marks, and move through the MCS cup move backwards, as depicted in \cref{fig:mcs_cup_backwards}. Afterwards, move the newly created dashed handleslide mark all the way to the right of the rightmost crossing of $\beta_j$ (or until they merge with other handleslide marks).
        \begin{figure}[!htb]
            \centering
            \includegraphics{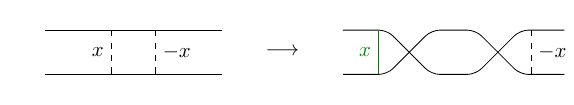}
            \caption{Choosing $x\in \C$ arbitrarily determines the handleslide marks before the MCS cup move.}
            \label{fig:mcs_cup_backwards}
        \end{figure}
        Thus, the handleslide marks after the backwards MCS cup move are determined by the condition that the monodromy is the identity; see \cref{lma:monodromy_trivalent_etc_id}. To summarize, every MCS trivalent move contributes a $\C^\ast$-factor and every MCS cup move contributes a $\C$-factor, finishing the proof. 
        \item For $\mathfrak m \colon \beta \to \Delta$, since $X(\Delta)$ may be identified with the origin in $\C^{\binom n2}$, the map $\phi_\mathfrak{m}$ is defined as the composition
        \[
        \C^c \times (\C^\ast)^t \longrightarrow \C^c \times (\C^\ast)^t \times \{0\} \longrightarrow X(\beta)
        \]
        and is injective by (1). \qedhere
    \end{enumerate}
\end{proof}

\begin{theorem}
    Let $n \in \Z_{\geq 1}$. There is a functor $\mathfrak M \colon \mathfrak{B}_n \to \mathfrak{C}$.
\end{theorem}
\begin{proof}
    On objects we define $\mathfrak M(\beta) \coloneqq X(\beta)$. On a morphism $\mathfrak{m} \colon \beta \to \beta'$, let $\mathfrak M(\mathfrak{m})$ be the affine algebraic subvariety of $\mathbb{V}^{\mathfrak{m}}$ defined in \cref{dfn:variety_of_morphism}. By \cref{lma:functoriality}, it follows that $\mathfrak M(\mathfrak{m})$ defines an algebraic correspondence 
    \[
    X(\beta) \longleftarrow \mathfrak M(\mathfrak{m}) \longrightarrow X(\beta').
    \]
    Given two morphisms $\mathfrak{m} \colon \beta \to \beta'$ and $\mathfrak{r}' \colon \beta' \to \beta''$, their composition $\mathfrak{m}' \circ \mathfrak{m}$ is defined to be the concatenation of the corresponding sequences of braids and braid moves. Since the monodromy conditions in $\mathfrak M(\mathfrak{m})$ and $\mathfrak M(\mathfrak{m}')$ are independent of each other, we see that there are natural projection maps
    \[
    \mathfrak M(\mathfrak{m}) \longleftarrow \mathfrak M(\mathfrak{m}'\circ \mathfrak{m}) \longrightarrow \mathfrak M(\mathfrak{m}').
    \]
    Recall from the proof of \cref{lma:functoriality}(1) that there are projection maps $\mathfrak M(\mathfrak{m}) \to \C^{\ell(\beta')}$ and $\mathfrak M(\mathfrak{m}') \to \C^{\ell(\beta')}$, where $\beta'$ is the last and first braid in the braid sequences associated with $\mathfrak{m}$ and $\mathfrak{m}'$, respectively. Therefore we have a commutative diagram
    \[
    \begin{tikzcd}[sep=scriptsize]
    \mathfrak M(\mathfrak{m}' \circ \mathfrak{m}) \rar \dar & \mathfrak M(\mathfrak{m}') \dar{\pi'} \\
    \mathfrak M(\mathfrak{m}) \rar{\pi'} & \C^{\ell(\beta')}
    \end{tikzcd},
    \]
    which is easily seen to be a pullback square.
\end{proof}
\begin{theorem}\label{thm:functors_from_mcs}
    Let $n\in \Z_{\geq 1}$. The following diagram commutes.
    \[
    \begin{tikzcd}[sep=scriptsize]
        \mathfrak{B}_n \ar[dr,swap,"\mathfrak M"] \ar[rr,"\mathfrak{A}"] && \mathfrak{W}_n \ar[dl,"\mathfrak{X}"] \\
        & \mathfrak{C}&
    \end{tikzcd}
    \]
\end{theorem}
\begin{proof}
    It is clear that the diagram is commutative on objects. Given a morphism $\mathfrak{m} \colon \beta \to \beta'$ in $\mathfrak{B}_n$, we need to show $\mathfrak M(\mathfrak{m}) = \mathfrak{X}(\mathfrak{A}(\mathfrak{m}))$, where $\mathfrak{A}(\mathfrak{m})$ is the algebraic weave defined by $\mathfrak m$ (see \cref{thm:ruling_to_weaves}). By definition of $\mathfrak{A}$, the collection of all handleslide marks in each of the A-form MCSs in $(A(\beta_1),\ldots, A(\beta_q))$ correspond to all weave segment variables assigned to $\mathfrak{A}(\mathfrak{m})$ (see \cref{sec:weave_decomp}). Furthermore, all dashed handleslide marks and marked points created by MCS trivalent and MCS cup moves (see \cref{dfn:mcs_braid_moves}) correspond to all dashed segment variables associated to $\mathfrak{A}(\mathfrak{m})$. Therefore $\mathbb V^{\mathfrak{m}} = \mathbb V^{\mathfrak{A}(\mathfrak{m})}$. Comparing \cref{lma:monodromy_trivalent_etc_id} with \cref{lma:weave_monodromy_vertex} we see that the monodromy conditions associated to MCS distant crossings, hexavalent, trivalent, and cup moves agree. Comparing \cref{cor:mcs_matrix} with \cref{lma:weave_monodromy_vertex} it follows that the monodromy conditions induced by the other MCS moves involving moving dashed handleslide marks and marked points agree with the monodromy around virtual vertices in the algebraic weave $\mathfrak{A}(\mathfrak{m})$. Thus showing $\mathfrak M(\mathfrak{m}) = \mathfrak{X}(\mathfrak{A}(\mathfrak{m}))$.
\end{proof}

The following result is completely analogous to the existence of the weave decomposition \cref{thm:weave_decomp} of the braid variety $X(\beta)$. Its proof is almost identical to that of \cref{thm:weave_decomp}.
\begin{proposition}\label{prop:mcs_cat_decompos}
    Let $\beta \in \Br_n^+$ and let $\mathfrak{m}\colon \beta \to \Delta$ be a morphism in $\mathfrak{B}_n$. There exists a tuple of morphisms $(\mathfrak{m}=\mathfrak{m}_1,\mathfrak{m}_2,\dots,\mathfrak{m}_k)$ such that
    \begin{itemize}
        \item $X(\beta) = \bigcup_{i=1}^k \mathfrak{M}({\mathfrak{m}_i})$,
        \item $\mathfrak{M}({\mathfrak{m}_i}) \subset X(\beta)$ are pairwise disjoint for $i\in \{1,\ldots,k\}$, and
        \item $\mathfrak M(\mathfrak{m}) \cong (\C^\ast)^{\ell(\beta)-\binom n2}$ is the unique piece of maximal dimension.
    \end{itemize}
    Such a tuple $(\mathfrak{m}_1,\dots,\mathfrak{m}_k)$ is called a \emph{decomposing tuple} for $X(\beta)$.
    \qed
\end{proposition}
\begin{remark}\label{rmk:decomposing_tuples_are_the_same}
    We note that \cref{thm:functors_from_mcs} implies that $(\mathfrak{m}_1,\ldots,\mathfrak{m}_k)$ is a decomposing tuple for $X(\beta)$ if and only if $(\mathfrak{A}(\mathfrak{m}_1),\ldots,\mathfrak{A}(\mathfrak{m}_k))$ is a decomposing tuple for $X(\beta)$ (see \cref{thm:weave_decomp}).
\end{remark}

\subsubsection{Framed Morse complex sequences and their monodromy varieties}
In \cref{thm:ruling_to_weaves} (and \cref{rmk:ruling_caps}) we showed that (a widening of) the braid category is equivalent to the weave category. In \cref{thm:functors_from_mcs} we show that for any morphism $\mathfrak{m} \colon \beta \to \beta'$ in $\mathfrak{B}_n$ and its corresponding algebraic weave $\w \coloneqq \mathfrak{A}(\mathfrak{m})$, their monodromy varieties $\mathfrak M(\mathfrak{m})$ and $\mathfrak{X}(\w)$ agree. Moreover, the parametrizations $\mathfrak M(\mathfrak{m}) \cong \C^c \times (\C^\ast)^t$ and $\mathfrak{X}(\w) \cong \C^c \times (\C^\ast)^t$ are identified via the identity map.

In this subsection, we consider the MCS analog of framed algebraic weaves, discussed in \cref{ssec:framed-weave}, and show that those yield other parametrizations $\mathfrak M(\mathfrak{m}) \cong \C^c \times (\C^\ast)^t$ and $\mathfrak{X}(\w) \cong \C^c \times (\C^\ast)^t$. Recall the definition of an A-form MCS in \cref{dfn:a_form_mcs}.

\begin{definition}[Framed A-form MCS]\label{dfn:framed_A-form_MCS}
     A \emph{framed A-form MCS} for a front diagram $D$ equipped with a Maslov potential is an A-form MCS without marked points on right cusps equipped with the following marked points: For each crossing between strands $i$ and $i+1$, there is a marked point with variable $-u^{-1}$ on strand $i$ and a marked point with variable $u$ on strand $i+1$, for some $u\in R^\ast$; Both marked points are located immediately to the right of the crossing, and there are no other marked points. 
    
\end{definition}

\begin{remark}
    The difference between a framed A-form MCS and an unframed one (\cref{dfn:a_form_mcs}) is the difference in the set of marked points, see \cref{fig:framed_mcs}.
\end{remark}

\begin{figure}[!htb]
    \centering
    \includegraphics{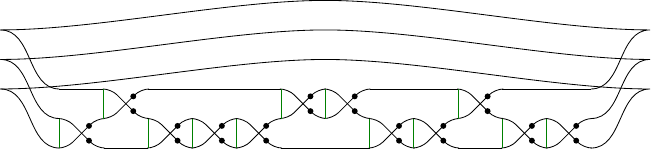}
    
    \vspace{2mm}
    
    \includegraphics{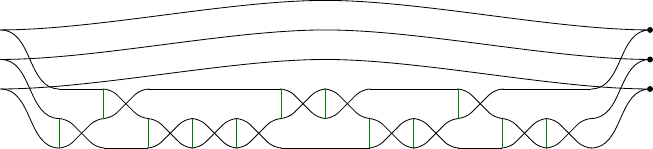}
    \caption{Top: A framed A-form MCS. Bottom: An unframed A-form MCS. All variables at handleslide marks and marked points are omitted.}
    \label{fig:framed_mcs}
\end{figure}

In analogy with \cref{sec:monodromy_mcs} we now define the monodromy of the associated framed A-form MCS sequence to a morphism $\mathfrak{m} \colon \beta \to \beta'$. To start, there are framed versions of the MCS braid moves defined in \cref{dfn:mcs_braid_moves}.

\begin{definition}[Framed MCS braid moves]
The \emph{framed MCS braid moves} are the following:
    \begin{enumerate}
        \item A \emph{framed MCS distant crossings move} is the local modification of the front diagram of a framed MCS defined by \cref{fig:mcs_distant_framed}.
        \item A \emph{framed MCS hexavalent move} is the local modification of the front diagram of a framed MCS defined by \cref{fig:mcs_hexavalent_framed}.
        \item A \emph{framed MCS trivalent move} is the local modification of the front diagram of a framed MCS defined by \cref{fig:mcs_trivalent_framed}.
        \item A \emph{framed MCS cup move} is the local modification of the front diagram of a framed MCS defined by \cref{fig:mcs_cup_framed}.
    \end{enumerate}
\end{definition}

\begin{figure}[!htb]
    \centering
    \begin{subfigure}{\textwidth}
        \centering
        \includegraphics{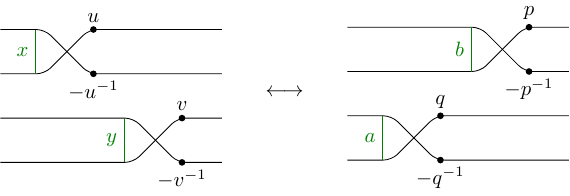}
        \caption{Framed MCS distant crossings move}\label{fig:mcs_distant_framed}
    \end{subfigure}
    \begin{subfigure}{\textwidth}
        \centering
        \includegraphics{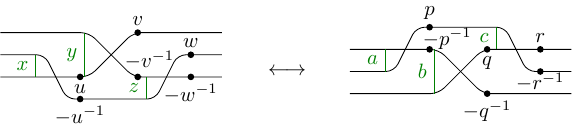}
        \caption{Framed MCS hexavalent move}\label{fig:mcs_hexavalent_framed}
    \end{subfigure}
    \begin{subfigure}{\textwidth}
        \centering
        \includegraphics{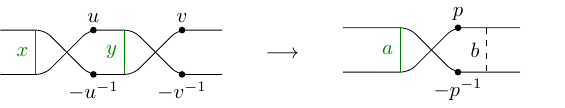}
        \caption{Framed MCS trivalent move}\label{fig:mcs_trivalent_framed}
    \end{subfigure}
    \begin{subfigure}{\textwidth}
        \centering
        \includegraphics{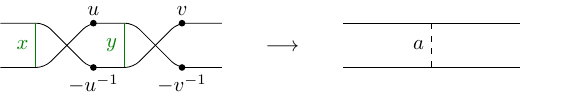}
        \caption{Framed MCS cup move}\label{fig:mcs_cup_framed}
    \end{subfigure}
    \caption{Framed MCS braid moves.}
    \label{fig:mcs_moves_framed}
\end{figure}

\begin{definition}[Associated framed A-form MCS sequence]
    Let $\mathfrak{m} \colon \beta \to \beta'$ be a morphism in $\mathfrak{B}_n$ with braid sequence $(\beta_1,\ldots,\beta_q)$. The \emph{associated framed A-form MCS sequence} is the sequence $A(\mathfrak{m}) \coloneqq (A(\beta_1),\ldots,A(\beta_q))$ of formal framed A-form MCSs such that for each $j\in \{2,\ldots,q\}$, exactly one of the following holds:
        \begin{enumerate}
             \item $A(\beta_{j-1})$ is related to $A(\beta_j)$ via a framed MCS distant crossings move or a framed MCS hexavalent move.
            \item $A(\beta_{j-1})$ is related to $A(\beta_j)$ via a framed MCS trivalent move or a framed MCS cup move followed by some number of MCS moves that take the dashed handleslide marks (but not the marked points) to the right of the rightmost crossing of $\beta_j$.
        \end{enumerate}
\end{definition}

Recall the definitions of the matrices $S_{i,j}(z)$, $D_i(t)$, $P_i$, and $B_i(z)$ from \cref{notn:matrices_mcs_rep} as well as the definition of $\chi_i(u)$ from \cref{not:xi(u)}. Also, recall the definition of the non-negative integers $L$ and $R$ defined in \eqref{eq:total_sp_braids}. Let $\mathfrak{m} \colon \beta \to \beta'$ be a morphism in $\mathfrak{B}_n$ and define
\[
    \mathbb V^{\mathfrak m}_\textit{fr} \coloneqq (\C \times \C^\ast)^{L} \times \left(\mathbb C^{\binom{n}2}\right)^{R}.
\]

For any positive braid $\beta$ and a framed A-form MCS associated to $\La(\beta\Delta)$, the monodromy of $\beta = \sigma_{\ell_1} \cdots \sigma_{\ell_{r}}$ is defined by
\begin{align*}
    \mu(\beta) &\coloneqq \chi_{\ell_r}(u_r)P_{\ell_{r}}S_{\ell_{r},\ell_{r}+1}(z_{r}) \cdots \chi_{\ell_1}(u_1)P_{\ell_{1}}S_{\ell_{1},\ell_{1}+1}(z_{1}) \\
    &= \chi_{\ell_r}(u_r)B_{\ell_{r}}(z_{r}) \cdots \chi_{\ell_1}(u_1)B_{\ell_1}(z_1).
\end{align*}

\begin{definition}[Framed variety of $\mathfrak{m}$]
    Let $\beta,\beta'\in \Br_n^+$ and suppose $\mathfrak{m} \colon \beta \to \beta'$ is a morphism in $\mathfrak{B}_n$. Let $\pi \in S_n$ be represented by a permutation matrix in $\GL(n,\C)$. We define $\mathfrak M_\textit{fr}(\mathfrak{m},\pi)$ to be the algebraic affine subvariety of $\mathbb V_\textit{fr}^{\mathfrak{m}}$ cut out by the following three conditions:
    \begin{enumerate}
        \item In $A(\mathfrak{m})$, the monodromy of each framed MCS distant crossings, hexavalent, trivalent, and cup moves relating $A(\beta_j)$ and $A(\beta_{j-1})$ for each $j$, equals the identity matrix; see \cref{fig:mcs_moves_framed_monodromyless}.
        \item The monodromy of each framed MCS move required to move the dashed handleslide marks to the right of the rightmost crossing of $\beta_j$, is equal to the identity matrix. 
        \item The matrix $\mu(\beta')\pi$ is upper triangular.
    \end{enumerate}
    We use the notation $\mathfrak M_\textit{fr}(\mathfrak{m}) \coloneqq \mathfrak M_\textit{fr}(\mathfrak{m},w_0)$.
\end{definition}

\begin{figure}[!htb]
    \centering
    \begin{subfigure}{\textwidth}
        \centering
        \includegraphics{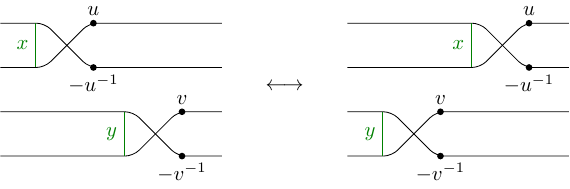}
        \caption{}\label{fig:mcs_distant_framed_monoless}
    \end{subfigure}
    \begin{subfigure}{\textwidth}
        \centering
        \includegraphics{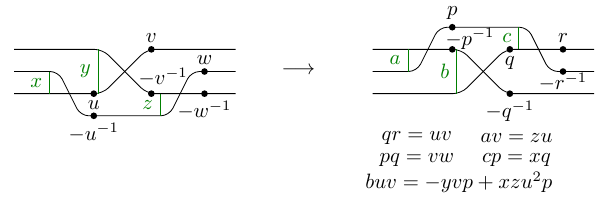}
        \caption{}\label{fig:mcs_hexavalent_framed_monoless1}
    \end{subfigure}
    \begin{subfigure}{\textwidth}
        \centering
        \includegraphics{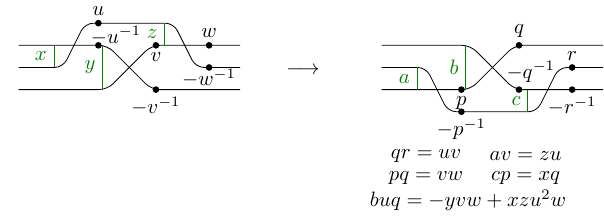}
        \caption{}\label{fig:mcs_hexavalent_framed_monoless2}
    \end{subfigure}
    \begin{subfigure}{\textwidth}
        \centering
        \hspace*{1cm}
        \includegraphics{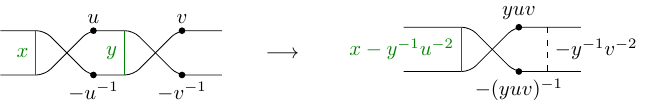}
        \caption{}\label{fig:mcs_trivalent_framed_monoless}
    \end{subfigure}
    \begin{subfigure}{\textwidth}
        \centering
        \includegraphics{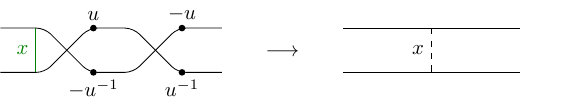}
        \caption{}\label{fig:mcs_cup_framed_monoless}
    \end{subfigure}
    \caption{Framed MCS braid moves with trivial monodromy.} 
    
    \label{fig:mcs_moves_framed_monodromyless}
\end{figure}

The proof of the following lemma is similar to the proof of \cref{lem:framed-weave=weave}.

\begin{lemma}\label{lem:framed_vs_unframed}
    Let $\mathfrak{m} \colon \beta \to \beta'$ be a morphism in $\mathfrak{B}_n$ that involves $c$ cup moves and $t$ trivalent moves. There is an isomorphism  $\mathfrak M_\textit{fr}(\mathfrak{m}) \cong \mathfrak M(\mathfrak{m})$ that fits into the commutative diagram
    \[\begin{tikzcd}
        X_\textit{fr}(\beta') \ar[d, "\cong"] & \mathfrak M_\textit{fr}(\mathfrak{m}) \ar[l,"\pi"] \ar[r] \ar[d, "\cong"] & 
        X(\beta) \ar[d, "="] \\
        X(\beta') & \mathfrak M(\mathfrak{m}) \ar[l, "\pi"] \ar[r] & X(\beta).
    \end{tikzcd}\]
    \qed
\end{lemma}

    A consequence of \cref{lem:framed_vs_unframed} is the framed version of \cref{thm:functors_from_mcs}. The proof is completely analogous, and is omitted.

\begin{theorem}\label{thm:frame_functors_from_mcs}
    Let $n \in \Z_{\geq 1}$. For any $\beta \in \Br_n^+$, there is a functor $\mathfrak M_\textit{fr}: \mathfrak{B}_n \to \mathfrak{C}$ such that $\mathfrak M_\textit{fr}(\beta) = \mathfrak M(\beta) = X(\beta)$ and a functor $\mathfrak X_\textit{fr}: \mathfrak{W}_n \to \mathfrak{C}$ such that $\mathfrak X_\textit{fr}(\beta) = \mathfrak X(\beta) = X(\beta)$. In addition, the following diagram commutes:
    \[\begin{tikzcd}
    \mathfrak B_n \ar[dr, "\mathfrak M_\textit{fr} \;\;\;" below] \ar[rr, "\mathfrak A"] & & \mathfrak W_n \ar[dl, "\mathfrak X_\textit{fr}"] \\
    & \mathfrak{C} &
    \end{tikzcd}\]
    \qed
\end{theorem}
\subsection{Category of normal rulings and decompositions of braid varieties}\label{sec:normal_rulings_cat}
    We show how to associate to a normal ruling $\rho$ of $\La(\beta\Delta)$ a certain set of morphisms in the braid category defined in \cref{dfn:braid_category}. Let $\beta = \sigma_{i_1} \cdots \sigma_{i_r} \in \Br_n^+$ with $\delta(\beta) = w_0$ and let $\La(\beta\Delta)$ denote the $(-1)$-closure of $\beta\Delta$.
    \begin{definition}[Top right crossings and direct connection]
        Let $\mathfrak{m} \colon \beta \to \beta'$ be a morphism in $\mathfrak{B}_n$ represented by the sequence $(\beta\eqqcolon\beta_1,\ldots,\beta_q\coloneqq\beta')$. Suppose $k\in \{1,\ldots,q\}$ and  that
        \[
        \beta_k = \gamma_Lx_1x_2\gamma_R \longrightarrow \beta_{k+1}
        \]
        is either a trivalent or cup move, where $x_1=x_2=\sigma_i$ for some $i$.
        \begin{enumerate}
            \item The letter $x_2$ is called the \emph{top right crossing} of the trivalent or cup move.
            \item The top right crossing $x_2$ is \emph{directly connected to the top of $\mathfrak{m}$} if $x_2$ is not involved in any braid move among the sequence of braids $(\beta_1,\ldots,\beta_{k-1})$.
        \end{enumerate}
    \end{definition}
    
    \begin{definition}[Ruling category]\label{dfn:ruling_category}
       Let $n\in \Z_{\geq 1}$. The \emph{ruling category} $\mathfrak{R}_n$ is defined to be the wide subcategory of $\mathfrak{B}_n^\cup$ only consisting of those morphisms $\mathfrak{r} = (\beta_1,\ldots,\beta_q) \colon \beta \to \beta'$ with $\mathfrak{r}$ such that there exists a normal ruling $\rho$ of $\La(\beta\Delta)$ such that:
       \begin{enumerate}
           \item For any trivalent move, the top right crossing is directly connected to a crossing in $\beta_1$ corresponding to a switch of $\rho$.
           \item For any cup move, the top right crossing is directly connected to $\beta_1$ corresponding to a departure of $\rho$.
       \end{enumerate}
       In this case we say that $\rho$ is the underlying normal ruling for the morphism $\mathfrak{r}$.
    \end{definition}
    Recall that for $\beta\in\Br_n^+$ with $\delta(\beta) = w_0$, there exists a unique normal ruling with maximal number of switches; see \cref{prop:max_ruling}.

    \begin{proposition}\label{prop:normal_ruling_inductive}
        Let $\beta \in \Br_n^+$ be such that $\delta(\beta) = w_0$, and suppose that $\rho$ is any normal ruling for $\La(\beta\Delta)$. There exists a morphism $\mathfrak{r}_\rho \colon \beta \to \Delta$ in $\mathfrak{R}_n$ whose underlying normal ruling is $\rho$. In fact, $\mathfrak{A}(\mathfrak{r}_\rho)$ is a right simplifying weave and, for the maximal ruling, $\mathfrak{A}(\mathfrak{r}_\rho)$ is a right inductive Demazure weave.
    \end{proposition}
    \begin{proof}
        Let $\beta_1=\beta$. We define the morphism $\mathfrak{r}_\rho \colon \beta \to \Delta$ inductively as follows: Given $\beta_\ell\in\Br_n^+$, let $x^L_\ell$ denote the leftmost crossing of the braid $\beta_\ell$ that is not a return for the underlying normal ruling $\rho$. Then there exist $\beta_\ell^L,\beta_\ell^R\in\Br_n^+$ such that $\beta_\ell = \beta_\ell^Lx^L_\ell\beta_\ell^R$ and there exists $i_\ell\in\{1,\ldots,n\}$ such that $x^L_\ell=\sigma_{i_\ell}$. Since $x^L_\ell$ is not a return in $\rho$, (if needed) we perform hexavalent and distant crossings moves on the $\beta_\ell^L$ portion of $\beta_\ell$ to get $\beta_\ell \to \cdots \to \beta'_{\ell}$ such that $\beta'_{\ell} = \gamma_{\ell}\sigma_{i_\ell}x_{\ell}^L\beta_\ell^R$ for some $\gamma_\ell\in\Br_n^+$. If $x^L_\ell$ is a switch in $\rho$, then perform a trivalent move $(\sigma_{i_\ell})^2 \to \sigma_{i_\ell}$ and let $\beta_{\ell+1}=\gamma_{\ell}\sigma_{i_\ell}\beta_\ell^R$. If $x^L_\ell$ is a departure in $\rho$, then perform a cup move $(\sigma_{i_\ell})^2 \to 1$ and let $\beta_{\ell+1}=\gamma_{\ell}\beta_\ell^R$. Repeat this until $\beta_{\ell+1}=\Delta$. This defines the morphism $\mathfrak{r}_\rho \colon \beta \to \Delta$.

        By construction, we see that $\mathfrak{A}(\mathfrak{r}_\rho)$ is a right simplifying weave (as defined in \cref{dfn:inductive_simplifying_weave}). For the maximal normal ruling of $\Lambda(\beta\Delta)$, by \cref{prop:max_ruling}, none of the crossings in the braid $\beta$ are departures. Thus, from \cref{dfn:ruling_category}, it follows immediately that $\mathfrak{A}(\mathfrak{r}_\rho)$ is a right inductive Demazure weave.
    \end{proof}
    \begin{definition}[Right inductive morphism]\label{def:right_inductive_morphism}
        Let $\beta \in \Br_n^+$ and suppose that $\rho$ is a normal ruling for $\La(\beta\Delta)$. A morphism $\mathfrak{r}_\rho \colon \beta \to \Delta$ defined as in \cref{prop:normal_ruling_inductive} is called a \emph{right inductive morphism} corresponding to $\rho$.
    \end{definition}
    \begin{lemma}\label{lma:rulings_right_induct_weaves}
        Let $\beta\in \Br_n^+$ be such that $\delta(\beta) = w_0$. The assignment $\rho \mapsto \mathfrak{A}(\mathfrak{r}_\rho)$ defines a bijection from the set of normal rulings of $\La(\beta\Delta)$ to the set of inductive equivalence classes of right simplifying weaves $\beta \to \Delta$.
    \end{lemma}
  
    \begin{proof}  
    Suppose $\beta = \sigma_{i_1} \cdots \sigma_{i_r}$ and denote the crossing of $\La(\beta\Delta)$ corresponding to $\sigma_{i_k}$ by $c_k$.  Let $\w\colon \beta\to\Delta$ be a right simplifying weave. We define a normal ruling $\rho_\w$ of $\La(\beta\Delta)$ as follows to define the map $\w \mapsto \rho_\w$ from right simplifying weaves to normal rulings of $\La(\beta\Delta)$. Denote the topmost weave segments of $\w$ by $w_1,\ldots,w_r$, where $w_k$ corresponds to $\sigma_{i_k}$ (and thus to the crossing $c_k$ of $\La(\beta\Delta)$). Since $\w$ is a right simplifying weave, for every trivalent or cup, the northeast strand is $w_k$ for some $k\in \{1,\ldots,r\}$. 
    If $k\in \{1,\ldots,r\}$ such that $w_k$ is the northeast strand in a trivalent vertex, we declare $c_k$ to be a switch, if $w_k$ is the right strand in a cup, we declare $c_k$ to be a departure, and for all other $w_k$, we declare $c_k$ to be a return. We show that $\rho_\w$ is a normal ruling.
    
    Given a ruling of $\La(\beta\Delta)$ note that all crossings to the left of $\beta$ must be departures, all crossings to the right of $\Delta$ must be returns, and all crossings in $\Delta$ must be departures. Thus, each ruling of $\La(\beta\Delta)$ specifies a unique sequence $(x_1,\ldots,x_r)$ where $x_i \in \{\text{s},\text{d},\text{r}\}$ is a label specifying a switch, departure, or return. We say that a crossing in a ruling of $\La(\beta\Delta)$ is a \emph{normal} switch, departure, or return, respectively, if the two ruling disks incident to the crossing are of the form of the rightmost diagram in each of \cref{fig:ruling_switches,fig:ruling_deps,fig:ruling_rets}, respectively. To show that the ruling $\rho_\w$ is normal, it suffices to show that every crossing $c_k$ is normal. 
    
    Enumerate the strands in any braid by enumerating the strands from bottom to top on the left of the braid. This labeling of the strands of $\beta$ gives an enumeration of the strands in $\La(\beta\Delta)$. We prove that every crossing $c_k$ is normal by strong induction on $k$. Assume that the claim is true up to and including the crossing $c_{k-1}$ of $\La(\beta\Delta)$ and assume $\sigma_{i_k}$ introduces an intersection between the $i$-th and $j$-th strands in $\beta$. Recall, from \cref{dfn:inductive_simplifying_weave}, that right simplifying weaves are constructed by inductively defining braids $\beta_j'$ and weaves $\w(\beta^{\leq j}):\beta^{\leq j}\to\beta_j'$ (where $\beta_j'$ is not necessarily $\delta(\beta^{\leq j})$).
    
    As a result of the definition, we also have that $\beta_0'=1$, $\beta_1'=\sigma_{i_1}$, and $\beta_r'=\Delta$. 
    There are three cases for how $\w(\beta^{\leq k})$ was obtained from $\w(\beta^{\leq k-1})$ in $\w$.
    \begin{description}
        \item[Vertical edge] If $\w(\beta^{\leq k})$ is obtained from $\w(\beta^{\leq k-1})$ by taking the horizontal composition with the trivial weave on $\sigma_{i_k}$, then, by the definition of $\rho_\w$ above, we know $c_k$ is a return. From \cref{dfn:inductive_simplifying_weave}, we know that $\delta(\beta_{k-1}'\sigma_{i_k})=\delta(\beta_{k-1}')s_{i_k}$ in this case. Thus $\beta_{k-1}'$ is not equivalent to a braid with $\sigma_{i_k}$ appearing on the right. Hence, the $i$-th and $j$-th strands do not intersect in $\beta_{k-1}'$ and so the ruling disks incident to $c_k$ are interlaced to the left of $c_k$ (as in the rightmost diagram in \cref{fig:ruling_rets}). This implies that the return at $c_k$ is normal.
                
        \item[Trivalent] In this case, $\w(\beta^{\leq k})$ is obtained from $\w(\beta^{\leq k-1})$ by adding a trivalent vertex, which, by the definition of $\rho_\w$ above, means that $c_k$ must be a switch. From \cref{dfn:inductive_simplifying_weave}, we know that $\delta(\beta_{k-1}'\sigma_{i_{k}}) = \delta(\beta_{k}')$ in this case. Thus, the braid $\beta_{k-1}'$ is equivalent to a braid with $\sigma_{i_k}$ appearing on the right,  implying that the $i$-th and $j$-th strands intersect in $\beta_{k-1}' = \beta_k'$. Consequently, the ruling disks incident to $c_k$ are nested (as in the rightmost diagram in \cref{fig:ruling_switches}) just before the crossing $c_k$, meaning that $c_k$ is a normal switch.
        
        \item[Cup] In this case, $\w(\beta^{\leq k})$ is obtained from $\w(\beta^{\leq k-1})$ by adding a cup, which, by the definition of $\rho_\w$ above, means that $c_k$ must be a departure. From \cref{dfn:inductive_simplifying_weave}, we know that $\beta_k'\sigma_{i_k}=\beta_{k-1}'$ in this case. Since $\beta_{k-1}'$ is equivalent to a braid with $\sigma_{i_k}$ on the right, the $i$-th and $j$-th strands intersect in $\beta_{k-1}'$. Moreover, they do not intersect in $\beta_k'$, as the cup introduced at this step removes the $\sigma_{i_k}$ appearing on the right of the braid equivalent to $\beta_{k-1}'$. Consequently, the ruling disks incident to $c_k$ are nested (as in the rightmost diagram in \cref{fig:ruling_deps}) just before the crossing $c_k$, meaning that $c_k$ is a normal departure.
    \end{description}
    Thus all crossings $c_k$ are normal and so $\rho_\w$ is a normal ruling of $\La(\beta\Delta)$.
    
    The above map $\w \mapsto \rho_\w$ now descends to a well-defined map on the set of inductive equivalence classes, since the normal ruling $\rho_\w$, by construction, only depends on the sequence of trivalent vertices and cups in $\w$ (and is therefore invariant under right inductive weave equivalence). We now see, by construction, that $\w \mapsto \rho_\w$ defines an inverse to $\rho \mapsto \mathfrak{A}(\mathfrak{r}_\rho)$ defined in \cref{def:right_inductive_morphism}.
 \end{proof}

    \begin{corollary}\label{cor:rulings_decomposing_tuple}
       Let $\beta\in \Br_n^+$ such that $\delta(\beta) = w_0$, and let $(\rho_1,\ldots,\rho_k)$ be a tuple of all normal rulings of $\La(\beta\Delta)$. There exists a tuple $(\mathfrak{r}_{R,1},\ldots,\mathfrak{r}_{R,k})$ of right inductive morphisms $\mathfrak{r}_{R,i} \colon \beta \to \Delta$ in $\mathfrak{R}_n$ with underlying normal ruling $\rho_i$ that is a decomposing tuple for $X(\beta)$ in the sense of \cref{prop:mcs_cat_decompos}.
    \end{corollary}
    \begin{proof}
       By \cref{rmk:decomposing_tuples_are_the_same} it suffices to prove that there exists a tuple $(\mathfrak{r}_{1},\ldots,\mathfrak{r}_{k})$ of right inductive morphisms such that the tuple $(\mathfrak{A}(\mathfrak{r}_{1}),\ldots,\mathfrak{A}(\mathfrak{r}_{k}))$ is a decomposing tuple for $X(\beta)$. By \cref{lma:rulings_right_induct_weaves} $\rho \mapsto \mathfrak{A}(\mathfrak{r}_\rho)$ defines a bijection between normal rulings of $\La(\beta\Delta)$ and inductive equivalence classes of right simplifying weaves $\beta\to \Delta$. By \cite[Theorem 5.35]{CGGS1}, any tuple of representatives of inductive equivalence classes of right simplifying weaves $\beta\to \Delta$ is a decomposing tuple, and \cref{lma:decomp_tuple_indep_equiv} shows that the decomposition is independent of the choice of representatives, therefore finishing the proof.
    \end{proof}

    \begin{remark}\label{rmk:non_unique_weave}
        The collection of right simplifying weaves and right inductive morphisms constructed in \cref{lma:rulings_right_induct_weaves,cor:rulings_decomposing_tuple} is only unique up to inductive equivalence. Namely, we do not specify a sequence of hexavalent or distant crossing moves before adding a trivalent or cup move. We expect that up to weave equivalences in the sense of \cite[Section 4.2]{CGGS1}, the collection of right simplifying weaves corresponding to the normal rulings is unique, see \cref{qst:inductive_equiv_implies_weave_equiv}.
    \end{remark}
    
    \begin{example}\label{ex:right_inductive_mcs}
        Consider $\beta = (\sigma_1^2\sigma_2^2)^2 \in \Br_3^+$, and consider the normal ruling $\rho$ of $\La(\beta\Delta)$ as depicted in \cref{fig:normal_rulings_ex}.
        \begin{figure}[!htb]
            \centering
            \includegraphics{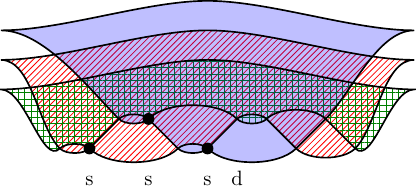}
            \caption{A normal ruling $\rho$ of the $(-1)$-closure of $(\sigma_1^2\sigma_2^2)^2\Delta$. The crossings marked with points represent the switches.}
            \label{fig:normal_rulings_ex}
        \end{figure}
        A right inductive morphism $\mathfrak{r} \colon (\sigma_1^2\sigma_2^2)^2 \to \Delta$ in the ruling category associated with the normal ruling $\rho$ is depicted in \cref{fig:ex_right_inductive_mcs}. A right simplifying weave $\mathfrak{A}(\mathfrak{r}_\rho)$ is depicted in \cref{fig:ex_right_inductive_weave}.
        \begin{figure}[!htb]
            \centering
            \includegraphics{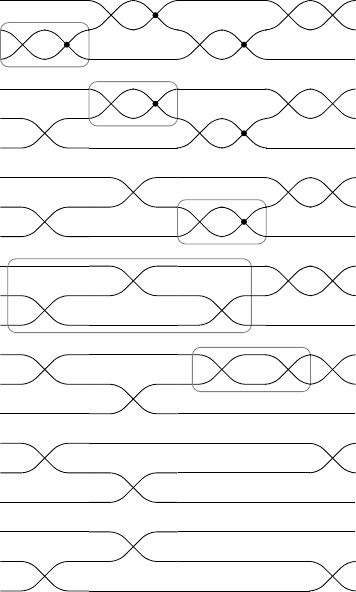}
            \caption{A right inductive morphism $\mathfrak{r}_\rho \colon \beta \to \Delta$ in the ruling category associated to the normal ruling $\rho$ of $\La(\beta\Delta)$ shown on the top braid by the crossings with marked points representing the switches. In each step, the gray rectangle indicates which MCS braid move is performed.}
            \label{fig:ex_right_inductive_mcs}
        \end{figure}
        \begin{figure}[!htb]
            \centering
            \includegraphics{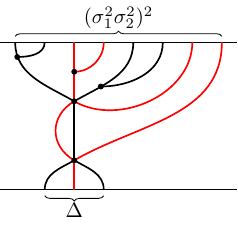}
            \caption{The weave $\mathfrak{A}(\mathfrak{r}_\rho)$ is a right simplifying weave on $\beta$.}
            \label{fig:ex_right_inductive_weave}
        \end{figure}
    \end{example}
        
    For a morphism $\mathfrak{r} \colon \beta \to \Delta$ in $\mathfrak{R}_n$ with $c$ cup moves, $t$ trivalent moves and underlying normal ruling $\rho$, we now define a map
    \[
    \psi_{\mathfrak{r}} \colon \C^{c} \times (\C^\ast)^{t} \longrightarrow \widehat{\MCS}{}^\rho(\La(\beta\Delta)).
    \]
    The idea is to perform the MCS (braid) moves in $A(\mathfrak{r})$ in reverse order as in the proof of \cref{lma:functoriality}(2), but for the backwards MCS trivalent move (\cref{fig:mcs_trivalent_backwards}) we do not move dashed handleslide marks to the right to leave behind an A-form MCS. Instead, keeping the dashed handleslide marks next to the crossings after trivalent moves leaves behind an SR-form MCS with respect to $\rho$; see \cref{fig:mcs_backwards_sr}.
    
    \begin{definition}[MCS-SR trivalent move]
        An \emph{MCS-SR trivalent move} is the local modification of the front diagram of an MCS defined by \cref{fig:mcs-sr_trivalent}.
    \end{definition}

    \begin{figure}[!htb]
        \centering
        \includegraphics{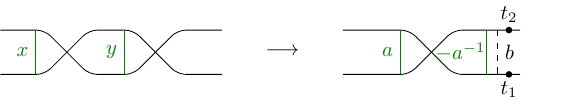}
        \caption{The MCS-SR trivalent move.}
        \label{fig:mcs-sr_trivalent}
    \end{figure}

    \begin{definition}[Associated SR-form MCS sequence]
        Let $\mathfrak{r} \colon \beta \to \beta'$ be a morphism in $\mathfrak{R}_n$ with braid sequence $(\beta_1,\ldots,\beta_q)$. The \emph{associated SR-form MCS sequence} is the sequence $SR(\mathfrak{r}) \coloneqq (SR(\beta_1),\ldots,SR(\beta_q))$ of formal SR-form MCSs such that for each $j\in \{2,\ldots,q\}$ exactly one of the following holds:
        \begin{enumerate}
            \item $SR(\beta_{j-1})$ is related to $SR(\beta_j)$ via an MCS distant crossings move or a MCS hexavalent move.
            \item $SR(\beta_{j-1})$ is related to $SR(\beta_j)$ via an MCS-SR trivalent move or a MCS cup move followed by some number of MCS moves that takes the (newly created) dashed handleslide marks and the marked points to the right of the rightmost crossing of $\beta_j$.
        \end{enumerate}
    \end{definition}

    \begin{lemma}
        The monodromy of each of the MCS-SR trivalent move is equal to the identity if and only if each handleslide mark outside of the local modification are equal to the corresponding handleslide mark after the move, and $a = x+y^{-1}$, $b = (x+y^{-1})^{-1}-y$, $t_1 = -y^{-1}$, and $t_2 = y$; see \cref{fig:mcs-sr_trivalent_monoless}. \qed
    \end{lemma}
    
    \begin{figure}[!htb]
        \centering
        \includegraphics{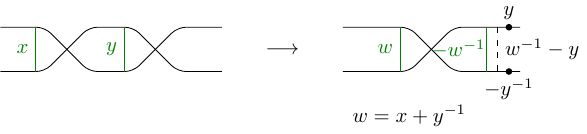}
        \caption{MCS-SR trivalent move with trivial monodromy.}
        \label{fig:mcs-sr_trivalent_monoless}
    \end{figure}

    \begin{figure}[!htb]
        \centering
        \hspace*{6mm}\includegraphics{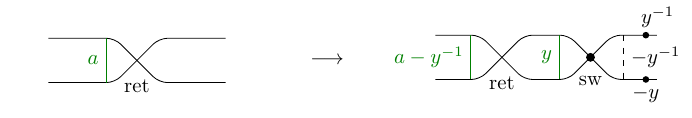}
        \caption{A backwards MCS-SR trivalent move with trivial monodromy used to define the map $\psi_\mathfrak{r}$.}
        \label{fig:mcs_backwards_sr}
    \end{figure}
    
    \begin{notation}
    Let $\mathfrak{r} \colon \beta \to \Delta$ be a morphism in $\mathfrak{R}_n$ with $c$ cups, $t$ trivalent moves and underlying ruling $\rho$. Recall the notation used in \eqref{eq:total_sp_braids}. We define
    \[
    \mathfrak{M}^\rho(\mathfrak{r}) \subset \mathbb{V}^\mathfrak{r} \coloneqq \C^{L} \times \left(\C^{\binom n2} \times (\C^\ast)^n \right)^R
    \]
    to be the algebraic subvariety of all $\C$-valued associated SR-form MCS sequences $SR(\mathfrak{r})$ with trivial monodromy.
    \end{notation}

    \begin{lemma}\label{lma:inj_map_sr-form}
        Let $\mathfrak{r} \colon \beta \to \Delta$ be a morphism in $\mathfrak{R}_n$ with $c$ cup moves, $t$ trivalent moves and underlying normal ruling $\rho$. There is an injective composition
        \[
        \psi_{\mathfrak{r}} \colon \C^{c} \times (\C^\ast)^{t} \overset{\cong}{\longleftarrow} \mathfrak{M}^\rho(\mathfrak{r}) \longrightarrow \widehat{\MCS}{}^\rho(\La(\beta\Delta)).
        \]
    \end{lemma}
    \begin{proof}
        The proof is completely analogous to the proof of \cref{lma:functoriality}, except that the definition of $\psi_{\mathfrak{r}}$ uses the backwards MCS-SR trivalent move (see \cref{fig:mcs_backwards_sr}) as opposed to the backwards MCS trivalent move use in the definition of $\phi_{\mathfrak{m}}$ in \cref{lma:functoriality}.
    \end{proof}

    \begin{lemma}\label{lma:SR-to-A-braid-category}
    Let $\mathfrak{r} \colon \beta \to \Delta$ be a morphism in $\mathfrak{R}_n$ with underlying ruling $\rho$. Then there is a bijection 
    \[\mathfrak{M}^\rho(\mathfrak{r}) \overset{\cong}{\longrightarrow} \mathfrak{M}(\mathfrak{r}).\]
    \end{lemma}
    \begin{proof}
        The map $\mathfrak{M}^\rho(\mathfrak{r}) \to \mathfrak{M}(\mathfrak{r})$ is defined by sending $SR(\mathfrak{r})$ to $A(\mathfrak{r})$ via the SR-to-A-form algorithm described in \cref{thm:one-to-one_A_SR,rmk:A_SR_algorithm}. It is a consequence of \cref{lma:monodromy_comp_prod,lma:mcs_braid_matrix} that the trivial monodromy condition is preserved.
    \end{proof}

    \begin{remark}\label{rmk:ruling_map_factor_SR_to_A}
        We can also define the map $\psi_\mathfrak{r}$ by first using the associated sequences of A-form MCS and then using the A-to-SR-form algorithm. In fact, $\psi_\mathfrak{r}$ factors as
        \[
            \mathfrak{M}(\mathfrak{r}) \longrightarrow \widehat{\MCS}{}^A(\La(\beta\Delta)) \overset{\cong}{\longrightarrow} \bigsqcup\nolimits_{\rho \in \mathfrak{R}(\Lambda(\beta\Delta))}\widehat{\MCS}{}^\rho(\La(\beta\Delta)),
        \]
        where the first map is defined similarly to $\psi_{\mathfrak{r}}$ except that we \emph{do} move the dashed handleslide marks to the right after performing a backwards MCS trivalent move.
    \end{remark}
    
    \begin{proposition}\label{lma:chart_to_mcs_inverse}
        The map $\psi_\mathfrak{r}$ is bijective for any morphism $\mathfrak{r}\colon \beta\to\Delta$ in $\mathfrak{R}_n$.
    \end{proposition}
    \begin{proof}
        Consider the sequence of braids $(\beta \eqqcolon\beta_1, \dots, \beta_q \coloneqq \Delta)$. Define the inverse map of $\psi_\mathfrak{r}$ by performing the MCS-SR moves in the forwards direction from $SR(\beta)$ to get the associated SR-form MCS sequence $SR(\mathfrak{r})$. By \cref{cor:mcs_matrix} and \cref{lma:monodromy_trivalent_etc_id}, this defines a well-defined inverse to $\psi_{\mathfrak{r}}$ for all MCS braid moves except for an MCS-SR trivalent move, since every other MCS braid move with trivial monodromy uniquely determines the handleslide marks after the moves. Moreover, by \cref{dfn:ruling_category}, the switches and departures are not involved in any braid operations, so an SR-form MCS is sent canonically to an SR-form MCS.

        For the MCS-SR trivalent move $\beta_j \to \beta_{j+1}$, we need to verify that the inverse map is well-defined (1)~when performing the move from an SR-form MCS and (2)~when moving dashed handleslide marks to the right and adjusting the green handleslide marks to an SR-form MCS. 
        
        For step (1), from an SR-form MCS $SR(\beta_j)$ with handleslide marks $x$ at the return and $y, -y^{-1}$ at the switch (where $y$ is a green handleslide and $-y^{-1}$ is a dashed handleslide), we get the unique MCS $SR(\beta_{j+1})$ with a green handleslide mark $x + y^{-1}$, a dashed handleslide mark $-y$, and marked points $y, -y^{-1}$, as well as the dashed handleslide mark $-y^{-1}$ which persists. By \cref{lma:mcs_braid_matrix}, the extra dashed handleslide mark $-y$ cancels out with $-y^{-1}$ after passing through the marked points; i.e.~no new dashed handleslide marks appear.
        
        For step (2), we only need to move the newly created marked points to the right and adjust the green handleslide marks to SR-form again. By \cref{lma:mcs_braid_matrix}, moving marked points changes handleslide marks by scaling. For handleslide marks that are not switches, this does not impose any extra non-vanishing conditions to take inverses. For the handleslide marks $z$ and $-z^{-1}$ at the switches, moving the marked point $t$ from the top (bottom) left strand to the bottom (top) right strand changes the handleslide marks to $tz$ and $-t^{-1}z^{-1}$ ($t^{-1}z$ and $-tz^{-1}$), which remains an SR-form MCS. Thus, the inverse map is also well-defined for an MCS-SR trivalent move and thus the inverse of $\psi_\mathfrak{r}$ is globally well defined.
    \end{proof}
    Recall from \cref{lem:sr_form_bijection_equiv,lma:sr-form_equiv_factors} that there is an isomorphism
    \begin{equation}\label{eq:sr_form_iso}
        \eta_\rho \colon \C^{r(\rho)-\binom n2} \times (\C^\ast)^{s(\rho)} \overset{\cong}{\longrightarrow} \widehat{\MCS}{}^\rho(\La(\beta\Delta)).
    \end{equation}
    \begin{remark}\label{rmk:iso_rainbow}
         If $\beta = \Delta\gamma$ for some $\gamma \in \Br_n^+$, the isomorphism \eqref{eq:sr_form_iso} is given by associating to $(z,x) \in \C^{r(\rho)-\binom n2} \times (\C^\ast)^{s(\rho)}$ the SR-form MCS with handleslide marks labeled by $z_i$ near returns of $\rho$ corresponding to crossings of $\gamma$ and handleslide marks labeled by $x_i$ and $-x_i^{-1}$, respectively, at switches. This defines an SR-form MCS after adding marked points with appropriate values, depending on $x_i$; see \cref{lem:a-form_bijection_equiv}. Furthermore, observe that the inverse to the isomorphism $\eta_\rho$ in the case $\beta = \Delta\gamma$ is defined by picking out the handleslide marks of the returns of $\rho$ corresponding to crossings of $\gamma$, and the switches of $\rho$.
    \end{remark}
    
    \begin{notation}\label{notn:rational_coord_change}
        Suppose $\mathfrak{r}\colon \beta \to \Delta$ is a morphism in $\mathfrak{R}_n$ with $c$ cup moves and $t$ trivalent moves. We denote by $f_{\mathfrak{r}}$ the composition
        \[
         \C^{c} \times (\C^\ast)^{t} \overset{\cong}{\longrightarrow} \mathfrak{M}^\rho(\mathfrak{r}) \longrightarrow \widehat{\MCS}{}^\rho(\La(\beta\Delta)) \overset{\eta_\rho^{-1}}{\longrightarrow} \C^{r(\rho)-\binom n2} \times (\C^\ast)^{s(\rho)},
        \]
        where the composition of the first two maps is $\psi_{\mathfrak{r}}$ and is obtained from \cref{lma:inj_map_sr-form}. The map $\eta_\rho$ is as in \eqref{eq:sr_form_iso}. Namely, we end up with the commutative diagram
        \[
        \begin{tikzcd}[sep=scriptsize]
            \C^{c} \times (\C^\ast)^{t} \dar[swap]{f_{\mathfrak{r}}} \rar{\eta_{\mathfrak{r}}}[swap]{\cong} \drar[swap,near start]{\psi_{\mathfrak{r}}} & \mathfrak{M}^\rho(\mathfrak{r}) \dar \\ 
            \C^{r(\rho)-\binom n2} \times (\C^\ast)^{s(\rho)} \rar{\eta_\rho}[swap]{\cong} & \widehat{\MCS}{}^\rho(\La(\beta\Delta))
        \end{tikzcd}.
        \]
    \end{notation}

    Note that the bijection
    \[
    f_\mathfrak{r} \colon \C^{c} \times (\C^\ast)^{t}\longrightarrow \C^{r(\rho)-\binom n2} \times (\C^\ast)^{s(\rho)}
    \]
    is explicitly computable in the case where $\beta = \Delta\gamma$, since $\eta_\rho$ has an explicit description in this case. We now illustrate this in a simple example.
    \begin{example}\label{ex:hopf}
        We consider the braid $\beta = \sigma_1^3 \in \Br_2^+$ and a morphism $\mathfrak{r}_\rho \colon \sigma_1^3 \to \sigma_1$, $(\beta_1=\sigma_1^3, \beta_2=\sigma_1^2, \beta_3=\sigma_1)$, defined by \cref{fig:mcs_ex_hopf}. The underlying normal ruling of $\mathfrak{r}$ is the maximal normal ruling described by the leftmost diagram in \cref{fig:mcs_ex_hopf}; the switched crossings are indicated by dots. Note that $\mathfrak{A}(\mathfrak{r}_\rho)$ is a right inductive weave $\sigma_1^3 \to \sigma_1$.

         \begin{figure}[!htb]
            \centering
            \hspace*{6mm}\includegraphics{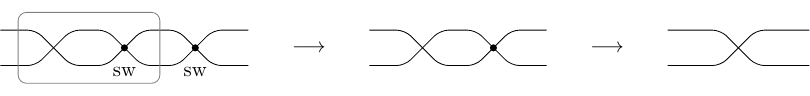}
            \caption{The morphism $\mathfrak{r}_\rho$.}\label{fig:mcs_ex_hopf}
        \end{figure}

        \begin{figure}[!htb]
            \centering
            \includegraphics{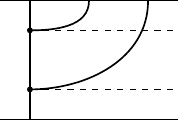}
            \caption{The right inductive weave $\mathfrak{A}(\mathfrak{r}_\rho)$.}\label{fig:mcs_ex_hopf_right}
        \end{figure}
        Let us now compute the map $\psi_{\mathfrak{r}_\rho} \colon (\C^\ast)^2 \to \widehat{\MCS}{}^\rho(\La(\sigma_1^4))$. First, the condition that $\mu(\sigma_1)w_0$ is upper triangular in \cref{dfn:variety_of_morphism}(3) means that the handleslide mark of $\beta_3 = \sigma_1$ is equal to zero. Performing the MCS trivalent moves shown in \cref{fig:mcs_ex_hopf} in the reverse order using the local model shown in \cref{fig:mcs_backwards_sr}, we pick $(y_1,y_2) \in (\C^\ast)^2$ and assign $y_1$ to the first MCS trivalent move and $y_2$ to the second MCS trivalent move. The full sequence of backwards moves (including movement of marked points) is depicted in \cref{fig:psi_to_sr-form}.

        \begin{figure}[!htb]
            \centering
            \includegraphics{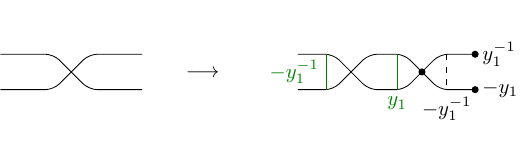}
            \includegraphics{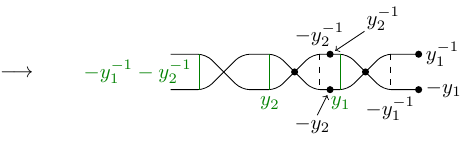}
            \includegraphics{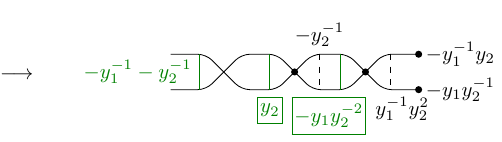}
            \caption{The function $\psi_{\mathfrak{r}_\rho}$ is defined by sending $(y_1,y_2) \in (\C^\ast)^2$ to the equivalence class of the SR-form MCS associated to the normal ruling $\rho$ depicted at the bottom.}
            \label{fig:psi_to_sr-form}
        \end{figure}

        The map $\eta_\rho^{-1} \colon \widehat{\MCS}{}^\rho(\La(\sigma_1^4)) \to (\C^\ast)^2$ is given by picking out the handleslide mark at each return except for the leftmost one, and the handleslide mark to the left of every switch. By \cref{fig:psi_to_sr-form} we therefore see that the map
        \[
            f_{\mathfrak{r}_\rho} \colon (\C^\ast)^2 \longrightarrow (\C^\ast)^2
        \]
        is given by
        \[
            f_{\mathfrak{r}_\rho}(y_1,y_2) = (y_2,-y_1y_2^{-2}).
        \]

        \begin{figure}[!htb]
            \centering
            \includegraphics{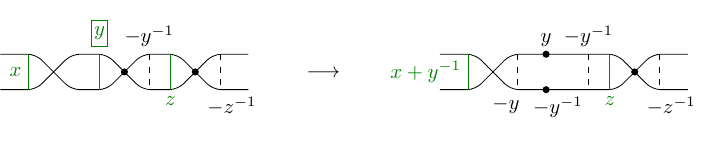}
            \includegraphics{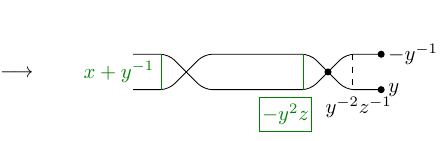}
            \includegraphics{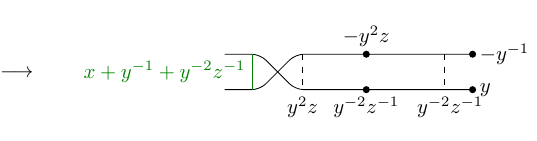}
            \caption{The inverse to $\psi_{\mathfrak{r}_\rho}$ is defined by sending an equivalence class of a SR-form MCS to $(-y^2z,y) \in (\C^\ast)^2$.}
            \label{fig:inverse_rational_coord}
        \end{figure}
        
        Let us also illustrate how to compute the inverse $f_{\mathfrak{r}_\rho}^{-1}$. By the definition of $f_{\mathfrak{r}_\rho}$ in \cref{notn:rational_coord_change}, $f_{\mathfrak{r}_\rho}^{-1}$ is defined by performing MCS trivalent moves with trivial monodromy in the forwards direction determined by $A(\mathfrak{r}_\rho)$. Picking out the handleslide marks at each MCS trivalent move in the reverse order as indicated by \cref{fig:inverse_rational_coord} gives that the map
        \[
        f_{\mathfrak{r}_\rho}^{-1} \colon (\C^\ast)^2 \longrightarrow (\C^\ast)^2
        \]
        is given by
        \[
        f_{\mathfrak{r}_\rho}^{-1}(y,z) = (-y^2z,y),
        \]
    \end{example}

    Recall from \cref{thm:-1=braid vty_new} that there is an isomorphism
    \[
    \alpha \colon X(\beta) \overset{\cong}{\longrightarrow} \Aug(\La(\beta\Delta))
    \]
    such that the braid matrix variables get mapped to the corresponding degree $0$ generators of $\A(\La(\beta\Delta))$. Also recall from \cref{thm:HenryRutherford_dg-homotopy} that given a normal ruling $\rho$ for $\La(\beta\Delta)$ there is an injective map
    \[
    \phi_\rho \colon \C^{r(\rho)-\binom n2} \times (\C^\ast)^{s(\rho)} \hooklongrightarrow \Aug(\La(\beta\Delta)),
    \]
    which factors through the isomorphism $\eta_\rho$ in \eqref{eq:sr_form_iso}.
    
    \begin{theorem}\label{thm:corr_decomps}
        Let $\mathfrak{r}\colon \beta \to \Delta$ be a morphism in $\mathfrak{R}_n$ with $c$ cup moves and $t$ trivalent moves with underlying normal ruling $\rho$. The following diagram commutes:
        \[
        \begin{tikzcd}[sep=scriptsize]
            \C^{c} \times (\C^\ast)^{t} \rar{\phi_{\mathfrak{r}}} \dar[swap]{f_\mathfrak{r}} & X(\beta) \dar{\alpha}[swap]{\cong} \\
            \C^{r(\rho)-\binom n2} \times (\C^\ast)^{s(\rho)} \rar{\phi_{\rho}} & \Aug(\La(\beta\Delta))
        \end{tikzcd}.
        \]
    \end{theorem}
    \begin{proof}
        By construction, the maps $\phi_{\mathfrak{r}}$ and $\phi_\rho$ factor through the maps $\psi_{\mathfrak{r}}$ and $\eta_\rho$ in \cref{lma:inj_map_sr-form} and \eqref{eq:sr_form_iso}, respectively, meaning we have two commutative diagrams
        \begin{equation}\label{eq:inj_maps_factor}
            \begin{tikzcd}[sep=scriptsize]
                \C^c \times (\C^\ast)^t \rar{\psi_\mathfrak{r}} \drar[swap]{\phi_\mathfrak{r}} & \widehat{\MCS}{}^\rho(\La(\beta\Delta)) \dar{\mu}\\
                & X(\beta)
            \end{tikzcd}
            \quad
            \begin{tikzcd}[sep=scriptsize]
                \C^{r(\rho)-\binom n2} \times (\C^\ast)^{s(\rho)} \rar{\eta_\rho}[swap]{\cong} \drar[swap]{\phi_\rho} & \widehat{\MCS}{}^\rho(\La(\beta\Delta)) \dar{\nu} \\
                & \Aug(\La(\beta\Delta))
            \end{tikzcd}.
        \end{equation}
        These diagrams together with the definition of $f_\mathfrak{r}$ (\cref{notn:rational_coord_change}) yield the following diagram
        \[
        \begin{tikzcd}[sep=scriptsize]
            \C^c \times (\C^\ast)^t \dar{f_\mathfrak{r}} \drar{\psi_\mathfrak{r}} \ar[rrd,out=0,in=155,"\phi_\mathfrak{r}"] & & \\
            \C^{r(\rho)-\binom n2} \times (\C^\ast)^{s(\rho)} \rar{\cong}[swap]{\eta_\rho} \ar[rrd,out=335,in=180,swap,"\phi_\rho"] & \widehat{\MCS}{}^\rho(\La(\beta\Delta)) \rar{\mu} \drar{\nu} & X(\beta) \dar{\alpha} \\
            & & \Aug(\La(\beta\Delta))
        \end{tikzcd},
        \]
        where the upper right and lower left triangles commute by \eqref{eq:inj_maps_factor}. The left triangle commutes by the definition of $f_\mathfrak{r}$. By the definition of the maps $\mu$ and $\nu$ in \eqref{eq:inj_maps_factor}, we see that the right triangle commutes. Hence the outer square commutes, finishing the proof.
    \end{proof}
    \begin{theorem}\label{thm:hr_and_weave_decomps}
        Let $\beta \in \Br_n^+$ such that $\delta(\beta) = w_0$. For every normal ruling $\rho$ of $\La(\beta\Delta)$, there exists a right simplifying weave $\mathfrak{w}_{\rho}$ such that the ruling decomposition of $\Aug(\La(\beta\Delta))$ coincides with the weave decomposition of $X(\beta)$ given by this collection of right simplifying weaves $\mathfrak{w}_{\rho}$ under the isomorphism $\alpha$.
    \end{theorem}
    \begin{proof}
        Let $(\rho_1,\ldots,\rho_k)$ be a tuple of all normal rulings of $\La(\beta\Delta)$. By \cref{cor:rulings_decomposing_tuple}, there is a tuple $(\mathfrak{r}_{R,1},\ldots,\mathfrak{r}_{R,k})$ of right inductive morphisms $\mathfrak{r}_{R,i} \colon \beta \to \Delta$ in $\mathfrak{R}_n$ with underlying normal ruling $\rho_i$, that is a decomposing tuple of $X(\beta)$ in the sense of \cref{prop:mcs_cat_decompos}. By \cref{thm:functors_from_mcs}, the decomposition of $X(\beta)$ induced by $(\mathfrak{r}_{R,1},\ldots,\mathfrak{r}_{R,k})$ corresponds to the weave decomposition of $X(\beta)$ induced by the corresponding tuple of right simplifying weaves $(\mathfrak{A}(\mathfrak{r}_{R,1}),\ldots,\mathfrak{A}(\mathfrak{r}_{R,k}))$. For each $i\in \{1,\ldots,k\}$, the two injective maps $\phi_{\mathfrak{r}_i}$ and $\phi_{\rho_i}$ are the associated pieces of the weave decomposition and ruling decomposition, respectively. By \cref{thm:corr_decomps}, $f_{\mathfrak{r}_{R,i}}$ is bijective and so it follows that $\im \phi_{R,i} = \im \phi_{\rho_i}$ for each $i\in \{1,\ldots,k\}$.
    \end{proof}

We now compare the ruling decomposition of $\Aug(\La(\beta\Delta))$ with the Deodhar decomposition of the braid Richardson variety $R^\circ_{w_0,\beta}$ defined in \cref{dfn:richardson_variety}.

\begin{theorem}\label{thm:deodhar=hr_decomp}
    Let $\beta \in \Br_n^+$ such that $\delta(\beta) = w_0$. For every normal ruling $\rho$ of $\La(\beta\Delta)$ there exists a distinguished sequence $\mathfrak{v}_{\rho}$ of $w_0$ such that the ruling decomposition of $\Aug(\La(\beta\Delta))$ coincides with the Deodhar decomposition of $R^\circ_{w_0,\beta}$ determined by the distinguished sequences $\mathfrak{v}_{\rho}$ under the isomorphisms from \cref{thm:-1=braid vty_new,prop:braid_richardson}.
\end{theorem}
\begin{proof}
    As described in \cref{prop:normal_ruling_inductive}, the normal rulings of $\La(\beta\Delta)$ determine a decomposing tuple of right simplifying weaves $\beta\to \Delta$, which, by \cref{lma:right_inductive_distinguished_sequences} and the proof of \cref{thm:deodhar=weave}, corresponds to a tuple of distinguished sequences that decomposes $R^\circ_{w_0,\beta}$. The result then follows from combining \cref{thm:deodhar=weave,thm:hr_and_weave_decomps}.
\end{proof}

\subsection{Cluster variables via Morse complex sequences}\label{sec:cluster_via_mcs}
    In this section we explain how to compute the cluster variables associated to right inductive weaves (see \cref{sec:cluster_variables}) solely in terms of Morse complex sequences.
    \begin{definition}[$s_t$ variables]\label{dfn:s_variable_framed_mcs}
        Let $\mathfrak{m} \colon \beta \to \beta'$ be a morphism in $\mathfrak{B}_n$ represented by $(\beta_1,\ldots,\beta_q)$. We associate to each framed trivalent move $t$ a variable $s_t$ that is a Laurent polynomial defined as follows: The morphism $\mathfrak{m}$ determines an associated framed A-form MCS with trivial monodromy so that $t$ corresponds to a framed MCS trivalent move. The variable $s_t$ is defined as the variable $y$ in the local model for the framed MCS braid move with trivial monodromy as in \cref{fig:mcs_trivalent_framed_monoless} after setting all the variables at marked points in $A(\beta_1)$ equal to one.
    \end{definition}
    \begin{remark}
        It is a direct consequence of \cref{thm:frame_functors_from_mcs} that the $s_t$-variables associated to $\mathfrak{m} \colon \beta \to \beta'$ defined in \cref{dfn:s_variable_framed_mcs} coincide with the $s_v$-variables of the associated weave $\mathfrak{A}(\mathfrak{m})$ defined in \cref{dfn:s-variable_weave}.
    \end{remark}
    
    \begin{example}\label{ex:s-variables_hopf_computation}
        Consider the morphism $\mathfrak{m} \colon \sigma_1^3 \to \sigma_1$ defined by the sequence of braid moves $(\sigma_1\sigma_1)\sigma_1 \to \sigma_1 \sigma_1 \to \sigma_1$. We observe that $\mathfrak{A}(\mathfrak{m})$ is a right inductive weave for $\sigma_1^3$. \Cref{fig:ex_s-variables_hopf} shows the sequence of the two framed MCS trivalent moves with trivial monodromy induced by $\mathfrak{m}$. After setting $u_1 = u_2 = u_3 = 1$, the $s_t$-variables are $s_1 = z_2$ and $s_2 = z_3-z_2^{-1}$. This coincides with the $s$-variables of the corresponding weave computed in \cref{ex:cluster-hopf}.
    \end{example}
    
        \begin{figure}[!htb]
            \centering
            \includegraphics{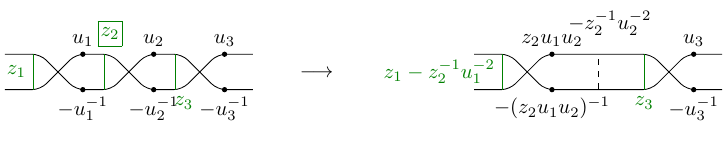}
            \includegraphics{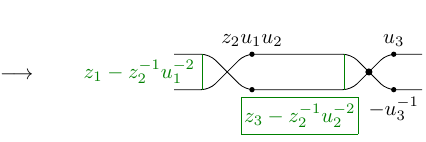}
            \includegraphics{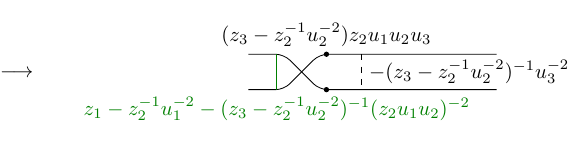}
            \caption{The sequence of framed A-form MCSs induced by $\mathfrak{m}$ in \cref{ex:s-variables_hopf_computation}. The labels which are used to find the $s$-variables associated to the framed MCS trivalent moves are indicated in boxes.}
            \label{fig:ex_s-variables_hopf}
        \end{figure}

    We now show that the $s$-variables are related to the A-to-SR-form algorithm used in the proof of \cref{thm:one-to-one_A_SR}; see also \cref{rmk:A_SR_algorithm}. The starting point of this comparison is a framed version of the formal A-form MCS; see \cref{dfn:formal_a-form_mcs,dfn:framed_A-form_MCS}.
    \begin{definition}[Formal framed A-form MCS]
        A (graded) \emph{formal framed A-form MCS} for a front diagram $D$ equipped with a Maslov potential is a graded framed A-form MCS where each handleslide mark is considered to be a formal variable valued in $R$, and where each $u$-variable associated to marked points immediately to the right of each crossing is equal to one. 
    \end{definition}
    \begin{remark}
        A marked point with variable equal to one is equivalent to the non-existence of that marked point. For the purpose of the formal framed A-form MCS, it is therefore enough to consider a marked point with variable $-1$ on the lower strand immediately to the right of each crossing; see \cref{fig:formal_framed_ex} for an example of the formal framed A-form MCS of the $(-1)$-closure of $\sigma_1\sigma_2\sigma_1^3\sigma_1(\sigma_2\sigma_1^2)^2 \in \Br_3^+$.
    \end{remark}
    \begin{figure}[!htb]
        \centering
        \includegraphics{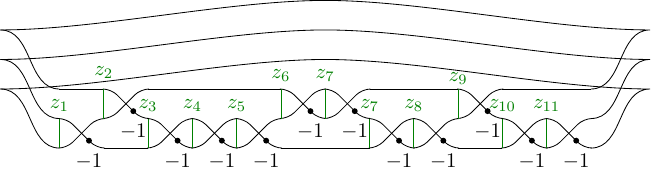}
        \caption{The formal framed A-form MCS of the $(-1)$-closure of $\sigma_1\sigma_2\sigma_1^3\sigma_1(\sigma_2\sigma_1^2)^2$.}
        \label{fig:formal_framed_ex}
    \end{figure}
    Recall from \cref{prop:max_ruling} that for any $\beta \in \Br_n^+$ with $\delta(\beta) = w_0$, the $(-1)$-closure of $\beta\Delta$ admits a unique normal ruling with the maximal number of switches.
    \begin{definition}[Formal framed SR-form MCS]
        Let $\beta \in \Br_n^+$ be such that $\delta(\beta) = w_0$, and let $\rho_0$ denote the maximally switched normal ruling of $\La(\beta\Delta)$. The \emph{formal framed SR-form MCS} of $\rho_0$ is the result of performing the A-to-SR-form algorithm to the formal framed A-form MCS of $\La(\beta\Delta)$.
    \end{definition}
    \begin{lemma}\label{lma:s-variables_sr-form}
        Let $\beta \in \Br_n^+$ such that $\delta(\beta) = w_0$. The $s$-variables associated with a right inductive morphism $\mathfrak{r}_\rho \colon \beta \to \Delta$ are equal to the handleslide marks to the left of each switch of the formal framed SR-form MCS of $\La(\beta\Delta)$.
    \end{lemma}
    \begin{proof}
        This follows by comparing the A-to-SR-form algorithm described in \cref{dfn:formal_a-form_mcs,dfn:framed_A-form_MCS} with the definition of a right inductive morphism in the proof of \cref{prop:normal_ruling_inductive}.

      Roughly, the A-to-SR-form algorithm consists of creating canceling handleslide marks to the right of each switch; see \cref{fig:a-to-sr-form_alg}. Then the dashed handleslide mark with label $z^{-1}$ is moved as far right in the braid as possible, via MCS moves. By definition of the algorithm, creation and movement of handleslide marks are done starting with the leftmost switch of $\rho_0$, and proceeds to the right in $\beta$. The movement of the dashed handleslide mark affects all handleslide marks to the right of the current switch, but none of the handleslide marks to the left of it.

        \begin{figure}[!htb]
            \centering
            \includegraphics{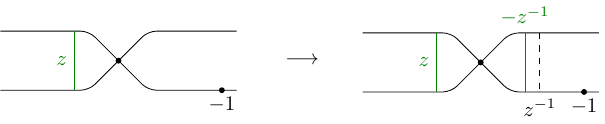}
            \caption{Creation of canceling handleslide marks in the A-to-SR-form algorithm}
            \label{fig:a-to-sr-form_alg}
        \end{figure}
        
          According to the proof of \cref{prop:normal_ruling_inductive}, a right inductive morphism $\mathfrak{r}_\rho$ performs (framed) MCS trivalent moves with trivial monodromy at the leftmost switch and proceeds right to left through $\beta$ (while performing intermediary hexavalent and distant crossing moves to the left of the switch so that the trivalent moves can be performed). In the framed MCS trivalent move depicted in \cref{fig:mcs_trivalent_framed_monoless}, the new dashed handleslide mark created is $-y^{-1}v^{-2}$, which is exactly the dashed handleslide mark created in the A-to-SR-form algorithm after it has moved past the marked point labeled by $-1$ and after setting all the variables at marked points in the first framed A-form MCS of $\mathfrak{r}_\rho$ equal to one (which implies $v = 1$). This means that when moving the dashed handleslide mark with label $-y^{-1}v^{-2}$ to the right in the braid, it possibly affects all handleslide marks to the right of the current switch in exactly the same way as how they are affected by the corresponding A-to-SR-form algorithm.
        
        Now, by definition, the $s$-variable at the trivalent move is $y$, which coincides with the handleslide mark at the corresponding switch of $\rho_0$, finishing the proof.
    \end{proof}
    
    \begin{example}\label{ex:s-variables_mcs}
        We consider the $(-1)$-closure of $\beta\Delta$ where $\beta = (\sigma_1^2\sigma_2^2)^2$. The formal framed A-form MCS and formal framed SR-form MCS of $\La(\beta\Delta)$ are depicted in \cref{fig:ex_framed_sr-form1} and \cref{fig:ex_framed_sr-form2}, respectively. 
        \begin{figure}[!htb]
        \centering
        \begin{subfigure}{\textwidth}
            \centering
            \includegraphics{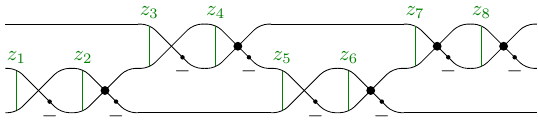}
            \caption{}\label{fig:ex_framed_sr-form1}
        \end{subfigure}
        \begin{subfigure}{\textwidth}
            \centering
            \includegraphics{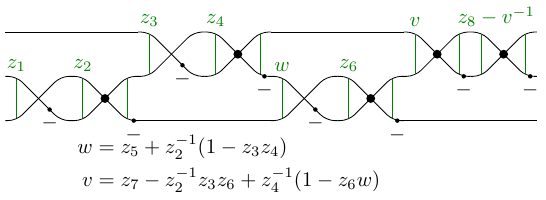}
            \caption{}\label{fig:ex_framed_sr-form2}
        \end{subfigure}
        \caption{The formal framed A-form MCS and formal framed SR-form MCS of $\La(\beta\Delta)$ where $\beta = (\sigma_1^2\sigma_2^2)^2$. Each marked point is labeled by $-1$.}
        \label{fig:ex_framed_sr-form}
    \end{figure}
    The crossings marked with $\bullet$ are the switches of the maximally switched normal ruling $\rho_0$, and if the handleslide mark immediately to the left of a switch is equipped with the variable $x$, the handleslide mark immediately to the right of it is equipped with the variable $-x^{-1}$. By \cref{lma:s-variables_sr-form}, we find the $s$-variables $s_1 = z_2$, $s_2 = z_4$, $s_3 = z_6$, $s_4 = v$, and $s_5 = z_8-v^{-1}$, where
        \[
        v = z_7-z_2^{-1}z_3z_6+z_4^{-1}(1-z_6(z_5+z_2^{-1}(1-z_3z_4))).
        \]
    \end{example}

        We now recast the definition of Lusztig cycles and \textsf Y-trees of algebraic weaves entirely in terms of morphisms in the braid category $\mathfrak{B}_n$ (see \cref{dfn:braid_category}).

    \begin{definition}[Letter set]
        Let $\mathfrak{m} \colon \beta \to \beta'$ be a morphism in $\mathfrak{B}_n$. The \emph{letter set} of $\mathfrak{m}$ is defined as the set
        \[
        L(\mathfrak{m}) \coloneqq \{(k,m) \mid k\in \{1,\ldots,q\}, \; m \in \{1,\ldots,\ell(\beta_k)\}\}.
        \]
        Each point $(k,m)$ in $L(\mathfrak{m})$ corresponds to the $m$-th letter in the braid word for $\beta_k$.
    \end{definition}

    \begin{definition}[Position of braid moves]
        Let $\mathfrak{m} \colon \beta \to \beta'$. We say that a trivalent move \emph{occurs at position} $(k,m) \in L(\mathfrak{m})$ if it is the move
        \[
            \beta_k = \sigma_{i_1} \cdots \sigma_{i_m} \sigma_{i_{m+1}} \cdots \sigma_{i_{\ell(\beta_k)}} \longrightarrow \sigma_{i_1} \cdots \sigma_{i_m}\sigma_{i_{m+2}}\cdots \sigma_{i_{\ell(\beta_{k})}} = \beta_{k+1},
        \]
        where $i_m = i_{m+1}$. In this case we say that $\sigma_{i_m}$ and $\sigma_{i_{m+1}}$ in $\beta_k$ are the two \emph{input letters}, and that $\sigma_{i_m}$ in $\beta_{k+1}$ is the \emph{output letter}. Note that the output letter of the trivalent move at position $(k,m)$ is at position $(k+1,m) \in L(\mathfrak{m})$. We define the position of the other braid moves similarly.
    \end{definition}

    \begin{definition}[Lusztig cycle]\label{dfn:lusztig_cycle_mcs}
        Let $\mathfrak{m} \colon \beta \to \beta'$ be a morphism in $\mathfrak{B}_n$. For each trivalent move $t$ occurring at position $(r,s) \in L(\mathfrak{m})$, the \emph{Lusztig cycle} $\gamma_t$ associated to $t$ is a map $\gamma_t \colon L(\mathfrak{m}) \to \mathbb N$ such that the following holds.
        \begin{itemize}
            \item $\gamma_t(k,m) = 0$ if $k\leq r$ and
            \[
            \gamma_t(r+1,m) = \begin{cases}1, & m = s, \\ 0, & m \neq s.\end{cases}
            \]
            \item For any trivalent move occurring at position $(r',s')$ with $r' > r$, we have
            \[
            \gamma_t(r'+1,m) = \begin{cases}
                \gamma_t(r',m), & m < s', \\
                \min(\gamma_t(r',s'),\gamma_t(r',s'+1)), & m = s', \\
                \gamma_t(r',m+1), & m > s'.
            \end{cases}
            \]
            \item For any hexavalent move occurring at position $(r',s')$ with $r' > r$, we have
            \[
            \gamma_t(r'+1,m) = \begin{cases}
                \gamma_t(r',m), & m < s', \\
                \gamma_t(r',s'+1) + \gamma_t(r',s'+2) - \min(\gamma_t(r',s'),\gamma_t(r',s'+2)), & m = s', \\
                \min(\gamma_t(r',s'),\gamma_t(r',s'+2)), & m = s'+1,\\
                \gamma_t(r',s') + \gamma_t(r',s'+1) - \min(\gamma_t(r',s'),\gamma_t(r',s'+2)), & m = s'+2,\\
                \gamma_t(r',m), & m > s'+2.
            \end{cases}
            \]
            \item For any distant crossings move occurring at position $(r',s')$ with $r' > r$, we have
            \[
            \gamma_t(r'+1,m) = \begin{cases}
                \gamma_t(r',m), & m < s', \\
                \gamma_t(r',m+1), & m = s', \\
                \gamma_t(r',m-1), & m = s'+1, \\
                \gamma_t(r',m), & m > s'+1.
            \end{cases}
            \]
            \item For any cup move occurring at position $(r',s')$ with $r' > r$, we have
            \[
            \gamma_t(r'+1,m) = \begin{cases}
                \gamma_t(r',m), & m < s', \\
                \gamma_t(r',m+2), & m \geq s'.
            \end{cases}
            \]
        \end{itemize}
    \end{definition}
    \begin{lemma}\label{lma:compare_lusztig_cycles}
        Let $\mathfrak{m} \colon \beta \to \beta'$ be a morphism in $\mathfrak{B}_n$. Under the functor $\mathfrak{A}$ in \cref{thm:ruling_to_weaves}, we have a one-to-one correspondence between Lusztig cycles $\gamma_t$ of $\mathfrak{m}$ and $\gamma_v$ of $\mathfrak{A}(\mathfrak{m})$.
    \end{lemma}
    \begin{proof}
        Firstly, by definition, it is immediate that each trivalent move corresponds to a trivalent vertex in $\mathfrak{A}(\mathfrak{m})$. The result then follows by comparing \cref{dfn:lusztig_cycle,dfn:lusztig_cycle_mcs}.
    \end{proof}

    \begin{definition}[Cover of a trivalent]
        Let $\beta \in \Br_n^+$ such that $\delta(\beta) = w_0$ and let $\mathfrak{r} \colon \beta \to \Delta$ be a right inductive morphism (see \cref{def:right_inductive_morphism}). If $t$ and $t'$ are two trivalent moves of $\mathfrak{r}$ occurring at positions $(r,s)$, and $(r',s')$, respectively, we say that $t'$ \emph{covers} $t$ if $\gamma_{t'}(r,s) \neq 0$.
    \end{definition}

    In analogy with \cref{thm:cluster_inductive}(2), we make the following definition.
    
    \begin{definition}[Cluster variable of a trivalent]\label{dfn:cluster_var_mcs}
        Let $\beta \in \Br_n^+$ such that $\delta(\beta) = w_0$ and let $\mathfrak{r} \colon \beta \to \Delta$ be a right inductive morphism. For any trivalent move $t$ in $\mathfrak{r}$, the \emph{cluster variable} $A_t(\mathfrak{r})$ is defined by the inductive formula
        \[
        A_t = \begin{cases}
        s_t \cdot \prod_{t' \text{ covers }t}A_{t'}^{\gamma_{t'}(r,s)}, &\text{if there is a trivalent }t'\text{ which covers }t,\\
        s_t,&\text{otherwise}
        \end{cases}
        \]
        where $(r,s)$ denotes the position of the trivalent move $t$.
    \end{definition}
    \begin{definition}[$u$-variable associated to a trivalent]
        Suppose that $\mathfrak{r} \colon \beta \to \Delta$ is a morphism in $\mathfrak{R}_n$ represented by the sequence $(\beta_1,\ldots,\beta_q)$. For any trivalent move $t$ in $\mathfrak{r}$, the \emph{$u$-variable associated to $t$} is defined to be the variable $yuv$ associated to the marked point on the upper strand in the local model for the framed MCS trivalent move as depicted in \cref{fig:mcs_trivalent_framed_monoless}, after setting all the variables at marked points in $A(\beta_1)$ equal to one.
    \end{definition}
    An equivalent way of computing the cluster variables is given by the following lemma.
    \begin{lemma}[{\cite[Theorem 5.19(3)]{CGGLSS}}]\label{lma:cluster_u-variable}
          Let $\beta \in \Br_n^+$ such that $\delta(\beta) = w_0$ and let $\mathfrak{r} \colon \beta \to \Delta$ be a right inductive morphism represented by the sequence $(\beta_1,\ldots,\beta_q)$. For any trivalent move $t$ in $\mathfrak{r}$, the cluster variable $A_t(\mathfrak{r})$ associated to $t$ is given by the $u$-variable associated to $t$, after setting all the variables at marked points in $A(\beta_1)$ equal to one.
    \end{lemma}
    \begin{theorem}\label{thm:handleslide_clusters}
        Let $\beta \in \Br_n^+$ such that $\delta(\beta) = w_0$ and let $\mathfrak{r} \colon \beta \to \Delta$ be a right inductive morphism. Under the functor $\mathfrak{A}$ in \cref{thm:ruling_to_weaves}, we have
        \[
        A_t(\mathfrak{r}) = A_v(\mathfrak{A}(\mathfrak{r})),
        \]
        where $v$ is the trivalent vertex in a right inductive weave $\mathfrak{A}(\mathfrak{r})$ corresponding to the trivalent move $t$.
    \end{theorem}
    \begin{proof}
        This follows from \cref{lma:compare_lusztig_cycles} and by comparing \cref{dfn:cluster_var_mcs} with \cref{thm:cluster_inductive}.
    \end{proof}

    \begin{example}
        Consider the morphism $\mathfrak{m} \colon \sigma_1^3 \to \sigma_1$ defined by the sequence of braid moves $(\sigma_1\sigma_1)\sigma_1 \to \sigma_1\sigma_1 \to \sigma_1$; see \cref{fig:ex_s-variables_hopf}. In \cref{ex:s-variables_hopf_computation} we found the $s_t$-variables to be $s_1 = {z}_2$ and $s_2 = {z}_3-{z}_2^{-1}$. The first trivalent move $t'$ of $\mathfrak{m}$ occurs at $(1,1)$, and the second trivalent move $t$ occurs at $(2,1)$. The only Lusztig cycle that covers another is given by
        \[
        \gamma_{t'}(k,m) = \begin{cases}
            1, & (k,m) = (2,1), \\
            0, & \text{otherwise},
        \end{cases}
        \]
        and hence we see that $t'$ covers $t$. By definition we have the cluster variables $A_1 = s_1 = {z}_2$ and $A_2 = s_2 A_1 = 1 - {z}_2{z}_3$. This computation coincides with the one in \cref{ex:cluster-hopf}.

        As an alternative computation, using \cref{lma:cluster_u-variable}, via \cref{fig:ex_s-variables_hopf} we see that the $u$-variable associated to the first trivalent move is ${z}_2u_1u_2$, so $A_1 = {z}_2$, and the $u$-variable associated to the second trivalent move is $({z}_3-{z}_2^{-1}u_2^{-2}){z}_2u_1u_2u_3$, so $A_2 = ({z}_3-{z}_2^{-1}){z}_2 = 1-{z}_2{z}_3$.
    \end{example}
    Finally, we combine \cref{thm:handleslide_clusters} with \cref{lma:s-variables_sr-form} to describe how to compute cluster variables using the formal framed SR-form MCS and the combinatorics of the Lusztig cycles:
    \begin{theorem}\label{thm:cluster_variables_mcs}
        Let $\beta \in \Br_n^+$, and let $\rho_0$ be the maximally switched normal ruling of $\La(\beta\Delta)$. There exists an algorithm to compute the cluster variables associated with a right inductive morphism $\mathfrak{r}_{\rho_0} \colon \beta \to \Delta$ from the formal framed SR-form MCS associated with $\rho_0$, together with the combinatorics of the Lusztig cycles determined by $\mathfrak{r}_{\rho_0}$.
        \qed
    \end{theorem}
    \begin{example}
        Let us consider $\beta = (\sigma_1^2\sigma_2^2)^2$ and compute the cluster variables associated with a right inductive morphism $\mathfrak{r}_\rho \colon \beta \to \Delta$ corresponding to the maximally switched normal ruling $\rho_0$ of $\La(\beta\Delta)$.

        In \cref{ex:s-variables_mcs} we saw that the $s$-variables are given by $s_1 = z_2$, $s_2 = z_4$, $s_3 = z_6$, $s_4 = v$, and $s_5 = z_8-v^{-1}$, where
        \[
        v = z_7-z_2^{-1}z_3z_6+z_4^{-1}(1-z_6(z_5+z_2^{-1}(1-z_3z_4))).
        \]
        By \cref{dfn:cluster_var_mcs}, we now have that the cluster variables $A_1,\ldots,A_5$ are determined inductively via the $s$-variable and the combinatorics of the Lusztig cycles.

        Since there is a one-to-one correspondence between switches in $\rho_0$ and MCS trivalent moves in $\mathfrak{r}_\rho$, we have five MCS trivalent moves $t_1,\ldots,t_5$. By \cref{fig:lusztig_cycle_comb}, we see that $t_1$ and $t_2$ both cover $t_4$, and $t_4$ covers $t_5$, and these are the only Lusztig cycles that cover. 
        \begin{figure}[!htb]
            \centering
            \includegraphics{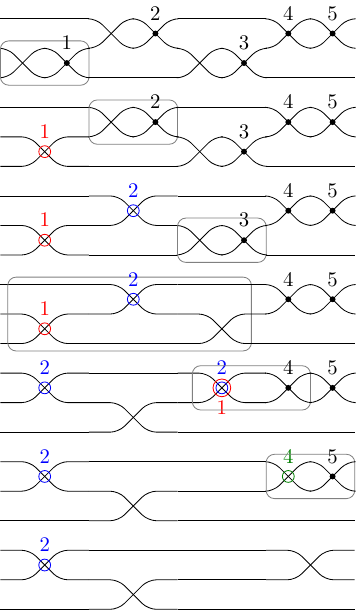}
            \caption{The Lusztig cycles in a right inductive morphism of the maximally switched normal ruling of $\La((\sigma_1^2\sigma_2^2)^2\Delta)$ for the three trivalents which cover another trivalent. Each Lusztig cycle corresponds to one of the colors of sequences of encircled crossings.}
            \label{fig:lusztig_cycle_comb}
        \end{figure}
        Therefore we compute the cluster variables as follows:
        \begin{align*}
            A_1 &= s_1 = z_2 \\
            A_2 &= s_2 = z_4 \\
            A_3 &= s_3 = z_6 \\
            A_4 &= s_4A_1A_2 = vz_2z_4 = z_2z_4z_7+z_2-z_2z_5z_6-z_6 \\
            A_5 &= s_5A_4 = (z_8-v^{-1})vz_2z_4 \\
            &= -z_2z_4+z_2z_4z_7z_8+z_2z_8-z_2z_5z_6z_8-z_6z_8.
        \end{align*}
    \end{example}

\subsection{Trivial monodromy algebraic weaves, augmentations, and sheaves}\label{sec:weaves_aug_sheaves}
    Given an algebraic weave $\w \colon \beta \to \beta'$, we call a point in $\mathfrak{X}(\w)$ a \emph{trivial monodromy algebraic weave}. From \cref{rmk:ruling_caps}, recall that the braid category $\mathfrak{B}_n$ involves caps. The functor $\mathfrak{A} \colon \mathfrak{B}_n \to \mathfrak{W}_n$ (see \cref{thm:ruling_to_weaves}) is an equivalence of categories.

    \begin{notation}
    By construction, the underlying Legendrian links $\La(\beta_{j-1}\Delta)$ and $\La(\beta_j\Delta)$, involved in the MCS braid moves \cref{dfn:mcs_braid_moves}, are related by an immersed exact Lagrangian cobordism determined by the MCS braid moves or, equivalently, the associated weave; see e.g.\@ \cite{CasalsZaslow20,pan2021functorial}. Denote the Legendrian lift of this immersed exact Lagrangian cobordism by $\La(\mathfrak{m})$. 
    \end{notation}
    A \emph{Morse complex 2-family} (MC2F) is a two-dimensional version of a Morse complex sequence introduced by Rutherford--Sullivan \cite[Definition 4.1]{rutherford2018generating}; see \cite[Definition 2.7]{pan2023augmented} for the generalization over $\Z$ with homology coefficients. 
    
    \begin{proposition}\label{prop:monodromyless_MCS_sequences_MC2F}
        Let $n \in \Z_{\geq 1}$ and let $\mathfrak{m} \colon \beta\to \beta'$ be a morphism in the braid category $\mathfrak{B}_n$. Given any point in $\mathfrak M(\mathfrak{m})$, we may construct an MC2F on the front projection of $\La(\mathfrak{m})$.
    \end{proposition}
    \begin{proof}
        Consider the front diagram of the Legendrian surface $\Lambda(\mathfrak{m})$. The handleslide set $H_0$ of the MC2F in the front diagram is the trace of the handleslide marks on the Legendrian associated to each MCS braid move and the super-handleslide set $H_{-1}$ of the MC2F is the set of points where handleslide marks become implicit as in \cref{lma:monodromy_comp_prod,fig:mcs_local13}. The comparison now follows from comparing our computation with \cite[Proposition 4.1]{rutherford2018generating} or \cite[Proposition 7.5]{PanRutherford23}. Then for an MCS equivalence in \cref{lma:monodromy_comp_prod}, \cref{fig:mcs_local3,fig:mcs_local5} correspond to the pentavalent vertex of $H_0$ in \cite[Axiom 4.2(1)]{rutherford2018generating}, \cref{fig:mcs_local13} corresponds to the trivalent singularity of $H_0$ at the super-handleslide $H_{-1}$ in \cite[Axiom 4.2(2)]{rutherford2018generating}, and \cref{fig:mcs_local_marked1,fig:mcs_local_marked2,fig:mcs_local_marked3,fig:mcs_local_marked4,fig:mcs_local_marked5,fig:mcs_local_marked6} correspond to the relation between handleslide marks and homology curves in \cite[Sections 2.3 and 2.4]{pan2023augmented}. For an MCS hexavalent move \cref{lma:monodromy_trivalent_etc_id}, \cref{fig:mcs_hexavalent_monoless1,fig:mcs_hexavalent_monoless2} correspond to a sequence of MCS moves with a pentavalent vertex in $H_0$ in \cite[Axiom 4.2(1)]{rutherford2018generating}; see \cref{fig:Monodromy_MCF_hexa}. For an MCS trivalent move, we can perturb the front diagram at the trivalent vertex, as in \cite[Figure 33]{CasalsZaslow20}, and obtain three swallowtail singularities. Then, \cref{fig:mcs_trivalent_monoless} corresponds to a sequence of MCS moves with three pentavalent vertices and three swallowtail points in $H_0$ \cite[Axiom 4.2(1) and (3)]{rutherford2018generating}; see \cref{fig:Monodromy_MCF_tri}. Finally, the trivial monodromy condition corresponds to \cite[Axiom 4.3 and Proposition 4.1]{rutherford2018generating}; see the concrete computations of handleslide marks in \cref{fig:Monodromy_MCF_hexa,fig:Monodromy_MCF_tri} via the MCS equivalence relations. 
    \end{proof}

    \begin{remark}
        Pan--Rutherford \cite{pan2023augmented} encoded the homology of the Legendrian surface by a collection of curves and the combinatorial spin structure by a second collection of curves ending at swallowtail singularities (see \cite[Definition 2.5]{pan2023augmented}). We implicitly combine the combinatorial spin structures with the curves for homology classes, following the convention in Casals--Ng \cite[Section 3.5]{CasalsNg}. More precisely, we add a pair of marked points for each trivalent vertex in the weave with formal variables $t$ and $-t^{-1}$, which agrees with the marked points given by a saddle cobordism in \cite[Section 3.5]{CasalsNg}. In the language of \cite{pan2023augmented}, the homology curves are defined by the trace of the pair of marked points with formal variables $t$ and $-t^{-1}$, and with endpoints on the boundary of the weave. Instead, the combinatorial spin structure curves are defined by the trace of the bottom marked point labeled $-t^{-1}$ starting from the trivalent vertex and ending on the boundary of the weave, oriented towards the marked point labeled $-t^{-1}$. Since a trivalent vertex can be perturbed into three swallowtails (see \cref{fig:Monodromy_MCF_tri}), our spin structure curves also start at swallowtails in the front and therefore agree with the spin structure curves in~\cite[Definition 2.5]{pan2023augmented}. The same choices of spin structure curves on weaves have also appeared in the work of Casals--Weng \cite[Section 4.5]{CasalsWeng} with the name of sign curves. These sign curves correspond to choices of spin structures as explained in \cite[Appendix B.2]{CasalsLi}.
    \end{remark}

    \begin{figure}[!htb]
            \centering
            \includegraphics{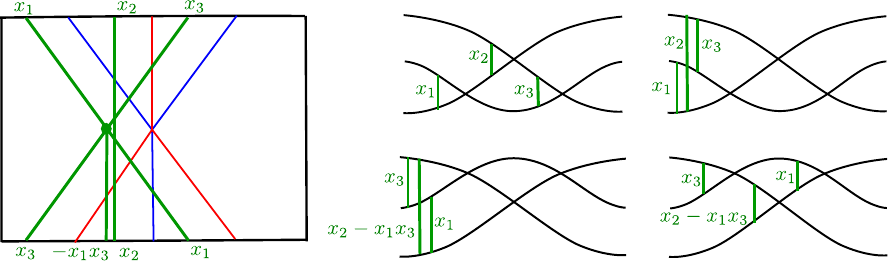}
            \caption{The MC2F associated to the MCS hexavalent move with trivial monodromy, where the thick green lines depict the handleslide sets in \cite[Definition 4.1]{rutherford2018generating} and the green dot depicts the pentavalent vertex in the handleslide set of the MC2F \cite[Axiom 4.2(1)]{rutherford2018generating}.}
            \label{fig:Monodromy_MCF_hexa}
    \end{figure}
    \begin{figure}[!htb]
            \centering
            \includegraphics[width=0.6\textwidth]{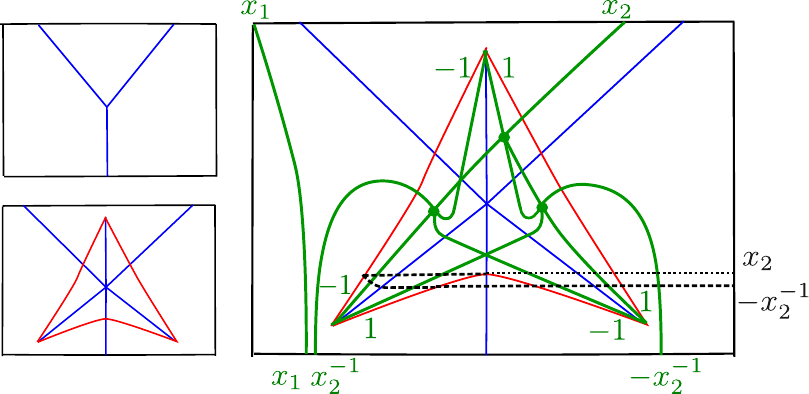}

            \vspace{2mm}
    
            \includegraphics{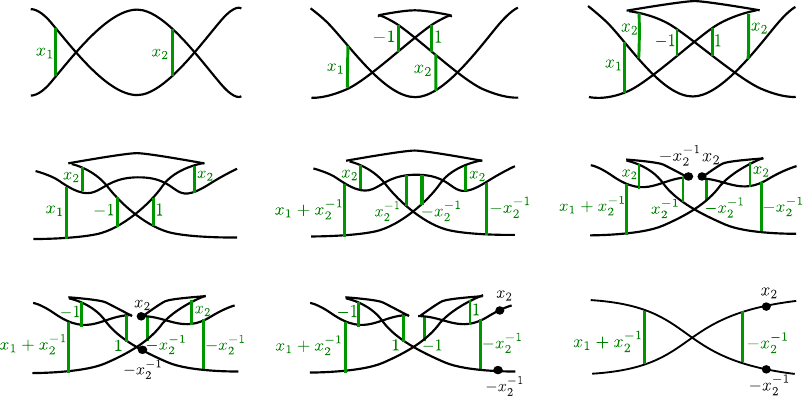}
            \caption{The MC2F associated to the MCS trivalent move with trivial monodromy, where the thick green lines depict the handleslide sets in \cite[Definition 4.1]{rutherford2018generating}, the three green dots depict the pentavalent vertices in the handleslide set \cite[Axiom 4.2(1)]{rutherford2018generating}, the cusp points of the thick green lines depict the swallowtail points in the handleslide set \cite[Axiom 4.2(3)]{rutherford2018generating}, and the dashed lines depict the marked points at the saddle cobordism.}
            \label{fig:Monodromy_MCF_tri}
        \end{figure}
    
    \begin{corollary}\label{cor:algebra-weave-to-MC2F}
        There is a map from trivial monodromy algebraic weaves to the set of Morse complex 2-families on the (front of the Legendrian lift of the) underlying immersed exact Lagrangian cobordism induced by the algebraic weave.
    \end{corollary}
    \begin{proof}
        This follows from \cref{prop:monodromyless_MCS_sequences_MC2F} and the equivalence of categories $\mathfrak{B}_n \cong \mathfrak{W}_n$.
    \end{proof}
    
    In terms of concrete presentations, Rutherford--Sullivan \cite{rutherford2018generating} showed that an MC2F is equivalent to an augmentation of the Chekanov--Eliashberg dga (over $\Z/2\Z$). For weaves, microlocal sheaves are related to moduli spaces of flags \cite{CasalsZaslow20} and moduli space of colorings in the case of 2-weaves \cite{sackel2024differential}. As such, the moduli space of trivial monodromy algebraic weaves (and therefore Morse complex 2-families) sits halfway between augmentations of the Chekanov--Eliashberg dga and simple microlocal sheaves on $\D^2 \times \R$, further elucidating the explicit maps defined by Rutherford--Sullivan \cite[Theorem 1.1]{rutherford2021sheaves}. 

    \begin{theorem}\label{thm:sheaves_mc2f}
        Let $\mathfrak{m} \colon \beta \to\beta'$ be a morphism in $\mathfrak{B}_n$. There is a bijection from the equivalence classes of simple sheaves on $\D^2 \times \R$ with singular support on $\Lambda(\mathfrak{m})$ to the equivalence classes of MC2Fs on the front projection of $\Lambda(\mathfrak{m})$.
    \end{theorem}
    \begin{proof}
        Following \cref{prop:sheaf-on-weave,cor:algebra-weave-to-MC2F}, for any simple sheaf with singular support on $\Lambda(\mathfrak{m})$ there is an associated Morse complex 2-family on $\Lambda(\mathfrak{m})$ defined by the A-form MCS sequence on $\mathfrak{m}$. 
        
        Given any Morse complex 2-family on $\Lambda(\mathfrak{m})$, we have an MCS sequence on $\mathfrak{m}$ where the MCSs for the braids in the braid sequence $(\beta \eqqcolon \beta_1, \dots, \beta_q \coloneqq \beta')$ are given by restricting the Morse complex 2-family to each braid $\beta_1, \dots, \beta_q$. The MCSs for the braids in the braid sequence $(\beta \eqqcolon \beta_1, \dots, \beta_q \coloneqq \beta')$ are related by the MCS braid moves \cref{lma:monodromy_comp_prod,lma:mcs_braid_matrix,lma:monodromy_trivalent_etc_id}. By \cref{lma:monodromy_comp_prod,lma:mcs_braid_matrix}, MCS equivalences on braids $\beta_j$ and $\beta_{j+1}$ commute with the MCS braid moves between the braids $\beta_j \to \beta_{j+1}$, so we can apply MCS equivalences on each braid without changing the trivial monodromy condition for the MCS braid moves $\beta_j \to \beta_{j+1}$. Thus, the MCS sequence on the braid sequence $(\beta_1, \dots, \beta_q)$ is equivalent to an A-form MCS sequence on $\mathfrak{m}$ and, by \cref{prop:sheaf-on-weave,cor:algebra-weave-to-MC2F}, is equivalent to a simple sheaf with singular support on $\Lambda(\mathfrak{m})$.
    \end{proof}

    Finally, following similar arguments as in \cref{sec:normal_rulings_cat}, we also compare the weave decomposition and ruling decomposition on the moduli space of sheaves.
    \begin{theorem}\label{thm:ruling_sheaf=weave_sheaf}
    Let $D$ be a front diagram of the $(-1)$-closure of $\beta\Delta$ for some $\beta \in \Br_n^+$ with $\delta(\beta) = w_0$. Then, for any graded normal ruling $\rho$, there is a commutative diagram
    \[\begin{tikzcd}[sep=scriptsize]
        \mathcal M_1^\textit{fr}(D, T)^\rho \ar[r,"\phi_\rho'"] \ar[d,"\cong" left] & \mathcal M_1^\textit{fr}(D, T) \ar[d,equal] \\
        \mathcal M_1^\textit{fr}(L(\w_\rho), T) \ar[r] & \mathcal M_1^\textit{fr}(D, T).
    \end{tikzcd}\]
    \end{theorem}
    
    \begin{proof}
    We follow a similar argument as in \cref{thm:ruling_sheaf_decomp}; see also \cite[Proposition 6.30]{STZ_ConstrSheaves}. Since we have fixed the framing for all the moduli spaces of sheaves by one base point per strand of the braid, we may assume that the flag on the left of all the crossings is the standard flag, i.e., $\{0\} \subset \C \subset \C^2 \subset \dots \subset \C^n$. Consider the weave $\w = \w_\rho$ associated to the morphism in the ruling category in \cref{dfn:ruling_category}. For any simple sheaf $\mathcal F_\w \in \mathcal M_1^\textit{fr}(L(\w), T)$, we can also define a filtration $R_\bullet$ such that
    \[
    R_i\mathcal F_\w(U) = \mathcal F_\w(U) \cap \C^i, \quad 0\leq i \leq n.
    \]
    We show that the restriction of $\mathcal F_\w \in \mathcal M_1^\textit{fr}(L(\w), T)$ to the boundary of the disk defines a sheaf in $\mathcal M_1^\textit{fr}(D, T)^\rho$. By \cref{lma:rulings_right_induct_weaves}, we know the weave $\w$ determines a unique normal ruling $\rho_j$ on each slice of the weave $\beta_j$. For the bottom slice $\Delta$, we know $\mathcal F_\w$ restricts to the unique sheaf with singular support on $\Delta$, and the associated ruling is the unique ruling with no switches. Suppose the restriction of $\mathcal F_\w$ to the slice $\beta_j$ determines the normal ruling $\rho_j$ associated to the (right inductive) weave, we show that its restriction to the slice $\beta_{j+1}$ also determines the ruling $\rho_{j+1}$ associated to the (right inductive) weave.

    The main observation we use is that, by \cref{thm:ruling_sheaf_decomp}, the restriction of the filtration $R_\bullet \mathcal F_\w$ to any slice $\beta_j$ must define some normal ruling on the slice $\beta_j$. When $\beta_j \to \beta_{j+1}$ is related by a distant crossing or hexavalent vertex, we know that no switches are involved. The filtration $R_\bullet\mathcal F_\w$ extends from $\beta_j$ to $\beta_{j+1}$ canonically. When $\beta_j \to \beta_{j+1}$ is related by a cup move, the flags on $\beta_j$ are $\dots, F_k, \dots$ and the flags on $\beta_{j+1}$ are $\dots, F_k, F_{k+1}, F_{k+2}, \dots$, where $F_k = F_{k+2}$. Thus, the sheaves $R_\bullet\mathcal F_\w$ must contain the same set of subspaces in $F_k$ and $F_{k+2}$. Since we know the left crossing has to be a return for any normal ruling of $\beta_{j+1}$, the filtration extends from $\beta_j$ to $\beta_{j+1}$ without introducing any switches. When $\beta_j \to \beta_{j+1}$ is related by a trivalent move, the filtration has to extend from $\beta_j$ to $\beta_{j+1}$ by creating one single switch at the right crossing, because the left crossing has to be a return for any normal ruling of $\beta_{j+1}$. This shows that the restriction functor of sheaves to the boundary of the disk factors as
    \[
    \mathcal M_1^\textit{fr}(L(\w_\rho), T) \hooklongrightarrow \mathcal M_1^\textit{fr}(D, T)^\rho \hooklongrightarrow \mathcal M_1^\textit{fr}(D, T).
    \]
    Since $\mathcal M_1^\textit{fr}(D, T)$ can be decomposed into a disjoint union of $\mathcal M_1^\textit{fr}(D, T)^\rho$ and, respectively, a disjoint union of $\mathcal M_1^\textit{fr}(L(\w_\rho), T)$, it follows that the first map has to be bijective.
    \end{proof}

    As a consequence of \cref{thm:corr_decomps,thm:ruling_sheaf=weave_sheaf}, the ruling decomposition on the moduli space of sheaves and the ruling decomposition on the augmentation variety agree under the isomorphism $\Aug(\La(\beta\Delta)) \cong X(\beta) \cong \mathcal M_1^\textit{fr}(\La(\beta\Delta),T)$ from \cref{thm:-1=braid vty_new,prop:sheaves_braid_var}.
    
    \begin{theorem}\label{thm:ruling_sheaf=ruling_aug}
    Let $D$ be a front diagram of the $(-1)$-closure of $\beta\Delta$ for some $\beta \in \Br_n^+$ with $\delta(\beta) = w_0$. Then, for any graded normal ruling $\rho$, there is a commutative diagram
    \[\begin{tikzcd}[sep=scriptsize]
        \mathcal M_1^\textit{fr}(D, T)^\rho \ar[r,"\phi_\rho'"] \ar[d,"\cong" left] & \mathcal M_1^\textit{fr}(D, T) \ar[d,"\cong"] \\
        \widehat{\MCS}{}^\rho(D) \ar[r,hook,"\phi_\rho"] & \Aug(D).
    \end{tikzcd}\]
    \qed
    \end{theorem}

\section{Ruling and weave decompositions via cycle deletion}\label{sec:weave_ruling_decomp_cycle_deletion}
    One would hope to construct right simplifying weaves by applying cycle deletions to the right inductive Demazure weave. In this section, we explain why this strategy does not work in general. Recall the discussion on Lusztig cycles from \cref{dfn:lusztig_cycle}. 
\subsection{Cycle deletion in weaves}\label{sec:cycle_del_weaves}
    We enlarge the category of cycles in weaves following \cite[Section 2.4]{CasalsZaslow20}. They correspond to 1-cycles in the Legendrian surfaces associated to the weaves.

\begin{definition}[Geometric $1$-cycles]
    Let $\beta_1,\beta_2 \in \Br_n^+$ for any $n\geq 1$ and let $\w \colon \beta_1 \to \beta_2$ be a weave. A \emph{geometric $1$-cycle} of $\w$ is a graph $\Gamma$ with only vertices of valency $\leq 3$ together with an inclusion of graphs $\iota \colon \Gamma \to \w$ such that the following holds:
    \begin{enumerate}
        \item Every univalent vertex of $\Gamma$ is mapped by $\iota$ to a trivalent vertex of $\w$; see \cref{fig:geom_cycle1}.
        \item Every bivalent vertex $v$ of $\Gamma$ satisfies one of the following:
        \begin{enumerate}
            \item $v$ is mapped by $\iota$ to a hexavalent vertex of $\w$ such that the two adjacent edges to $v$ are mapped via $\iota$ to an edge of $G_i$ and $G_{i+1}$, respectively, for some $i$; see \cref{fig:geom_cycle2}.
            \item $v$ is mapped by $\iota$ to a tetravalent vertex of $\w$ such that the two adjacent edges to $v$ are mapped via $\iota$ to an edge of $G_i$ for some $i$; see \cref{fig:geom_cycle4}.
        \end{enumerate}
        \item Every trivalent vertex $v$ of $\Gamma$ is mapped to a hexavalent vertex of $\w$ via $\iota$, such that every edge adjacent to $v$ is mapped via $\iota$ to an edge of $G_i$ for some $i$; see \cref{fig:geom_cycle3}.
    \end{enumerate}
\end{definition}
\begin{figure}[!htb]
    \centering
    \begin{subfigure}{0.49\textwidth}
        \centering
        \includegraphics{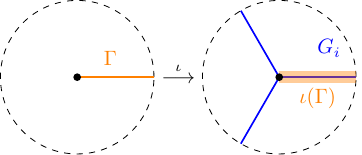}
        \caption{}\label{fig:geom_cycle1}
    \end{subfigure}
    \begin{subfigure}{0.49\textwidth}
        \centering
        \includegraphics{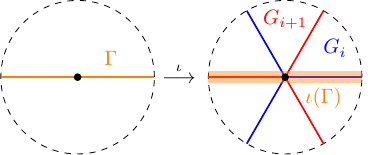}
        \caption{}\label{fig:geom_cycle2}
    \end{subfigure}
    \begin{subfigure}{0.49\textwidth}
        \centering
        \includegraphics{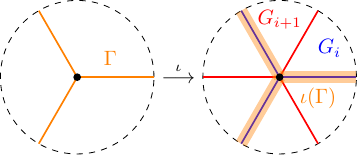}
        \caption{}\label{fig:geom_cycle3}
    \end{subfigure}
     \begin{subfigure}{0.49\textwidth}
        \centering
        \includegraphics{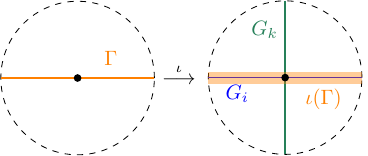}
        \caption{}\label{fig:geom_cycle4}
    \end{subfigure}
    \caption{Local models for geometric $1$-cycles in weaves.}
\end{figure}
\begin{definition}[\textsf I- and \textsf Y-cycles {\cite[Section 2.4]{CasalsZaslow20}}]
    Let $\w$ be a weave.
    \begin{enumerate}
        \item An \emph{\textsf I-cycle} is a geometric $1$-cycle $\iota \colon \Gamma \to \w$ where $\Gamma$ only consists of two univalent vertices and some number of bivalent vertices.
        \item A \emph{short \textsf I-cycle} is an \textsf I-cycle with no bivalent vertices.
        \item A \emph{long \textsf I-cycle} is an \textsf I-cycle with at least one bivalent vertex.
        \item A \emph{\textsf Y-cycle} is a geometric $1$-cycle $\iota \colon \Gamma \to \w$ where $\Gamma$ consists of at least one trivalent vertex.
        \item A \emph{\textsf Y-tree} is a \textsf Y-cycle $\iota \colon \Gamma \to \w$ such that $\Gamma$ is a tree and $\iota$ is injective.
    \end{enumerate}
\end{definition}
\begin{lemma}[{\cite[Section 2.4]{CasalsZaslow20}}]
    Let $\w$ be a weave. Then any \textsf I-cycle and any \textsf Y-cycle in $\w$ defines a cycle in $H_1(L(\w);\Z)$.
    \qed
\end{lemma}

The following definition of cycle deletions is also found in \cite[Section 4.9.1]{CasalsWeng}. Geometrically, cycle deletion corresponds to locally modifying the Legendrian surface associated to the weave by replacing a cylinder around a 1-cycle (which bounds a Lagrangian disk) by two disks \cite[Lemma 4.27]{CasalsWeng}.
\begin{definition}[Deletion of a geometric $1$-cycle]
    Let $\w$ be a weave, and let $\iota \colon \Gamma \to \w$ be a geometric $1$-cycle that is injective on the set of vertices. \emph{Deletion of $\Gamma$} is an operation that modifies the weave $\w$ locally near every trivalent, tetravalent and hexavalent vertex, as depicted in \cref{fig:local_models_geom_cycle_del_weave}.
\end{definition}
\begin{figure}[!htb]
    \centering
    \begin{subfigure}{0.49\textwidth}
        \centering
        \includegraphics{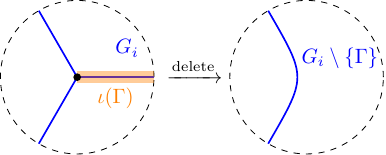}
        \caption{}\label{fig:geom_cycle_del1}
    \end{subfigure}
     \begin{subfigure}{0.49\textwidth}
        \centering
        \includegraphics{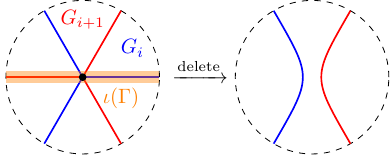}
        \caption{}\label{fig:geom_cycle_del2}
    \end{subfigure}
     \begin{subfigure}{0.49\textwidth}
        \centering
        \includegraphics{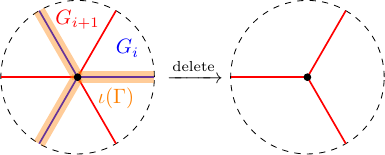}
        \caption{}\label{fig:geom_cycle_del3}
    \end{subfigure}
    \begin{subfigure}{0.49\textwidth}
        \centering
        \includegraphics{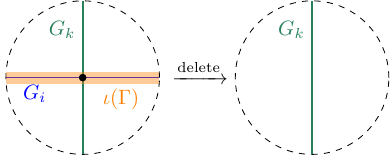}
        \caption{}\label{fig:geom_cycle_del4}
    \end{subfigure}
    \caption{Defining local models for geometric $1$-cycle deletion of a weave.}\label{fig:local_models_geom_cycle_del_weave}
\end{figure}
\subsection{Cycle deletion in the braid category}\label{sec:lusztig_cycles_braid_cat}
    We recall the definition of Lusztig cycles in the braid category in \cref{dfn:lusztig_cycle_mcs} and define \textsf Y-trees in this setting.
    \begin{definition}[\textsf Y-tree of an MCS sequence]\label{dfn:geometric_cycle_mcs}
        Let $\mathfrak{m} \colon \beta \to \beta'$ be a morphism in $\mathfrak{B}_n$ represented by the sequence of braids $(\beta_1,\ldots,\beta_q)$. A \emph{\textsf Y-tree} of $\mathfrak{m}$ is a sequence $(x^1,\ldots,x^q)$ where 
        \[
        x^i = (x^i_1,\ldots,x^i_{r_i})
        \]
        is a subword of $\beta_i$ for every $i\in \{1,\ldots,q\}$, such that
        \begin{enumerate}
            \item $x^{i+1}$ and $x^i$ are related via one of the local modifications depicted in \cref{fig:local_models_geom_cycle_del_mcs}, depending on how $\beta_{i+1}$ and $\beta_i$ are related in $\mathfrak{m}$, and
            \item $x^1$ and $x^q$ are both the empty subword.
        \end{enumerate}
    \end{definition}
    \begin{figure}[!htb]
    \centering
    \begin{subfigure}{0.49\textwidth}
        \centering
        \includegraphics[scale=0.9]{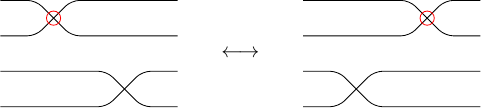}
        \caption{}\label{fig:mcs_cycle_dist1}
    \end{subfigure}
    \begin{subfigure}{0.49\textwidth}
        \centering
        \includegraphics[scale=0.9]{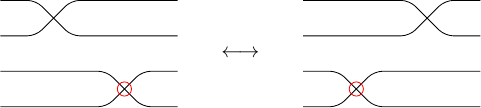}
        \caption{}\label{fig:mcs_cycle_dist2}
    \end{subfigure}
    \begin{subfigure}{0.49\textwidth}
        \centering
        \includegraphics[scale=0.9]{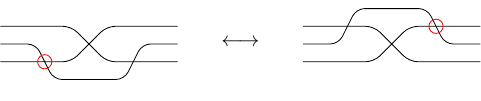}
        \caption{}\label{fig:mcs_cycle_hexa1}
    \end{subfigure}
    \begin{subfigure}{0.49\textwidth}
        \centering
        \includegraphics[scale=0.9]{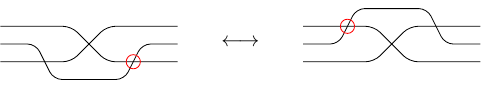}
        \caption{}\label{fig:mcs_cycle_hexa2}
    \end{subfigure}
    \begin{subfigure}{0.49\textwidth}
        \centering
        \includegraphics[scale=0.9]{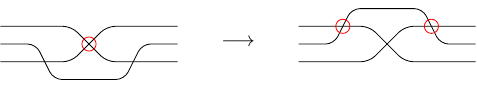}
        \caption{}\label{fig:mcs_cycle_hexa3}
    \end{subfigure}
    \begin{subfigure}{0.49\textwidth}
        \centering
        \includegraphics[scale=0.9]{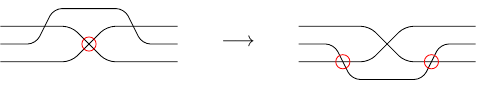}
        \caption{}\label{fig:mcs_cycle_hexa4}
    \end{subfigure}
    \begin{subfigure}{0.49\textwidth}
        \centering
        \includegraphics[scale=0.9]{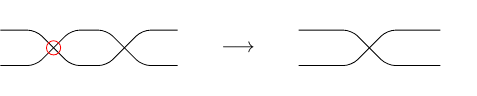}
        \caption{}\label{fig:mcs_cycle_tri1}
    \end{subfigure}
    \begin{subfigure}{0.49\textwidth}
        \centering
        \includegraphics[scale=0.9]{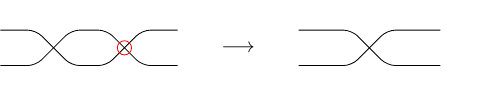}
        \caption{}\label{fig:mcs_cycle_tri2}
    \end{subfigure}
    \begin{subfigure}{0.49\textwidth}
        \centering
        \includegraphics[scale=0.9]{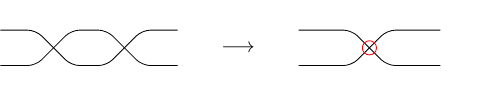}
        \caption{}\label{fig:mcs_cycle_tri3}
    \end{subfigure}
    \caption{Defining local models for a \textsf Y-tree in an MCS sequence. A red circle indicates an intersection point that is present in one of the subwords $x^i$ (and the intersection points without a red circle are not part of the same subword, but are possibly part of other ones).}\label{fig:local_models_geom_cycle_del_mcs}
\end{figure}

    \begin{definition}[Cycle deletion]\label{defn:cycle_deletion}
        Let $\mathfrak{m}\colon \beta \to \beta'$ be a morphism in $\mathfrak{B}_n$ with underlying braid sequence $(\beta_1,\ldots,\beta_q)$ and let $x \coloneqq (x^1,\ldots,x^q)$ be a \textsf Y-tree of $\mathfrak{m}$. The operation of \emph{deleting the \textsf Y-tree $x$} from $\mathfrak{m}$ yields another morphism $\mathfrak{m}' \colon \beta \to \beta'$ with underlying braid sequence $(\beta'_1,\ldots,\beta'_q)$ such that $\beta_i'$ is obtained from $\beta_i$ by smoothing every crossing of $\beta_i$ that belongs to the subword $x^i$ for every $i\in \{1,\ldots,q\}$, as depicted in \cref{fig:mcs_cycle_del}.
    \end{definition}
    \begin{figure}[!htb]
        \centering
        \includegraphics{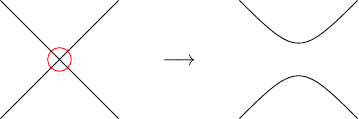}
        \caption{Smoothing of crossings defines the operation of cycle deletion for a sequence of MCSs.}
        \label{fig:mcs_cycle_del}
    \end{figure}
    It is clear from the definitions above that the transitive closure of cycle deletion of morphisms in $\mathfrak{B}_n$ defines a partial order. Therefore, we upgrade $\mathfrak{B}_n$ to a $2$-category with horizontal composition defined in the obvious way. Namely, recall that the composition of $1$-morphisms in $\mathfrak{B}_n$ is defined as concatenation of their braid sequences and corresponding braid moves. 
    
    Let $x \colon \mathfrak{m} \Rightarrow \mathfrak{m}_\circ$ and $x' \colon \mathfrak{m}' \Rightarrow \mathfrak{m}'_\circ$ be two $2$-morphisms. By \cref{dfn:geometric_cycle_mcs}, each \textsf Y-tree involved in the $2$-morphisms $x$ and $x'$ correspond to non-relative cycles in $H_1(L(\mathfrak{A}(\mathfrak{m}))$ (and therefore stay in a compact subset). This guarantees that we get a well-defined $2$-morphism $\mathfrak{m}'\circ \mathfrak{m} \Rightarrow \mathfrak{m}'_\circ \circ \mathfrak{m}_\circ$; it is defined by deleting the \textsf Y-trees corresponding to $x$ and $x'$ in any order (while keeping the respective order of the cycle deletions within $x$ and $x'$ fixed).
    
    \begin{proposition}\label{prop:2-functor_mcs_weave}
        Let $n\in \Z_{\geq 1}$. There is a $2$-functor $\mathfrak{A} \colon \mathfrak{B}_n \to \mathfrak{W}_n$ whose underlying $1$-functor is the one in \cref{thm:ruling_to_weaves}.
    \end{proposition}
    \begin{proof}
        It is clear from the definition that a \textsf Y-tree of $\mathfrak{m} \colon \beta \to \beta'$ yields a \textsf Y-tree of the corresponding algebraic weave $\mathfrak{A}(\mathfrak{m})$, and by comparing the local models depicted in \cref{fig:local_models_geom_cycle_del_mcs,fig:local_models_geom_cycle_del_weave}, we see that cycle deletion in $\mathfrak{m}$ corresponds precisely to deletion of the corresponding \textsf Y-tree in $\mathfrak{A}(\mathfrak{m})$.
    \end{proof}

\subsection{Decompositions via cycle deletion}\label{sec:decompositions_cycle_del}
    For a positive braid $\beta \in \Br_n^+$, recall the definition of the weave decomposition of the braid variety $X(\beta)$ from \cref{thm:weave_decomp}. Sometimes a decomposable tuple of weaves $(\w_1,\ldots,\w_k)$ is completely determined by the maximal piece $\w_1$ together with a sequence of Lusztig cycle deletions.
    
    \begin{definition}[Cycle deletable]
        An algebraic weave $\w \colon \beta \to \beta'$ is called \emph{cycle deletable} if there exists a \textsf{Y}-tree $\gamma$ in $\w$ such that $\w\setminus\gamma \colon \beta \to \beta'$ is an algebraic weave.
    \end{definition}
    
    \begin{definition}[Cycle decomposable]\label{dfn:cycle_decomposable}
       Let $n\in \Z_{\geq 1}$. A braid $\beta\in \Br_n^+$ is called \emph{cycle decomposable} if there exists a cycle deletable right inductive weave $\w \colon \beta \to \Delta$ and a decomposing tuple $(\w=\w_1,\ldots,\w_k)$ for $X(\beta)$ such that $\w_i = \w_{i-1} \setminus \gamma_{i-1}$ for $i\in \{2,\ldots,k\}$.
    \end{definition}
    
    The upshot of a braid $\beta$ being cycle decomposable is that there is a distinguished decomposition of the braid variety $X(\beta)$ generated by a (cycle deletable) weave $\w \colon \beta \to \Delta$ with underlying braid sequence $(\beta_1, \ldots, \beta_q)$ via repeated cycle deletions, which, via \cref{prop:2-functor_mcs_weave}, corresponds to certain crossing smoothings of the braids $\beta_i$.
    
    \begin{proposition}\label{prop:torus_n2_decomp}
        For any $k \geq 1$, the braid $\sigma_1^{k+1} \in \Br_2^+$ is cycle decomposable.
    \end{proposition}
    
    \begin{proof}
        The maximal ruling of $\La(\sigma_1^{k+1}\Delta) = \La(\sigma_1^{k+2})$ is such that the $k$ crossings of $\La(\sigma^{k+2}_1)$ corresponding to the middle $k$ letters of $\sigma_1^{k+2}$ are switches. A right inductive morphism $\mathfrak{r}_\rho \colon \sigma_1^{k+1} \to \Delta$ whose underlying normal ruling is the maximal one is given by the sequence of braids
        \[
        (\beta_1,\ldots,\beta_k) \coloneqq (\sigma_1^{k+1},\sigma_1^k,\ldots,\sigma_1),
        \]
        and $k-1$ trivalent moves given by
        \[
        \beta_j = (\sigma_1)^2\sigma_1^{k-j} \longrightarrow \sigma_1\sigma_1^{k-j} = \beta_{j-1}, \qquad j\in \{2,\dots,k\}.
        \]
        There are precisely $k-2$ \textsf Y-trees of the morphism $\mathfrak{m}$ denoted by $x_1,\ldots,x_{k-2}$ where $x_j = ((x_j)^1,\ldots,(x_j)^k)$ and
        \[
        (x_j)^\ell = \begin{cases}
            (), & \ell \neq j, \\
            (1), & \ell = j,
        \end{cases}
        \]
        where $(1)$ denotes the subword of $\beta_j = \sigma_1^{k+2-j}$ being the leftmost letter. We see that all possible cycle deletions in the morphism $\mathfrak{m}$ yields all possible right inductive algebraic simplifying weaves $\sigma_1^{k+1} \to \Delta$ which yields a decomposing tuple of simplifying weaves.
    \end{proof}
    
    \begin{question}\label{ques:ytree2}
        For any braid $\beta\in \Br_n^+$, does there exist an inductive weave $\w(\beta)$ such that all its mutable Lusztig cycles are cycle deletable \textsf Y-trees?
    \end{question}
    \begin{remark}
        See~\cite[Definition 3.7 and 3.8]{CasalsWeng} for the definition of a mutable $I$ and $Y$-cyle. Casals and Weng~\cite[Section 3.3]{CasalsWeng} answer this question affirmatively for any braid $\beta = \Delta\gamma$ where $\gamma \in \Br_n^+$. 
    \end{remark}

    An affirmative answer to \cref{ques:ytree2} would imply that any braid $\beta$ has a cycle deletable inductive weave. However, note that after deleting a \textsf Y-tree from the cycle deletable inductive weave it is not guaranteed that we obtain a cycle deletable algebraic weave, which means that an affirmative answer to \cref{ques:ytree2} would not guarantee that $\beta$ is cycle decomposable.

    \begin{question}\label{qst:sufficient_cond_decomp}
        Is there a sufficient condition on $\beta \in \Br_n^+$ to guarantee that $\beta$ is cycle decomposable?
    \end{question}
    
    \begin{remark}
        We expect one can show that for any braid $\beta = \Delta\gamma$ where $\gamma \in \Br_n^+$, there exists an inductive weave such that its mutable Lusztig cycles are cycle deletable, and the weaves after arbitrary cycle deletions are still cycle deletable using \cite[Section 3.3]{CasalsWeng}. However, it is not clear how to show that they always give a decomposition tuple.
    \end{remark}

    If $\beta$ is cycle decomposable, repeated cycle deletion yields a weave decomposition of $X(\beta)$. However, it may not agree with the ruling decomposition on $\Aug(\La(\beta\Delta))$ under the correspondence in \cref{thm:hr_and_weave_decomps}, since cycle deletion may not preserve the property of a simplifying weave being right inductive. We demonstrate this in the following example.

\begin{example}\label{ex:deletion_neq_ruling}
    Let $\beta= \Delta(\sigma_2\sigma_1\sigma_2)^2 \in \Br_3^+$. The Legendrian link $\La(\beta\Delta)$ is smoothly the $(3, 3)$-torus link and its maximal normal ruling produces, via \cref{prop:normal_ruling_inductive}, the right inductive weave $\w$ depicted in \cref{fig:ex_cycle_del_mut_ruling}. 
    \begin{figure}[!htb]
        \centering
        \includegraphics{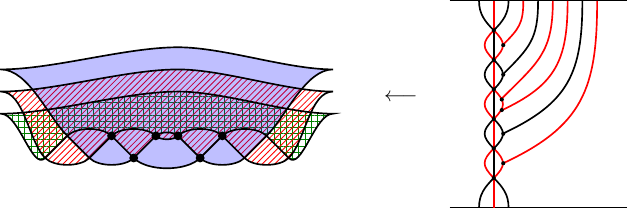}
        \caption{A right inductive weave $\w$ produced from the maximal normal ruling of $\La(\beta\Delta)$ via \cref{prop:normal_ruling_inductive}.}
        \label{fig:ex_cycle_del_mut_ruling}
    \end{figure}

    \begin{figure}[!htb]
        \centering
        \begin{subfigure}{.32\textwidth}
            \centering
            \includegraphics{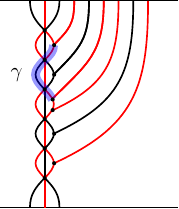}
            \caption{The weave $\w$.}\label{fig:ex_cycle_del_mut1}
        \end{subfigure}
        \begin{subfigure}{.32\textwidth}
            \centering
            \includegraphics{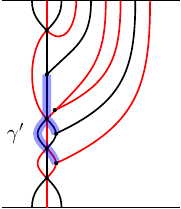}
            \caption{The weave $\w'$.}\label{fig:ex_cycle_del_mut2}
        \end{subfigure}
        \begin{subfigure}{.32\textwidth}
            \centering
            \includegraphics{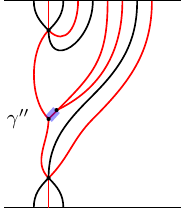}
            \caption{The weave $\w''$.}\label{fig:ex_cycle_del_mut3}
        \end{subfigure}
        \caption{Repeated cycle deletion in a right inductive weave for $\beta = \Delta(\sigma_2\sigma_1\sigma_2)^2$.}
        \label{fig:ex_cycle_del_mut}
    \end{figure}
    Deleting the Lusztig cycle $\gamma$ from $\w$ (see \cref{fig:ex_cycle_del_mut1}) produces a Legendrian weave $\w'$ with a new Lusztig cycle labeled by $\gamma'$; see \cref{fig:ex_cycle_del_mut2}. Deleting $\gamma'$ from $\w'$ produces a third weave $\w''$ with cycle $\gamma''$; see \cref{fig:ex_cycle_del_mut3}. We note that the weave $\w''$ is a simplifying weave but \emph{not} a right inductive one. Therefore $\w''$ does not correspond to a stratum in the ruling decomposition of $\Aug(\La(\beta\Delta))$. Therefore the weave decomposition of $X(\beta)$ obtained from cycle deletion cannot agree with the ruling decomposition of $\Aug(\La(\beta\Delta))$ under the correspondence in \cref{thm:hr_and_weave_decomps} (or the Deodhar decomposition of $R_{w_0,\beta}^\circ$ in \cref{thm:deodhar=hr_decomp}). Furthermore, observe that mutating $\w''$ at $\gamma''$ \emph{does} yield a right simplifying weave, and so it corresponds to a stratum in the ruling decomposition of $\Aug(\La(\beta\Delta))$, see \cref{fig:ex_cycle_del_mut_result}.

    \begin{figure}[!htb]
        \centering
        \includegraphics{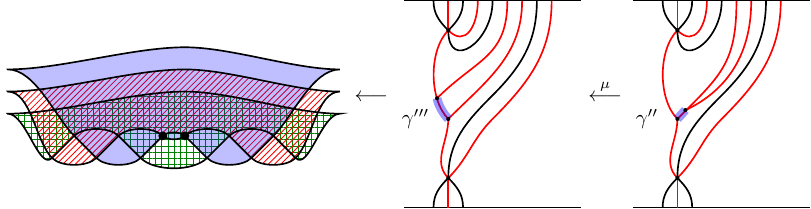}
        \caption{Mutation of $\w''$ at the Lusztig cycle $\gamma''$ produces a right simplifying weave corresponding to a normal ruling of $\La(\beta\Delta)$ with two switches.}
        \label{fig:ex_cycle_del_mut_result}
    \end{figure}
\end{example}

\begin{example}\label{ex:deletion_neq_ruling_codim1}
    Let $\beta = \sigma_2\sigma_1^2\sigma_2^2\sigma_1\sigma_2 \in \Br_3^+$ and let $\w$ be a right inductive weave corresponding to the maximal normal ruling via \cref{prop:normal_ruling_inductive}. Deleting the Lusztig cycle $\gamma$, depicted in \cref{fig:ex_cycle_del_mut_codim1}, from $\w$ immediately produces a weave $\w'$ that does not correspond to a stratum in the ruling decomposition of $\Aug(\Lambda(\beta\Delta))$. Therefore, even the codimension 1 piece in the weave decomposition of $X(\beta)$ does not agree with the ruling decomposition of $\Aug(\Lambda(\beta\Delta))$ under \cref{thm:hr_and_weave_decomps} (or equivalently the Deodhar decomposition of $R_{w_0,\beta}^\circ$ in \cref{thm:deodhar=hr_decomp}).

    \begin{figure}[!htb]
        \centering
        \includegraphics{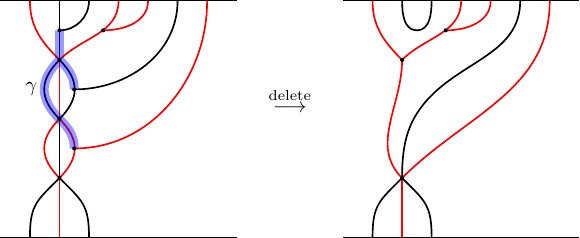}
        \caption{Cycle deletion for a right inductive weave $\w$ for $\beta = \sigma_2\sigma_1^2\sigma_2^2\sigma_1\sigma_2$.}\label{fig:ex_cycle_del_mut_codim1}
    \end{figure}
\end{example}

\begin{remark}
    Note that \cref{ex:deletion_neq_ruling_codim1} is in contrast with the case of double Bott--Samelson cells where $\beta = \Delta\gamma$ for some $\gamma \in \Br_n^+$ \cite{ShenWeng,GaoShenWeng}. For double Bott--Samelson cells (and in fact for braids that arise from grid plabic graphs \cite[Definition 2.1]{CasalsWeng}), Casals--Weng showed that the zero loci of the cluster variables in the initial seed are given by cycle deletions of a right inductive weave \cite[Proposition 4.40 and 4.50]{CasalsWeng}. For braid varieties, Galashin--Lam--Sherman-Bennet--Speyer showed that the zero loci of the cluster variables in the initial seed are given by the Deodhar hypersurfaces corresponding to the (almost positive) distinguished sequences \cite[Proposition 7.10]{GLSS}. Since the cluster structures in \cite{CasalsWeng}, \cite{CGGLSS}, and \cite{GLSS} agree (see \cite[Theorem 1.1]{CGGSBS} for the comparison of the second and third ones in general, and \cite[Theorem 1.2]{CasalsLi} for the comparison between the first and second ones in the Bott--Samelson case), it follows that the codimension 1 pieces in the weave decomposition obtained by cycle deletions of a right inductive weave correspond to the codimension 1 pieces in the Deodhar decomposition or the ruling decomposition.
\end{remark}

    The ruling decomposition and corresponding weave decompositions are not preserved under braid equivalences.

\begin{example}
    Let $\beta'=\Delta(\sigma_1\sigma_2^2)^2 \in \Br_3^+$ and note that it is braid equivalent to $\beta$ in \cref{ex:deletion_neq_ruling} (they differ by two moves $\sigma_1\sigma_2\sigma_1 \leftrightarrow \sigma_2\sigma_1\sigma_2$). Note that this means that
    \[
    \Aug(\La(\beta\Delta)) \cong X(\beta) \cong X(\beta') \cong \Aug(\La(\beta'\Delta)).
    \]
    The braid $\beta'$ admits a right inductive weave depicted in \cref{fig:ex_cycle_del_mut_ruling2}. This right inductive weave is in fact cycle decomposable and the decomposing tuple of $X(\beta')$ obtained via cycle deletion consists of right simplifying weaves. Hence the resulting weave decomposition of $X(\beta')$ does correspond to the ruling decomposition of $\Aug(\La(\beta'\Delta)$ via \cref{thm:hr_and_weave_decomps} (or, equivalently, the Deodhar decomposition of $R_{w_0,\beta'}^\circ$ in \cref{thm:deodhar=hr_decomp}).

    \begin{figure}[!htb]
        \centering
        \includegraphics{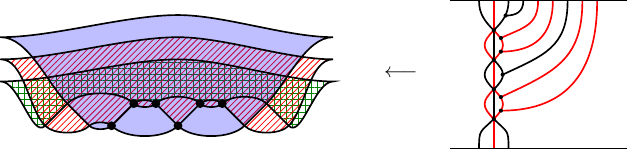}
        \caption{A right inductive weave produced from the maximal normal ruling of $\La(\beta'\Delta)$ via \cref{prop:normal_ruling_inductive}.} 
        \label{fig:ex_cycle_del_mut_ruling2}
    \end{figure}
\end{example}

\bibliographystyle{alpha}
\bibliography{biblio}
\end{document}